\documentclass[a4paper,10pt,onecolumn,twoside]{article}
\usepackage{mathrsfs}
\usepackage{mathrsfs}
\usepackage{amsfonts}
\usepackage{fancyhdr}
\usepackage{amsmath,amsfonts,amssymb,graphicx,amsthm}
\usepackage[pdftex,bookmarksnumbered=true,bookmarksopen=true,          
            colorlinks=true,pdfborder=001,citecolor=blue,
            linkcolor=red,anchorcolor=green,urlcolor=blue]{hyperref}
\allowdisplaybreaks

\newtheorem{lemma}{Lemma}[section]
\newtheorem{theorem}{Theorem}[section]

\newtheorem{remark}{Remark}[section]
\newtheorem{corollary}{Corollary}[section]
\numberwithin{equation}{section}

\begin{document}
\title{\bf Negative Sobolev Spaces and the Two-species Vlasov-Maxwell-Landau System in the Whole Space}
\author{{\bf Yuanjie Lei}\quad and \quad {\bf Huijiang Zhao}\thanks{Corresponding author. E-mail: hhjjzhao@hotmail.com}\\[2mm]
       School of Mathematics and Statistics,
       Wuhan University, Wuhan 430072, China}
\date{}
\vskip 0.2cm
\maketitle
\begin{abstract}
A global solvability result of the Cauchy problem of the two-species Vlasov-Maxwell-Landau system near a given global Maxwellian is established by employing an approach different than that of \cite{Duan-VML}. Compared with that of \cite{Duan-VML}, the minimal regularity index and the smallness assumptions we imposed on the initial data are weaker. Our analysis does not rely on the decay of the corresponding linearized system and the Duhamel principle and thus it can be used to treat the one-species Vlasov-Maxwell-Landau system for the case of $\gamma>-3$ and the one-species Vlasov-Maxwell-Boltzmann system for the case of $-1<\gamma\leq 1$ to deduce the global existence results together with the corresponding temporal decay estimates.
\end{abstract}
\section{Introduction and main results}
\setcounter{equation}{0}

The motion of a fully ionized plasma consisting of only two species particles (e.g. electrons and ions) under
the influence of the self-consistent Lorentz force and binary collisions is governed by the following two-species Vlasov-Maxwell-Landau (called VML in the sequel for simplicity) system
\begin{equation} \label{VML}
\begin{cases}
 \partial_tF_++ v  \cdot\nabla_xF_++(E+v\times B)\cdot\nabla_{ v  }F_+=Q(F_+,F_+)+Q(F_+,F_-),\\
 \partial_tF_-+ v  \cdot\nabla_xF_--(E+v\times B)\cdot\nabla_{ v  }F_-=Q(F_-,F_+)+Q(F_-,F_-),\\
 \partial_tE-\nabla_x\times B=-{\displaystyle\int_{\mathbb{R}^3}}v(F_+-F_-)dv,\\
 \partial_tB+\nabla_x\times E=0,\\
 \nabla_x\cdot E={\displaystyle\int_{\mathbb{R}^3}}(F_+-F_-)dv,\quad \nabla_x\cdot B=0.
  \end{cases}
  \end{equation}
Here the unknown functions $F_\pm= F_\pm(t,x, v) \geq  0$
are the number density functions for the ions ($+$) and electrons ($-$) respectively, at time $t\geq 0$, position $x = (x_1, x_2, x_3)\in {\mathbb{R}}^3$, velocity $ v=( v_1,  v_2,  v_3) \in {\mathbb{R}}^3$. $E(t,x)$ and $B(t,x)$ denote the electro and magnetic fields, respectively. The collision between charged particles is given by
\begin{equation*}
   \begin{split}
     Q(F,G)=&\nabla_{ v}\cdot\int_{\mathbf{R}^3}\Phi( v- v')
     \Big\{F( v')\nabla_{ v}G( v)-\nabla_{ v'}F( v')G( v)\Big\}d v'\\
     =&\sum_{i,j=1}^3\partial_i\int_{\mathbf{R}^3}\Phi^{ij}( v- v')
     \Big\{F( v')\partial_jG( v)-\partial_jF( v)G( v')\Big\}d v',
   \end{split}
\end{equation*}
where $\partial_i=\frac{\partial}{\partial v_i}$ for $i=1,2,3$ and the non-negative $3\times 3$ matrix $\Phi(v)=\left(\Phi^{ij}( v)\right)_{3\times 3}$ is the fundamental Landau (or Fokker-Planck) kernel \cite{Hinton}:
\begin{equation*}
  \Phi^{ij}( v)=\bigg(\delta_{ij}-\frac{ v_i v_j}{| v|^2}\bigg)| v|^{\gamma+2}, \quad -3\leq\gamma< -2.
\end{equation*}
The case $\gamma=-3$ corresponds to the Coulomb potential for the classical Landau operator which is originally found by Landau (1936). For the mathematical discussions of the general collision kernel with $\gamma>-3$, see the
review paper \cite{Villani-02} by Villani.

Notice that for notational simplicity and without loss of generality, all the
physical parameters appearing in the system, such as the particle masses $m_\pm$, their charges $e_\pm$, the light
speed $c$, and all other involving constants, have been chosen to be unit.
Moreover, the relativistic effects are neglected in our discussion of the model. This is normally
justified if the plasma temperature is much lower than the electron rest mass, cf. \cite{Helander-Sigmar}.
In addition, for the non dimensional representation and the physical
background of the system (\ref{VML}), interested readers may refer to \cite{Helander-Sigmar,Hinton,Krall-Trivelpiece,Lions,Villani-02} and the references cited therein.

As in \cite{Duan-VML}, we study the Cauchy problem (\ref{VML}) around the following normalized global Maxwellian
$$
  \mu=\mu(v)=(2\pi)^{-{3}/{2}}e^{-| v  |^2/2}
$$
with prescribed initial data
 \[
  F_\pm (0,x,v)=F_{0,\pm}(v,x), \quad E(0,x)=E_0(x), \quad B(0,x)=B_0(x)
  \]
satisfying the compatibility conditions
  \[
  \nabla_x\cdot E_0=\int_{\mathbb{R}^3}(F_{0,+}-F_{0,-})dv, \quad \nabla_x\cdot B_0=0.
  \]

For this purpose, we define the perturbation $f_\pm=f_\pm(t,x, v)$ by
$
F_\pm(t, x,  v  ) = \mu+ \mu^{1/2}f_\pm(t, x,  v).
$
Then, the Cauchy problem (\ref{VML}) of the VML system is reformulated as
\begin{equation} \label{f}
\begin{cases}
  \partial_tf_\pm+ v  \cdot\nabla_xf_\pm\pm(E+v\times B)\cdot\nabla_{ v  }f_\pm\mp E \cdot v \mu^{1/2}\mp \frac12 E\cdot v f_\pm+{ L}_\pm f={\Gamma}_\pm(f,f),\\
\partial_tE-\nabla_x\times B=-{\displaystyle\int_{\mathbb{R}^3}}v\mu^{1/2}(f_+-f_-)dv,\\
\partial_tB+\nabla_x\times E=0,\\
\nabla_x\cdot E={\displaystyle\int_{\mathbb{R}^3}}\mu^{1/2}(f_+-f_-)dv,\quad \nabla_x\cdot B=0
\end{cases}
\end{equation}
with initial data
\begin{equation}\label{f-initial}
f_\pm(0,x,v)=f_{0,\pm}(x,v),  \quad E(0,x)=E_0(x), \quad B(0,x)=B_0(x),
\end{equation}
which satisfy the compatibility conditions
\begin{equation}\label{compatibility conditions}
  \nabla_x\cdot E_0=\int_{\mathbb{R}^3}\mu^{1/2}(f_{0,+}-f_{0,-})dv, \quad \nabla_x\cdot B_0=0.
\end{equation}

If we denote $f=\left[f_+,f_-\right]$, then $(\ref{f})_1$ can be rewritten as
\begin{equation}\label{f-sgn}
\partial_t f+v\cdot\nabla_xf+q_0(E+v\times B)\cdot\nabla_v f-E\cdot v\mu^{1/2}q_1+Lf=\frac{q_0}{2}E\cdot vf+\Gamma(f,f)
\end{equation}
Here, $q_0=diag(1,-1)$, $q_1=[1,-1]$, and the linearized collision operator $Lf$ and the nonlinear collision term $\Gamma(f,f)$ are respectively defined by
\[Lf=[L_+f,L_-f],\quad\quad\quad\Gamma(f,g)=[\Gamma_+(f,g),\Gamma_-(f,g)]\]
with
\begin{equation*}
\begin{aligned}
{ L}_\pm f =& -{\bf \mu}^{-1/2}
\left\{{Q\left( \mu,{\bf \mu}^{1/2}(f_\pm+f_\mp)\right)+ 2Q\left( \mu^{1/2}f_\pm, \mu\right)}\right\},\\
{ \Gamma}_{\pm}(f,g) =&{\bf \mu}^{-1/2}Q\left({\bf \mu}^{1/2}f_{\pm},{\bf \mu}^{1/2}g_\pm\right)
+{\bf \mu}^{-1/2}Q\left({\bf \mu}^{1/2}f_{\pm},{\bf \mu}^{1/2}g_\mp\right).
\end{aligned}
\end{equation*}

For the linearized Landau collision operator ${L}$, it is well known, cf. \cite{Guo-CMP-02}, that it is non-negative and the null space $\mathcal{N}$ of ${ L}$ is given by
\begin{equation*}
{\mathcal{ N}}={\textrm{span}}\left\{[1,0]\mu^{1/2} , [0,1]\mu^{1/2}, [v_i,v_i]{\mu}^{1/2} (1\leq i\leq3),[|v|^2,|v|^2]{\bf \mu}^{1/2}\right\}.
\end{equation*}
If we define ${\bf P}$ as the orthogonal projection from $L^2({\mathbb{R}}^3_ v)\times L^2({\mathbb{R}}^3_ v)$ to $\mathcal{N}$, then for any given function $f(t, x, v )\in L^2({\mathbb{R}}^3_ v)$, one has
\begin{equation*}
  {\bf P}f ={a_+(t, x)[1,0]\mu^{1/2}+a_-(t, x)[0,1]\mu^{1/2}+\sum_{i=1}^{3}b_i(t, x) [1,1]v_i{\mu}^{1/2}+c(t, x)[1,1](| v|^2-3)}{\bf \mu}^{1/2}
\end{equation*}
with
\begin{equation*}
  a_\pm=\int_{{\mathbb{R}}^3}{\bf \mu}^{1/2}f_\pm d v,\quad
  b_i=\frac12\int_{{\mathbb{R}}^3} v  _i {\bf \mu}^{1/2}(f_++f_-)d v,\quad
  c=\frac{1}{6}\int_{{\mathbb{R}}^3}(| v|^2-3){\bf \mu}^{1/2}(f_++f_-) d v.
\end{equation*}
Therefore, we have the following macro-micro decomposition with respect to the given global Maxwellian $\mu$ which was introduced in \cite{Guo-IUMJ-04}
\begin{equation}\label{macro-micro}
 f(t,x, v)={\bf P}f(t,x, v)+\{{\bf I}-{\bf P}\}f(t, x, v)
\end{equation}
where ${\bf I}$ denotes the identity operator and ${\bf P}f$ and $\{{\bf I}-{\bf P}\}f$ are called the macroscopic and the microscopic component of $f(t,x,v)$, respectively.

What we are interested in this paper is to construct global smooth solutions to the Cauchy problem (\ref{f}), (\ref{f-initial}) near the global Maxwellian $\mu$ for $-3\leq \gamma<-2$ which includes the Coulomb potential. Before stating our main results, we first introduce some notations used throughout this manuscript. $C$ denotes some positive constant (generally large) and $\varepsilon$, $\kappa$, $\lambda$ stand for some positive constant (generally small), where $C$, $\varepsilon$, $\kappa$, and $\lambda$ may take different values in different places.
$A\lesssim B$ means that there is a generic constant $C> 0$ such that $A \leq   CB$. $A \sim B$ means $A\lesssim B$ and $B\lesssim A$. The multi-indices $ \alpha= [\alpha_1,\alpha_2, \alpha_3]$ and $\beta = [\beta_1, \beta_2, \beta_3]$ will be used to record spatial and velocity derivatives, respectively. And $\partial^{\alpha}_{\beta}=\partial^{\alpha_1}_{x_1}\partial^{\alpha_2}_{x_2}\partial^{\alpha_3}_{x_3}
\partial^{\beta_1}_{ v_1}\partial^{\beta_2}_{ v_2}\partial^{\beta_3}_{ v_3}$. Similarly, the notation $\partial^{\alpha}$ will be used when $\beta=0$ and likewise for $\partial_{\beta}$. The length
of $\alpha$ is denoted by $|\alpha|=\alpha_1 +\alpha_2 +\alpha_3$. $\alpha'\leq  \alpha$ means that no component of $\alpha'$ is greater than the corresponding component of $\alpha$, and $\alpha'<\alpha$ means that $\alpha'\leq  \alpha$ and $|\alpha'|<|\alpha|$. And it is convenient to write $\nabla_x^k=\partial^{\alpha}$ with $|\alpha|=k$. We also use $\langle\cdot,\cdot\rangle$ to denote the ${L^2_{ v}}$ inner product in ${\mathbb{ R}}^3_{ v}$, with the ${L^2}$ norm $|\cdot|_{2}$. For notational simplicity, $(\cdot, \cdot)$ denotes the ${L^2}$ inner product either in ${\mathbb{ R}}^3_{x}\times{\mathbb{ R}}^3_{ v }$ or in ${\mathbb{R}}^3_{x}$ with the $L^2$ norm $\|\cdot\|$.

As in \cite{Guo-CPDE-12}, we introduce the operator $\Lambda^s$ with $s\in\mathbb{R}$ by
\[
\left(\Lambda^sg\right)(t,x,v)=\int_{\mathbb{R}^3}|\xi|^{s}\hat{g}(t,\xi,v)e^{2\pi ix\cdot\xi}d\xi
=\int_{\mathbb{R}^3}|\xi|^{s}\mathcal{F}[g](t,\xi,v)e^{2\pi ix\cdot\xi}d\xi
\]
with $\hat{g}(t,\xi,v)\equiv\mathcal{F}[g](t,\xi,v)$ being the Fourier transform of $g(t,x,v)$ with respect to $x$.  The homogeneous Sobolev space $\dot{H}^s$ is the Banach space consisting of all $g$ satisfying  $\|g\|_{\dot{H}^{s}}<+\infty$, where
\[
\|g(t)\|_{\dot{H}^s}\equiv\left\|\left(\Lambda^s g\right)(t,x,v)\right\|_{L^2_{x,v}}=\left\||\xi|^s\hat{g}(t,\xi,v)\right\|_{L^2_{\xi,v}}.
\]
We will also use $H^s\left({\mathbb{R}}^3\right)$ to denote the usual Sobolev spaces with norm $\|\cdot\|_{H^s}$. We shall use $\|\cdot\|_p$ to record $L^p$ norms in either ${\mathbb{R}}^3_x\times{\mathbb{R}}^3_v$ or ${\mathbb{R}}^3_x$. Letting $w(v)\geq 1$ be a weight function, we use $\|\cdot\|_{p,w}$ to denote the weighted $L^p$ norms in ${\mathbb{R}}^3_x\times{\mathbb{R}}^3_v$ or ${\mathbb{R}}^3_x$. We also use $|\cdot|_p$ to denote $L^p$ norms in  ${\mathbb{R}}^3_v$, and $|\cdot|_{p,w}$ for the weighted $L^p$ norms in  ${\mathbb{R}}^3_v$. We will use the mixed
spatial-velocity spaces, e.g., $L^2_vH^s_x=L^2\left({\mathbb{R}}^3_v; H^s\left({\mathbb{R}}^3_x\right)\right)$, etc.

The Landau collision frequency $\sigma^{ij}( v)$ is given by
\begin{equation*}
  \sigma^{ij}( v)=\Phi^{ij}\ast{\bf \mu}( v)=\int_{{\mathbb{R}}^3}\Phi^{ij}( v- v^*){\bf \mu}( v^*)d v^*.
\end{equation*}
As in \cite{Duan-VML}, we introduce the following time-velocity weight function corresponding to the Landau operator:
\begin{equation}\label{weight}
    w_{\ell}(t, v)
    =\langle v\rangle^{-(\gamma+2)\ell}e^{\frac{q\langle v\rangle^2}{(1+t)^{\vartheta}}},
    \quad \langle v\rangle=\sqrt{1+| v|^2},\\
   \quad -3\leq\gamma<-2,\quad 0<q \ll 1
\end{equation}
with the parameter $\vartheta$ being taken as
$$
 \left\{
 \begin{array}{ll}
0<\vartheta\leq \frac s2, \quad & when\ \  s\in[\frac12,1],\\[2mm]
 0<\vartheta\leq \frac s2-\frac12,\quad & when \ \ s\in(1,\frac32),
 \end{array}
 \right.
 $$
if we define the the weighted norms
\begin{equation}\label{norm}
\begin{split}
    |f|_{\sigma,w}
    \sim&\left|\langle v\rangle^{\frac{\gamma+2}{2}}f\right|_{2,w}
    +\left|\langle v\rangle^{\frac{\gamma}{2}}\nabla_{ v}f\cdot\frac{ v}{| v|}\right|_{2,w}
    +\left|\langle v\rangle^{\frac{\gamma+2}{2}}\nabla_{ v}f\times\frac{ v}{| v|}\right|_{2,w}
\end{split}
\end{equation}
with
$$
|f|_{2,w}=|wf|_2,\quad \|f\|_{\sigma,w}=\left\||f|_{\sigma,w}\right\|,\quad |f|_{\sigma}=|f|_{\sigma,0},\quad \|f\|_{\sigma}=\left\||f|_{\sigma}\right\|,
$$
then it is well known, cf. \cite{Guo-CMP-02, Wang-12}, that the linear operator ${L}\geq0$ and is locally coercive in the sense that
\begin{equation}\label{coercive}
  \langle{ L}f,f\rangle\gtrsim |{\bf \{I-P\}}f|_{\sigma}^2.
\end{equation}
Moreover, for an integer $N\geq0$, $\ell\in \mathbb{R}$, and a given $f(t,x,v)$, we define the energy functional $\bar{\mathcal{E}}_{N,\ell}(t)$ and the corresponding energy dissipation rate functional $\bar{\mathcal{D}}_{N,\ell}(t)$ by
\begin{equation}\label{E}
\mathcal{\bar{E}}_{N,\ell}(t)\sim{\mathcal{E}}_{N,\ell}(t)+\left\|\Lambda^{-s}(f,E,B)\right\|^2
\end{equation}
and
\begin{equation}\label{D}
\mathcal{\bar{D}}_{N,\ell}(t)\sim{\mathcal{D}}_{N,\ell}(t)+\left\|\Lambda^{1-s}(a,b,c,E,B)\right\|^2
+\left\|\Lambda^{-s}(a_+-a_-,E)\right\|^2,
\end{equation}
respectively. Here
\begin{equation}\label{E-}
{\mathcal{E}}_{N,\ell}(t)\sim\sum_{|\alpha|+|\beta|\leq N}\left\|w_{\ell-|\beta|}
\partial^{\alpha}_{\beta}f\right\|^2+\|(E,B)\|_{H^N}^2,
\end{equation}
\begin{equation}\label{D-}
\begin{split}
{\mathcal{D}}_{N,\ell}(t)\sim&\sum_{1\leq|\alpha|\leq N} \left\|\partial^{\alpha}(a_{\pm},b,c)\right\|^2 +\sum_{|\alpha|+|\beta|\leq N}\left\|w_{\ell-|\beta|}
\partial^{\alpha}_{\beta}{\bf\{I-P\}}f\right\|^2_{\sigma}+\|a_+-a_-\|^2\\[2mm]
&+\|E\|_{H^{N-1}}^2+\left\|\nabla_x B\right\|_{H^{N-2}}^2+(1+t)^{-1-\vartheta}\sum_{|\alpha|+|\beta|\leq N}\left\|\langle v\rangle w_{\ell-|\beta|}\partial^{\alpha}_{\beta}{\bf\{I-P\}}f\right\|^2,
\end{split}
\end{equation}
In our analysis, we also need to define the energy functional without weight $\mathcal{E}_{N}(t)$, the higher order energy functional without weight $\mathcal{E}_{N_0}^{k}(t)$, and the higher order energy functional with weight $\mathcal{E}^k_{N_0,\ell}(t)$ as follows:
\begin{equation}\label{E_N}
\mathcal{E}_{N}(t)\sim\sum_{|\alpha|=0}^{N}\left\|\partial^\alpha(f,E,B)\right\|^2,
\end{equation}
\begin{equation}\label{E_k}
\mathcal{E}_{N_0}^{k}(t)\sim\sum_{|\alpha|=k}^{N_0}\left\|\partial^\alpha(f,E,B)\right\|^2,
\end{equation}
\begin{equation}\label{E_l-k}
\mathcal{E}^k_{N_0,\ell}(t)\sim\sum_{|\alpha|+|\beta\leq N,\atop{|\alpha|\geq k}}\left\|w_{\ell-|\beta|}\partial^\alpha_\beta f\right\|^2+\sum_{|\alpha|=k}^{N_0}\left\|\partial^\alpha(E,B)\right\|^2,
\end{equation}
and the corresponding energy dissipation rate functionals are given by
\begin{equation}\label{D_N}
\begin{split}
\mathcal{D}_{N}(t)\sim&
\|(E,a_+-a_-)\|^2+\sum_{1\leq|\alpha|\leq N-1}\left\|
\partial^\alpha({\bf P}f,E,B)\right\|^2\\[2mm]
&+\sum_{|\alpha|=N}\left\|\partial^\alpha {\bf P}f\right\|^2+\sum_{ |\alpha| \leq N}\left\|\partial^\alpha{\bf\{I-P\}}f\right\|^2_{\sigma},
\end{split}
\end{equation}
\begin{equation}\label{D_k}
\begin{split}
\mathcal{D}_{N_0}^{k}(t)\sim&
\left\|\nabla^{k}(E,a_+-a_-)\right\|^2+\sum_{k+1\leq|\alpha|\leq N_0-1}\left\|
\partial^\alpha({\bf P}f,E,B)\right\|^2\\[2mm]
&+\sum_{|\alpha|=N_0}\left\|\partial^\alpha {\bf P}f\right\|^2+\sum_{k\leq |\alpha| \leq N_0}\left\|\partial^\alpha{\bf\{I-P\}}f\right\|^2_{\sigma},
\end{split}
\end{equation}
\begin{equation}\label{D_l-k}
\begin{split}
\mathcal{D}_{N_0,\ell}^{k}(t)\sim&
\left\|\nabla^{k}(E,a_+-a_-)\right\|^2+\sum_{k+1\leq |\alpha|\leq N_0-1}\left\|
\partial^\alpha({\bf P}f,E,B)\right\|^2+\sum_{|\alpha|=N_0}\left\|\partial^\alpha {\bf P}f\right\|^2\\[2mm]
&+(1+t)^{-1-\vartheta}\sum_{|\alpha|+|\beta\leq N,\atop{|\alpha|\geq k}}\left\|\langle v\rangle w_{\ell-|\beta|}\partial^\alpha_\beta{\bf\{I-P\}}f\right\|^2+\sum_{|\alpha|+|\beta\leq N,\atop{|\alpha|\geq k}}\left\|w_{\ell-|\beta|}\partial^\alpha_\beta{\bf\{I-P\}}f\right\|^2_{\sigma},
\end{split}
\end{equation}
respectively. It is worth pointing out that the energy functional $\mathcal{E}_{N}^{k}(t)$ consists of only the terms whose possible derivatives are with respect to the spatial variable $x$ only.

With the above preparation in hand, our main result concerning the global solvability of the Cauchy problem (\ref{f}), (\ref{f-initial}) can be stated as follows.
\begin{theorem}\label{Th1.1} Suppose that $F_0(x,v)=\mu+\sqrt{\mu}f_0(x,v)\geq0$, $\frac 12\leq s<\frac 32$, and $-3\leq\gamma<-2$. Let
 $$
 \left\{
 \begin{array}{ll}
 N_0\geq 4, N=2N_0-1, \quad & when\ \  s\in[\frac12,1],\\[2mm]
 N_0\geq 3, N=2N_0,\quad & when \ \ s\in(1,\frac32),
 \end{array}
 \right.
 $$
and take  $l\geq N$ and $l_0\geq l+\frac{\gamma-2}{2(\gamma+2)}$.
Then there exists a positive constant $l'$ which depends only on $\gamma$ and $N_0$ such that if we assume further that
\begin{equation}\label{Y_0}
Y_0=\sum_{|\alpha|+|\beta|\leq N_0}\left\|w_{l_0+l^*-|\beta|}\partial^\alpha_\beta f_0\right\|+\sum_{|\alpha|+|\beta|\leq N}\left\|w_{l-|\beta|}\partial^\alpha_\beta f_0\right\|
+\|(E_0,B_0)\|_{H^N\bigcap \dot{H}^{-s}}+\|f_0\|_{\dot{H}^{-s}},\ l^*=l'+\frac{N_0-1}{2}
\end{equation}
is sufficiently small, the Cauchy problem (\ref{f}), (\ref{f-initial}) admits a unique global solution $(f(t,x,v),$ $E(t,x),$ $ B(t,x))$ satisfying $F(t,x,v)=\mu+\sqrt{\mu}f(t,x,v)\geq0$.
\end{theorem}
\begin{remark} Several remarks concerning Theorem 1.1 are listed below:
\begin{itemize}
\item Although only the case $-3\leq\gamma<-2$ is studied in this manuscript which corresponds to the so-called soft potential case, the case for hard potentials, i.e $-2\leq\gamma\leq0$, is much simpler and similar result can also be obtained by employing the argument used in this manuscript.
\item In the proof of Lemma \ref{lemma3.12}, we need to ask that the time decay rate of $\mathcal{E}^1_{N_0,l_0}(t)$ is strictly greater than $1$ and to guarantee that such a fact holds, we need to assume that $N_0\geq 4$ for $s\in[\frac12,1]$ and $N_0\geq 3$ for $s\in(1,\frac32)$.
\item Since in the proof of Lemma \ref{lemma4.3}, $N$ is assumed to satisfy $N>\frac53 N_0-\frac53$, while in the proof of Lemma \ref{Lemma4.5}, $N$ is further required to satisfy $N\geq 2N_0-2+s$. Putting these assumptions together, we can take $N=2N_0-1$ for $s\in[\frac12,1]$ and $N=2N_0$ for $s\in(1,\frac32)$.
\item The minimal regularity index, i.e., the lower bound on the parameter $N$, we imposed on the initial data is $N=7$, $N_0=4$ for $s\in[\frac12,1]$ and $N=6$, $N_0=3$ for $s\in(1,\frac32)$.
\item The precise value of the parameter $l'$ can be specified in the proofs of Theorem 1.1, cf. the proof of Lemma \ref{lemma4.3}.
\end{itemize}
\end{remark}
Our second result is concerned with the temporal decay estimates on the global solution $(f(t,x,v),$ $E(t,x),$ $B(t,x))$ obtained in Theorem 1.1. For result in this direction, we have from Theorem 1.1 that
\begin{theorem}\label{Th1.2}
Under the assumptions of Theorem \ref{Th1.1}, we have the following results:
\begin{itemize}
\item[(1).]
For $k=0,1,2,\cdots, N_0-2$, it holds that
\begin{equation}\label{TH2-1}
\mathcal{E}^k_{N_0}(t)\lesssim Y_0^2(1+t)^{-(s+k)}.
\end{equation}
\item[(2).]Let $0\leq i\leq k\leq N_0-3$ be an integer,
it holds that
\begin{equation}\label{TH2-2}
 \begin{split}
\mathcal{E}^k_{N_0,l_0+i/2}(t)
\lesssim Y_0^2(1+t)^{-k-s+i}, \quad \quad i=0,1,\cdots,k.
\end{split}
\end{equation}
Especially when $s\in[1,3/2)$, one has
\begin{equation}\label{TH2-3}
\mathcal{E}^k_{N_0,l_0+k/2+1/2}(t)
\lesssim Y_0^2(1+t)^{1-s}.\quad\quad
\end{equation}
\item[(3).]When $N_0+1\leq|\alpha|\leq N-1$, we have
\begin{equation}\label{TH2-4}
\begin{split}
\left\|\partial^\alpha f\right\|^2
\lesssim Y_0^2(1+t)^{-\frac{(N-|\alpha|)(N_0-2+s)}{N-N_0}}
\end{split}
\end{equation}
and
\begin{equation}\label{TH2-5}
\begin{split}
\left\|w_l\partial^\alpha f\right\|^2
\lesssim Y_0^2(1+t)^{-\frac{(N-|\alpha|)(N_0-3+s)-(1+\epsilon_0)(|\alpha|-N_0)/2}{N-N_0}}.
\end{split}
\end{equation}
\end{itemize}
\end{theorem}
\begin{remark} Some remarks related to Theorems \ref{Th1.1} and \ref{Th1.2} are given below:
\begin{itemize}
\item Theorem \ref{Th1.2} tells us that for $N_0\leq|\alpha|\leq N-1$, we can still get the time decay of $\|\partial^\alpha f\|^2$  and  $\|w_l\partial^\alpha f\|^2$. Compared with \cite{Duan-VML}, it is worth pointing out that we can obtain the time decay of the $k$-order energy functional $\mathcal{E}^k_{N_0,l_0}(t)$ defined in ($\ref{E_k}$) for the Vlasov-Maxwell-Landau system $(\ref{f})$.
\item In Theorem \ref{Th1.2}, it is worth pointing out that the largest value of the parameter $k$ appeared in $\mathcal{E}^k_{N_0}(t)$ is $N_0-2$ while the corresponding largest value of the parameter $k$ appeared in $\mathcal{E}^k_{N_0,\ell}(t)$ is $N_0-3$. The reason for such a difference is caused by the fact that the highest order $\|\partial^\alpha E\|^2$ appeared in (\ref{lemma3.11-1}) does not belong to the corresponding energy dissipation rate functional $\mathcal{D}^k_{N_0,\ell}(t)$.
\item In Theorems \ref{Th1.1} and \ref{Th1.2}, the parameter $s$ is assumed to satisfy $s\in[\frac12, \frac32)$, we note that by employing the argument used in \cite{Sohinger-Strain-2012}, if we replace the negative Sobolev norm $\|\cdot\|_{\dot{H}^{-s}}$ by the Besov norm $\|\cdot\|_{\dot{B}_{2,\infty}^{-s}}$ in Theorems \ref{Th1.1} and \ref{Th1.2}, we can also obtain the corresponding results similar to that of Theorem \ref{Th1.1} and  Theorem \ref{Th1.2} when $s\in[\frac12,\frac32]$.
\end{itemize}
\end{remark}

Before we outline our main ideas to deduce Theorems \ref{Th1.1} and \ref{Th1.2}, we first review some former results on the construction of global smooth solutions to the Vlasov-Maxwell-Landau system (\ref{VML}) near Maxwellians (For the corresponding results on the Boltzmann equation, the Landau equation, the Vlasov-Poisson-Landau system, the Vlasov-Poisson-Boltzmann system, and the Vlasov-Maxwell-Boltzmann equation based on the energy method developed in \cite{Guo-IUMJ-04,Liu-Yang-Yu,Liu-Yu}, interested readers are referred to  \cite{Duan-SIAM,Duan_Liu-Yang_Zhao-VMB-2013,DS-ARMA-11,DS-CPAM-11,DY-SIAM-10,DYZ-hard,DYZ-soft,DYZ-Landau,Guo-CMP-02,Guo-CPAM-02,Guo-IM-2003,Guo-JAMS-11,Guo-CPDE-12,Hsiao-Yu,Lions,Strain-CMP-2006,S-G-08,S-Z-12,Wang-12,Zhan-1994,Zhan-1994-TTSP} and the references cited therein):  Firstly, the case for $\gamma\geq -1$ can be treated  by employing the argument used in \cite{Guo-IM-2003} and \cite{Strain-CMP-2006}. In fact, although these two papers deal with
the Vlasov-Maxwell-Boltzmann system for only the hard-sphere model, the coercive estimate (\ref{coercive}) of the linearized Landau operator $L$ shows that if $\gamma\geq -1$, one has $\gamma+2\geq 1$ and consequently the linearized Landau operator $L$ has the stronger dissipative property than the corresponding linearized Boltzmann collision operator which can control those nonlinear terms with the velocity-growth rate $|v|$, typically occurring to $E\cdot vf$ suitably for hard sphere model to yield the desired global solvability result; Secondly, the relativistic version of the
Vlasov-Maxwell-Landau system is studied by Strain-Guo \cite{S-G-04}. For such a case, due to
the boundedness of the relativistic velocity and the special form of the relativistic
Maxwellian, the velocity growth phenomenon in the nonlinear term disappears,
and instead the more complex property of the relativistic collision operator was
analyzed there; Finally, by employing the time-velocity weight function $w_\ell(t,v)$ given in (\ref{weight}) which is first introduced in \cite{DYZ-hard, DYZ-soft} to treat the Vlasov-Poisson-Boltzmann system for non-hard sphere interactions and in \cite{DYZ-Landau} to deal with the Vlasov-Poisson-Landau system which includes the Coulomb potential, Duan \cite{Duan-VML}
can treat the whole range of $\gamma\geq-3$. Let's use more words to explain the main ideas used in \cite{Duan-VML}. As pointed out in \cite{Duan-VML}, the main mathematical difficulties to deuce the global solvability result to the Vlasov-Maxwell-Landau system (\ref{VML}) near Maxwellians are the following:
\begin{itemize}
\item The degeneration of dissipation at large velocity for the linearized Landau
operator $L$ with soft potentials;

\item The velocity-growth of the nonlinear term with the velocity-growth rate $|v|$;

\item The regularity-loss of the electromagnetic field.
\end{itemize}
To overcome the above mentioned difficulties, the main arguments used in \cite{Duan-VML} is
\begin{itemize}
\item A refined weighted energy method based on the time-velocity weight function $w_\ell(t,v)$ defined in (\ref{weight}) is introduced which use essentially the extra dissipative term generated by such a weight function which corresponds to the last term in the energy dissipation rate functional $\mathcal{\bar{D}}_{N,l}(t)$;

\item Motivated by the argument developed in \cite{Hosono-Kawashima-2006} to deduce the decay property of solutions to nonlinear equations of regularity-loss type, a time-weighted energy estimate is designed to close the analysis, which implies that although the $L^2-$norm of terms with the highest order derivative with respect to $x$ of the solutions of the Vlasov-Maxwell-Landau system may can only be bounded by some function of $t$ which increases as time evolves, the $L^2-$norm of terms with lower order derivatives with respect to $x$ still enjoy some decay rates.
\end{itemize}
Based on the above arguments and by combining the decay of solutions to the corresponding linearized system with the Duhamel principle, Duan \cite{Duan-VML} can indeed close the analysis provided that the regularity index imposed on the initial perturbation is 14, i.e. $N\geq 14$ and certain norms of the initial perturbation, especially $\|w_\ell f_0\|_{Z^1}$ and  $\|(E_0,B_0)\|_{L^1}$, are assumed to be sufficiently small.

Motivated by \cite{Guo-CPDE-12} and \cite{Wang-12}, the main purpose of our present manuscript is trying to study such a problem by a different method which does not rely on the decay of the corresponding linearized system and the Duhamel principle. Before outlining the main ideas, we first give the philosophy of our present study as follows:
\begin{itemize}
\item As mentioned above, the analysis in \cite{Duan-VML} asks that the minimal regularity index imposed on the initial perturbation is 14, i.e. $N\geq 14$ and certain norms of the initial perturbation, especially $\|w_\ell f_0\|_{Z^1}$ and  $\|(E_0,B_0)\|_{L^1}$, are assumed to be sufficiently small. Thus a natural question is whether these assumptions on the minimal regularity index of the initial data and on the smallness of the initial perturbation can be relaxed or not? In fact the main results obtained in this manuscript show that the smallness of $\|w_\ell f_0\|_{Z^1}$ and  $\|(E_0,B_0)\|_{L^1}$ can be replaced by the weaker assumption that  $\|(E_0,B_0)\|^2_{\dot{H}^{-s}}+\|f_0\|^2_{\dot{H}^{-s}}$ is small and as pointed out in Remark 1.1, the minimal regularity index is relaxed to $N=7$ for $s\in[\frac12,1]$ and $N=6$ for $s\in(1,\frac32)$.

\item Another motivation of our present study is to deduce the temporal decay estimates on global solutions of some kinetic equations such as the one-species Vlasov-Maxwell-Landau systems, one-species Vlasov-Maxwell-Boltzmann systems, etc.. For result in this direction, Duan studied the one-species Vlasov-Maxwell-Boltzmann system for hard sphere interaction in \cite{Duan-SIAM}. Although the global solvability result can be established there by employing the energy estimates, since the decay of the corresponding linearized systems is too slow, it seems impossible to combine the decay of the linearized systems with the Duhamel principle to yield the desired temporal decay of the global solutions obtained, cf. \cite{Duan-SIAM} for some discussions on this problem. We note, however, that our recent calculations show that the argument used in this manuscript can be adapted to treat the one-species Vlasov-Maxwell-Landau system for $\gamma>-3$ and the one-species Vlasov-Maxwell-Boltzmann system for $\gamma>-1$ and the global solvability results together with the temporal decay estimates can also be obtained, cf. \cite{Lei-Zhao-13}.

\end{itemize}

Now we sketch the main ideas to deduce our main results:
\begin{itemize}
\item The first point is to add $\|\Lambda^{-s}(f,E,B)\|^2$  and  $\|\Lambda^{1-s}(f,E,B)\|^2$  into the energy functional $\mathcal{\bar{E}}_{N,l}(t)$ defined by ($\ref{E}$) and the energy dissipation rate functional $\mathcal{\bar{D}}_{N,l}(t)$ defined by ($\ref{D}$) respectively. It is worth pointing out that it is to deduce the desired estimates on $\|\Lambda^{-s}(f,E,B)\|^2$  and  $\|\Lambda^{1-s}(f,E,B)\|^2$ that we need to ask that $\frac 12\leq s< \frac 32$. For details see the discussions before Lemma \ref{Lemma4.2};
\item Secondly,  as in \cite{Duan-VML}, our method is a refined weighted energy method which is based on the time-velocity exponential weight function $w_{\ell}(t,v)=\langle v\rangle^{-(\gamma+2)\ell}e^{\frac{q\langle v\rangle^2}{(1+t)^{\vartheta}}}$ given in (\ref{weight}) and the temporal decay rates of the electromagnetic field $(E,B)$ plays an essential role in our analysis. It is worth to emphasizing that the reason why one chooses the exponential factor $e^{\frac{q\langle v\rangle^2}{(1+t)^{\vartheta}}}$ in the weight function $w_\ell(t,v)$ is that it can generate extra dissipation term which corresponds to the last term in the energy dissipation rate functional $\mathcal{\bar{D}}_{N,l}(t)$ defined by \eqref{D}. Unlike \cite{Duan-VML}, we use the interpolation technique to deduce the decay of the solutions, especially on the electromagnetic field $(E,B)$. To relax the assumption on the minimal regularity index imposed on the initial data, our main observation is that the energy functional $\mathcal{E}^k_{N_0}(t)$ decays faster as $k$ increases. To deduce such a result, we have by the interpolation argument that
\begin{equation}\label{1.26}
 \frac{d}{dt}\mathcal{E}_{N_0}^{k}(t)+\mathcal{D}_{N_0}^{k}(t)\leq 0,\quad\quad \frac{d}{dt}\mathcal{E}^k_{N_0,\ell}(t)+\mathcal{D}^k_{N_0,\ell}(t)\lesssim \sum_{|\alpha|=N_0}\|\partial^\alpha E\|^2,
\end{equation}
 so that we can deduce the time decay of $\mathcal{E}^k_{N_0}(t)$ which can further yield the time decay of $\mathcal{E}^k_{N_0,l_0}(t)$. Based on the time decay of $\mathcal{E}^k_{N_0}(t)$ and $\mathcal{E}^1_{N_0,l_0}(t)$, we can then close the a priori assumption given in (\ref{E-priori}) and then the global solvability result follows immediately from the continuation argument. Here it is worth pointing out that although the energy energy functional $\mathcal{E}^k_{N_0}(t)$ contains only the spatial derivatives, the coercive estimate (\ref{coercive}) of the linearized Landau operator $L$ can indeed be used to control the corresponding nonlinear terms related to the Lorenz force to yield the desired estimates listed in (\ref{1.26}).
\end{itemize}

The rest of this paper is organized as follows. In Section 2, we list some basic lemmas which will be used later. Section 3 is devoted to deducing the desired energy estimates and the proofs of Theorem \ref{Th1.1} and Theorem \ref{Th1.2} will be given in Section 4.

\section{Preliminary}
\setcounter{equation}{0}

In this section, we collect several fundamental results to be used frequently later. The first lemma concerns the weighted coercivity estimate on the linearized collision operators $L$ and the weighted estimates on the nonlinear collision operator $\Gamma$ with respect to the weight $w_l(t,v)$ given in \eqref{weight} whose proofs can be found in \cite{Wang-12}.
\begin{lemma}\label{Lemma2.1} It holds that
\begin{itemize}
\item[(i).] (\cite{Wang-12}) Let $w_{l}$ be given in $(\ref{weight})$, there exist positive constants $\eta>0$ and $C_\eta>0$ such that
\begin{equation}\label{L_0}
  \left\langle w_{l}^2(t, v){ L}f,f\right\rangle\geq \eta|f|_{\sigma,w_{l}}^2-C_{\eta}\left|\chi_{\{| v|\leq2C_{\eta}\}}f\right|^2.
\end{equation}
Here $\chi_{\{| v|\leq2C_{\eta}\}}$ denotes the characteristic function of the set $\{v=( v_1,  v_2,  v_3) \in {\mathbb{R}}^3: | v|\leq2C_{\eta}\}$. Moreover, for the case of $|\beta|>0$, we have for any small $\kappa>0$ that there exists a positive constant $C_{\eta,\kappa}>0$ such that
\begin{equation}\label{L_v}
  \left\langle w_{l}^2(t, v)\partial_{\beta}{L}f,\partial_{\beta}f\right\rangle\geq \eta\left|\partial_{\beta}f\right|_{\sigma,w_{l}}^2
  -\kappa \sum_{|\beta'|=|\beta|}\left|\partial_{\beta'}f\right|_{\sigma,w_{l}}^2
  -C_{\eta,\kappa}\sum_{|\beta'|<|\beta|}\left|\partial_{\beta'}f\right|_{\sigma,w_{l}}^2.
\end{equation}
\item[(ii).](\cite{Wang-12}) Let $w_{l}$ be given in $(\ref{weight})$, then we have
\begin{equation}\label{Gamma-w}
  \left\langle w_{l}^2(t, v)\partial^{\alpha}_{\beta}\Gamma(g_1,g_2),\partial^{\alpha}_{\beta}g_3\right\rangle
  \lesssim \sum_{\alpha_1\leq \alpha\atop \bar{\beta}\leq\beta_1\leq\beta}
 \left|\mu^{\delta}\partial^{\alpha_1}_{\bar{\beta}}g_1\right|_2
  \left|\partial^{\alpha-\alpha_1}_{\beta-\beta_1}g_2\right|_{\sigma,w_{l}}
  \left(\left|\partial^{\alpha}_{\beta}g_3\right|_{\sigma,w_{l}}
  +l\left|\partial^{\alpha}_{\beta}g_3\right|_{2,\frac{w_{l}}{\langle v\rangle^{-\gamma/2}}}\right).
\end{equation}
Here $\delta>0$ is a sufficiently small universal constant. In particular, we have
\begin{equation}\label{Gamma-0}
  \left\langle\Gamma(g_1,g_2),g_3\right\rangle
  \lesssim \left|\mu^{\delta}g_1\right|_2\left|g_2\right|_{\sigma}\left|g_3\right|_{\sigma}.
\end{equation}
\end{itemize}
\end{lemma}
In what follows, we will collect some inequalities which will be used throughout the rest of this manuscript. The first one is the Sobolev interpolation inequalities.
\begin{lemma}\label{lemma2.2}(\cite{Wang-12})
Let $2\leq p<\infty$ and $k,\ell, m\in\mathbb{R}$, we have
\begin{equation}
\left\|\nabla^k f\right\|_{L^p}\lesssim\left\|\nabla^\ell f\right\|^{\theta}\left\|\nabla^m f\right\|^{1-\theta},
\end{equation}
where $0\leq \theta\leq1$ and $\ell$ satisfy
\begin{equation}
\frac{1}{p}-\frac k3=\left(\frac12-\frac\ell3\right)\theta+\left(\frac12-\frac m3\right)(1-\theta).
\end{equation}
For the case $p=+\infty$, we have
\begin{equation}
\left\|\nabla^k f\right\|_{L^\infty}\lesssim\left\|\nabla^\ell f\right\|^{\theta}\left\|\nabla^m f\right\|^{1-\theta},
\end{equation}
where $0\leq \theta\leq1$ and $\ell$ satisfy
\begin{equation}
-\frac k3=\left(\frac12-\frac\ell3\right)\theta+\left(\frac12-\frac m3\right)(1-\theta).
\end{equation}
Here we need to assume that $\ell\leq k+1$, $m\geq k+2$. Throughout the rest of this manuscript, for each positive integer $k$,
$\left\|\nabla^k f\right\|\sim\sum\limits_{|\alpha|=k}\left\|\partial^\alpha f\right\|.$
\end{lemma}
For the $L^q-L^q$ type estimate on $\Lambda^{-s}f$, we have
\begin{lemma}\label{lemma2.3}(\cite{Stein-1970})
Let $0<s<3$, $1<p<q<\infty$, $\frac1q+\frac s3=\frac1p$, then we have
\begin{equation}
\|\Lambda^{-s}f\|_{L^q}\lesssim\|f\|_{L^p}.
\end{equation}
\end{lemma}

Based on Lemma \ref{lemma2.2} and Lemma \ref{lemma2.3}, we have the following corollary which will be used frequently later
\begin{corollary}\label{corrollary}
Let $0<s<\frac 32$, then for any positive integer $k$ and any nonnegative integer $j$ satisfying $0\leq j\leq k$, we have
$$
\|f\|_{L^{\frac{12}{3+2s}}_x}\lesssim \left\|\Lambda^{\frac34-\frac s2}f\right\|,\quad \|f\|_{L^{\frac 3s}_x}\leq \left\|\Lambda^{\frac32-s}f\right\|,
$$
$$
\|f\|_{L_x^\infty}\lesssim\left\|\Lambda^{-s}f\right\|^{\frac{2k-1}{2(k+1+s)}}
\left\|\nabla^{k+1}f\right\|^{\frac{3+2s}{2(k+1+s)}},
$$
$$
\left\|\nabla^j f\right\|_{L^6_x}\lesssim \left\|\Lambda^{-s}f\right\|^{\frac{k-j}{k+1+s}} \left\|\nabla^{k+1}f\right\|^{\frac{j+s+1}{k+1+s}},
$$
and
$$
\left\|\nabla^jf\right\|_{L^3_x}\lesssim \left\|\Lambda^{-s}f\right\|^{\frac{2k-2j+1}{2k+2+2s}} \left\|\nabla^{k+1}f\right\|^{\frac{2j+2s+1}{2k+2+2s}}.
$$
\end{corollary}
In many places, we will use Minkowski's integral inequality to interchange the orders of integration over $x$ and $v$.
\begin{lemma}\label{lemma2.4}(\cite{Guo-CPDE-12})
For $1\leq p\leq q\leq \infty$, we have
\begin{equation}
\|f\|_{L^q_xL^p_v}\leq \|f\|_{L^p_vL^q_x}.
\end{equation}
\end{lemma}

Before concluding this section, we list below some identities concerning the the macroscopic part of $f(t,x,v)$.
To this end, for any solution $f(t,x,v)$ of the VML system (\ref{f}), by applying the macro-micro decomposition $(\ref{macro-micro})$ introduced in \cite{Guo-IUMJ-04} and by defining moment functions $A_{mj}(f)$ and $B_j(f),~1\leq m,j\leq3,$ by
\begin{equation*}
A_{mj}(f)=\int_{{\mathbb{R}}^3}\left( v_m v_j-1\right)\mu^{1/2}fd v,\quad B_j(f)=\frac{1}{10}\int_{{\mathbb{R}}^3}\left(| v|^2-5\right) v_j\mu^{1/2}fd v
\end{equation*}
as in \cite{DS-ARMA-11}, one can then derive from $(\ref{f})$ a fluid-type system of equations
\begin{equation}\label{Macro-equation}
\begin{cases}
\partial_ta_\pm+\nabla_x\cdot b+\nabla_x\cdot
\left\langle v\mu^{1/2},\{{\bf I_\pm-P_\pm}\}f\right\rangle=\left\langle \mu^{1/2}, g_\pm\right\rangle,\\
\partial_t\left(b_i+\left\langle v_i\mu^{1/2},\{{\bf I_\pm-P_\pm}\}f\right\rangle\right)+\partial_i\left(a_\pm+2c\right)\mp E_i\\
\quad\quad+\nabla_x\cdot\left\langle vv_i\mu^{1/2}, \{{\bf I_\pm-P_\pm}\}f \right\rangle=\left\langle v_i\mu^{1/2}, g_\pm+L_\pm f\right\rangle,\\
\partial_t\left(c+\frac16\left\langle (|v|^2-3)\mu^{1/2},\{{\bf I_\pm-P_\pm}\}f\right\rangle\right)+\frac13\nabla_x\cdot b\\
\quad\quad+\frac16\nabla_x\left\langle (|v|^2-3)v\mu^{1/2},\{{\bf I_\pm-P_\pm}\}f\right\rangle=\left\langle(|v|^2-3)\mu^{1/2},g_\pm-L_\pm f \right\rangle,
\end{cases}
\end{equation}
and
\begin{equation}\label{Micro-equation}
\begin{cases}
\partial_t[A_{ii}(\{{\bf I_\pm-P_\pm}\}f)+2c]+2\partial_ib_i
=A_{ii}(r_\pm+g_\pm),\\
\partial_tA_{ij}(\{{\bf I_\pm-P_\pm}\}f)+\partial_jb_i+\partial_ib_j
+\nabla_x\cdot \langle v\mu^{1/2},\{{\bf I_\pm-P_\pm}\}f\rangle
=A_{ij}(r_\pm+g_\pm),\\
\partial_t B_{j}(\{{\bf I_\pm-P_\pm}\}f)+\partial_jc=B_j(r_\pm+g_\pm),
\end{cases}
\end{equation}
where
\begin{equation}\label{r-G}
  r_\pm=- v\cdot\nabla_x\{{\bf I_\pm-P_\pm}\}f-{L_\pm}f,\quad g_\pm=\frac12 v\cdot E f_\pm\mp (E+v\times B)\cdot\nabla_{ v}f_\pm+{ \Gamma}_\pm(f,f).
\end{equation}
Setting
$$
G\equiv\left\langle v\mu^{1/2},\{{\bf I-P}\}f \cdot q_1 \right\rangle
$$
and noticing $\langle \mu^{1/2}, g_\pm\rangle=0$, we can get from (\ref{Macro-equation}), (\ref{Micro-equation}) that
\begin{equation}\label{Macro-equation1}
\begin{cases}
\partial_t\big(\frac{a_++a_-}2\big)+\nabla_x\cdot b=0\\
\partial_tb_i+\partial_i\big(\frac{a_++a_-}2+2c\big)+\frac12\sum\limits_{j=1}^3\partial_jA_{ij}(\{{\bf I-P}\}f\cdot [1,1])=\frac{a_+-a_-}{2}E_i+[G\times B]_i,\\
\partial_tc+\frac13\nabla_x\cdot b+\frac56\sum\limits_{i=1}^3\partial_i B_i(\{{\bf I-P}\}f\cdot [1,1])=\frac16 G\cdot E,
\end{cases}
\end{equation}
and
\begin{equation}\label{Micro-equation1}
\begin{cases}
\partial_t[\frac12A_{ij}(\{{\bf I-P}\}f\cdot [1,1])+2c\delta_{ij}]+\partial_jb_i+\partial_ib_j
=\frac12A_{ij}(r_++r_-+g_++g_-),\\
\frac12\partial_t B_{j}(\{{\bf I-P}\}f\cdot [1,1])+\partial_jc=\frac12B_{j}(r_++r_-+g_++g_-).
\end{cases}
\end{equation}
Moreover, by using the third equation of (\ref{Macro-equation1}) to replace $\partial_t c$ in the first equation of (\ref{Micro-equation1}), one has
\begin{equation}\label{Micro-equation2}
\begin{split}
\frac12\partial_tA_{ij}(\{{\bf I-P}\}f&\cdot [1,1])+\partial_jb_i+\partial_ib_j-\frac23\delta_{ij}\nabla_x\cdot b\\
&-\frac53\delta_{ij}\nabla_x\cdot B(\{{\bf I-P}\}f\cdot [1,1])=\frac12A_{ij}(r_++r_-+g_++g_-)-\frac13\delta_{ij}G\cdot E.
\end{split}
\end{equation}
In order to further obtain the dissipation rate related to $a_\pm$ from the formula
\[
a_+^2+a_-^2=\frac{|a_++a_-|^2}{2}+\frac{|a_+-a_-|^2}{2},
\]
we need to consider the dissipation of $a_+-a_-$. For that purpose, one can get from $(\ref{Macro-equation1})_1$ and $(\ref{Macro-equation1})_2$ that
\begin{equation}\label{a_+-a_--original}
\begin{cases}
\partial_t(a_+-a_-)+\nabla_x\cdot G=0,\\
\partial_tG+\nabla_x(a_+-a_-)-2E+\nabla_x\cdot A(\{{\bf I-P}\}f\cdot q_1)\\
\qquad \qquad=E(a_++a_-)+2b\times B+\left\langle [v,-v]\mu^{1/2},Lf+\Gamma(f,f)\right\rangle.
\end{cases}
\end{equation}

\section{Energy estimates}
This section is devoted to deducing certain uniform-in-time a priori estimates on the solution $(f(t,x,v), E(t,x),$ $ B(t,x))$ to the Cauchy problem (\ref{f}) and (\ref{f-initial}). For this purpose, suppose that the Cauchy problem (\ref{f}) and (\ref{f-initial}) admits a smooth solution $f(t,x,v)$ over $ 0\leq t\leq T$ for some $0< T \leq\infty$, we now turn to deduce certain energy type a priori estimates on the solution $(f(t,x,v), E(t,x), B(t,x))$ to the Cauchy problem (\ref{f}) and (\ref{f-initial}). Before doing so, let's recall that the parameter $\gamma$ is assumed to satisfy $-3\leq \gamma<-2$ throughout this manuscript.

The first one is on the estimate on $\|(f,E,B)(t)\|_{\dot{H}^{-s}}$.
\begin{lemma}\label{Lemma4.1} Under the assumptions stated above, we have for $s\in[\frac12,\frac32)$ that
\begin{equation}\label{Lemma4.1-1}
\begin{aligned}
\frac{d}{dt}\left(\left\|\Lambda^{-s}f\right\|^2+\left\|\Lambda^{-s}(E,B)\right\|^2\right)+&\left\|\Lambda^{-s}\{{\bf I-P}\}f\right\|_{\sigma}^2\lesssim\left(\mathcal{\bar{E}}_{0,0}(t)\right)^{1/2}\left(\left\|\Lambda^{\frac34-\frac s2}(E,B)\right\|^2+\left\|\Lambda^{\frac34-\frac s2}f\right\|^2_{\sigma}\right)\\
&+\mathcal{\bar{E}}_{1,\frac{\gamma-2}{2(\gamma+2)}}(t)\left(\left\|\Lambda^{\frac32-s}(E,B)\right\|^2
+\left\|\Lambda^{\frac32-s}f\right\|_\sigma^2+\left\|\Lambda^{\frac32-s}\nabla_v\{{\bf I-P}\}f\right\|_{\sigma}^2\right).
\end{aligned}
\end{equation}
\end{lemma}
\begin{proof}
We have by taking Fourier transform of (\ref{f}) with respect to $x$,  multiplying the resulting identity by $|\xi|^{-2{s}}\bar{\hat{f}}_\pm$ with $\bar{\hat{f}}_\pm$ being the complex conjugate of $\hat{f}_\pm$, and integrating the final result with respect to $\xi$ and $v$ over $\mathbb{R}^3_\xi\times\mathbb{R}^3_v$ that
we can get
\begin{equation}\label{3.3}
\left(\partial_t\hat{f}_\pm+ v  \cdot\mathcal{F}[\nabla_xf_\pm]\pm\mathcal{F}[(E+v\times B)\cdot\nabla_{ v  }f_\pm]\mp\frac 1 2 v  \cdot\mathcal{F}[E f_\pm]\mp\hat{E}\cdot v {\bf \mu}^{\frac12}+\mathcal{F}[{ L}_\pm f]-\mathcal{F}[{ \Gamma}_\pm(f,f)]\ \Big|\ |\xi|^{-2s}\hat{f}\right)=0.
\end{equation}
Here and in the rest of this manuscript $(f\ |\ g)$ is used to denote $(f,\bar{g})$. Recall that throughout this manuscript, $\mathcal{F}[g](t,\xi,v)=\hat{g}(t,\xi,v)$ denotes the Fourier transform of $g(t,x,v)$ with respect to $x$.

(\ref{3.3}) together with (\ref{coercive}) yield
\begin{equation}\label{4.6}
\begin{aligned}
&\frac{d}{dt}\left(\left\|\Lambda^{-s}f\right\|^2+\left\|\Lambda^{-s}(E,B)\right\|^2\right)+\left\|\Lambda^{-s}\{{\bf I-P}\}f\right\|_{\sigma}^2\\
\lesssim&\underbrace{\sum_\pm\left|\left(\mathcal{F}[E\cdot\nabla_{ v  }f_\pm]\ \Big|\ |\xi|^{-2s}\hat{f}_\pm\right)\right|}_{I_1}+\underbrace{\sum_\pm\left|\left(\mathcal{F}[(v\times B)\cdot\nabla_{ v  }f_\pm]\ \Big|\ |\xi|^{-2s}\hat{f}_\pm\right)\right|}_{I_2}\\
&+\underbrace{\sum_\pm\left|\left(v  \cdot\mathcal{F}[E f_\pm]\ \Big|\ |\xi|^{-2s}\hat{f}_\pm\right)\right|}_{I_3}+\underbrace{\sum_\pm\left|\left(\mathcal{F}[{ \Gamma}_\pm(f,f)]\ \Big|\ |\xi|^{-2s}\hat{f}_\pm\right)\right|}_{I_4}.
\end{aligned}
\end{equation}
To estimate $I_j$ $(j=1,2,3,4)$, we have from Lemma \ref{lemma2.2}, Lemma \ref{lemma2.3}, Corollary \ref{corrollary}, and Lemma \ref{lemma2.4} that
\begin{equation*}
\begin{aligned}
I_1\lesssim&\left|\left(\mathcal{F}[E\cdot\nabla_{ v  }f]\ \Big|\ |\xi|^{-2s}\mathcal{F}[{\bf P}f]\right)\right|
+\left|\left(\mathcal{F}[E\cdot\nabla_{ v  }f]\ \Big|\ |\xi|^{-2s}\mathcal{F}[{\bf \{I-P\}}f]\right)\right|\\
\lesssim&\left\|\Lambda^{-s}\left(E\cdot \mu^{\delta}f\right)\right\|\left\|\Lambda^{-s}\left(\mu^{\delta}f\right)\right\|
+\left\|\Lambda^{-s}\left(E f \langle v\rangle^{-\frac{\gamma}{2}}\right)\right\| \left\|\Lambda^{-s}\left({\bf\{I-P\}}f\right)\right\|_{\sigma}\\
\lesssim&\left\|E\cdot \mu^{\delta}f\right\|_{L_x^{\frac{6}{3+2s}}}\left\|\Lambda^{-s}\left(\mu^{\delta}f\right)\right\|
+\left\|E f \langle v\rangle^{-\frac{\gamma}{2}}\right\|_{L_x^{\frac{6}{3+2s}}}\left\|\Lambda^{-s}({\bf\{I-P\}}f)\right\|_{\sigma}\\
\lesssim&\|E\|_{L_x^{\frac{12}{3+2s}}}\left\|\mu^{\delta}f\right\|_{L_x^{\frac{12}{3+2s}}}
\left\|\Lambda^{-s}\left(\mu^{\delta}f\right)\right\|+\left(\|E\|_{L_x^{\frac{3}{s}}}
\left\|f \langle v\rangle^{-\frac{\gamma}{2}}\right\|\right)^2+\varepsilon\left\|\Lambda^{-s}({\bf\{I-P\}}f)\right\|^2_{\sigma}\\
\lesssim&\left\|\Lambda^{\frac34-\frac s2}E\right\|\left\|\Lambda^{\frac34-\frac s2}\left(\mu^{\delta}f\right)\right\|\left\|\Lambda^{-s}\left(\mu^{\delta}f\right)\right\|
+\left\|\Lambda^{\frac32-s}E\right\|^2\left\|f \langle v\rangle^{-\frac{\gamma}{2}}\right\|^2+\varepsilon\left\|\Lambda^{-s}({\bf\{I-P\}}f)\right\|^2_{\sigma}\\
\lesssim&\left(\mathcal{\bar{E}}_{0,0}(t)\right)^{1/2}\left(\left\|\Lambda^{\frac34-\frac s2}E\right\|^2+\left\|\Lambda^{\frac34-\frac s2}f\right\|^2_{\sigma}\right)+\mathcal{\bar{E}}_{0,\frac{\gamma}{2(\gamma+2)}}(t)\left\|\Lambda^{\frac32-s}E\right\|^2
+\varepsilon\left\|\Lambda^{-s}({\bf\{I-P\}}f)\right\|^2_{\sigma}.
\end{aligned}
\end{equation*}
For $I_2$ and $I_3$, we have by repeating the argument used in deducing the estimate on $I_1$ that
\begin{eqnarray*}
I_2+I_3&\lesssim&\left(\mathcal{\bar{E}}_{0,0}(t)\right)^{1/2}\left(\left\|\Lambda^{\frac34-\frac s2}(E,B)\right\|^2+\left\|\Lambda^{\frac34-\frac s2}f\right\|^2_{\sigma}\right)\\
&&+\mathcal{\bar{E}}_{0,\frac{\gamma-2}{2(\gamma+2)}}(t)\left\|\Lambda^{\frac32-s}(E,B)\right\|^2
+\varepsilon\left\|\Lambda^{-s}({\bf\{I-P\}}f)\right\|^2_{\sigma}.
\end{eqnarray*}
As to $I_4$, due to
\begin{eqnarray*}
&&\left(\mathcal{F}[{\bf \Gamma}(f,f)]\ \Big|\ |\xi|^{-2s}\mathcal{F}\left[{{\bf\{I-P\}}f}\right]\right)\\
&=&\left(\mathcal{F}\left[\partial_i(\{\Phi^{ij}\ast[\mu^{1/2}f]\}\partial_j f)-\{\Phi^{ij}\ast[v_i\mu^{1/2}f]\}\partial_j f\right.\right.\\
&&\left.\left.-\partial_i(\{\Phi^{ij}\ast[\mu^{1/2}\partial_jf]\} f)+\{\Phi^{ij}\ast[v_i\mu^{1/2}\partial_jf]\} f\right],|\xi|^{-2s}\mathcal{F}\left[{{\bf\{I-P\}}f}\right]\right)\\
&=&\underbrace{-\left(\mathcal{F}\left[\{\Phi^{ij}\ast[\mu^{1/2}f]\}\partial_j f\right]\ \Big|\ |\xi|^{-2s}\mathcal{F}\left[{{\bf\{I-P\}}f}\right]\right)}_{I_{4,1}}\\
&&+\underbrace{\left(\mathcal{F}\left[\{\Phi^{ij}\ast[\mu^{1/2}\partial_jf]\} f\right]\ \Big|\ |\xi|^{-2s}\mathcal{F}\left[{{\bf\{I-P\}}f}\right]\right)}_{I_{4,2}}\\
&&\underbrace{-\left(\mathcal{F}\left[\{\Phi^{ij}\ast[v_i\mu^{1/2}f]\}\partial_j f\right]\ \Big|\ |\xi|^{-2s}\mathcal{F}\left[{{\bf\{I-P\}}f}\right]\right)}_{I_{4,3}}\\
&&+\underbrace{\left(\mathcal{F}\left[\{\Phi^{ij}\ast[v_i\mu^{1/2}\partial_jf]\} f\right]\ \Big|\ |\xi|^{-2s}\mathcal{F}\left[{{\bf\{I-P\}}f}\right]\right)}_{I_{4,4}},
\end{eqnarray*}
we have from Lemma \ref{lemma2.2}, Lemma \ref{lemma2.3}, Corollary \ref{corrollary}, and Lemma \ref{lemma2.4} that
\begin{equation*}
\begin{aligned}
I_{4,1}\lesssim&\left\|\Lambda^{-s}\left(\langle v \rangle^{-\frac\gamma2}\{\Phi^{ij}\ast[\mu^{1/2}f]\}\partial_j f\right)\right\|\left\|\Lambda^{-s}\partial_i{\bf\{I-P\}}f\langle v \rangle^{\frac{\gamma}2}\right\|\\
\lesssim&\left\|\langle v \rangle^{-\frac\gamma2}\left\{\Phi^{ij}\ast\left[\mu^{1/2}f\right]\right\}\partial_j f\right\|_{L_x^{\frac{6}{3+2s}}}\left\|\Lambda^{-s}{\bf\{I-P\}}f\right\|_\sigma\\
\lesssim&\left\|\left|\mu^\delta f\right|^2_2\left|\langle v\rangle^{\frac\gamma2+2}\nabla_vf\right|^2_2\right\|_{L_x^{\frac{6}{3+2s}}}
+\varepsilon\left\|\Lambda^{-s}{\bf\{I-P\}}f\right\|^2_\sigma\\
\lesssim&\left\|\mu^\delta f\right\|^2_{L_x^{\frac3s}}\left\|\langle v\rangle^{\frac\gamma2+2}\nabla_vf\right\|^2+\varepsilon\left\|\Lambda^{-s}{\bf\{I-P\}}f\right\|^2_\sigma\\
\lesssim&\left\|\Lambda^{\frac32-s}\left(\mu^\delta f\right)\right\|^2\left\|\langle v\rangle^{\frac\gamma2+2}\nabla_vf\right\|^2+\varepsilon\left\|\Lambda^{-s}{\bf\{I-P\}}f\right\|^2_\sigma.
\end{aligned}
\end{equation*}
Similarly we can get that
\begin{equation*}
\begin{aligned}
\sum_{j=2}^4I_{4,j}\lesssim\left\|\Lambda^{\frac32-s}\left(\mu^\delta \nabla_vf\right)\right\|^2\left\|\langle v\rangle^{\frac\gamma2+2}f\right\|^2+\left\|\Lambda^{\frac32-s}\left(\mu^\delta f\right)\right\|^2\left\|\langle v\rangle^{\frac\gamma2+1}\nabla_vf\right\|^2+\varepsilon\left\|\Lambda^{-s}{\bf\{I-P\}}f\right\|^2_\sigma.
\end{aligned}
\end{equation*}
Consequently
$$
I_4\lesssim\left\|\Lambda^{\frac32-s}\left(\mu^\delta \nabla_vf\right)\right\|^2\left\|\langle v\rangle^{\frac\gamma2+2}f\right\|^2+\left\|\Lambda^{\frac32-s}\left(\mu^\delta f\right)\right\|^2\left\|\langle v\rangle^{\frac\gamma2+1}\nabla_vf\right\|^2+\varepsilon\left\|\Lambda^{-s}{\bf\{I-P\}}f\right\|^2_\sigma.
$$

Substituting the estimates on $I_j\ (j=1,2,3,4)$ into (\ref{4.6}) yields
\begin{equation*}
\begin{aligned}
\frac{d}{dt}\left(\left\|\Lambda^{-s}f\right\|^2+\left\|\Lambda^{-s}\nabla\phi\right\|^2\right)+&\left\|\Lambda^{-s}\{{\bf I-P}\}f\right\|_{\sigma}^2
\lesssim\left(\mathcal{\bar{E}}_{0,0}(t)\right)^{1/2}\left(\left\|\Lambda^{\frac34-\frac s2}(E,B)\right\|^2
+\left\|\Lambda^{\frac34-\frac s2}f\right\|^2_{\sigma}\right)\\
&\quad\quad+\mathcal{\bar{E}}_{1,\frac{\gamma-2}{2(\gamma+2)}}(t)\left(\left\|\Lambda^{\frac32-s}(E,B)\right\|^2
+\left\|\Lambda^{\frac32-s}f\right\|_\sigma^2+\left\|\Lambda^{\frac32-s}\nabla_v\{{\bf I-P}\}f\right\|_{\sigma}^2\right).
\end{aligned}
\end{equation*}
Having obtained the above inequality, (\ref{Lemma4.1-1}) follows immediately and the proof of Lemma $\ref{Lemma4.1}$ is complete.
\end{proof}

(\ref{Lemma4.1-1}) tells us that to deduce an estimate on $\|(f,E,B)(t)\|_{\dot{H}^{-s}}$, one has to estimate
$\left\|\Lambda^{\frac34-\frac s2}(f,E,B)\right\|$, $\left\|\Lambda^{\frac32-s}(f,E,B)\right\|_\sigma$, and $\left\|\Lambda^{\frac32-s}\nabla_v\{{\bf I-P}\}f\right\|_{\sigma}$ first. Due to the fact that the assumption $s\leq \frac 32$ implies $\frac 34-\frac s2\geq 0$ and $\frac 32-s\geq 0$ and noticing that the nonlinear energy method developed by \cite{Guo-IUMJ-04} and \cite{Liu-Yang-Yu} for the Boltzmann type equations can indeed yield a nice estimate on the microscopic component $\{{\bf I-P}\}f$ together with its integer-order derivatives with respect to $x$ and $v$, one can control $\left\|\Lambda^{\frac34-\frac s2}\{{\bf I-P}\}f\right\|$, $\left\|\Lambda^{\frac32-s}\{{\bf I-P}\}f\right\|_\sigma$, and $\left\|\Lambda^{\frac32-s}\nabla_v\{{\bf I-P}\}f\right\|_{\sigma}$ suitably. Thus to estimate $\left\|\Lambda^{\frac34-\frac s2}(f,E,B)\right\|$, $\left\|\Lambda^{\frac32-s}(f,E,B)\right\|_\sigma$, and $\left\|\Lambda^{\frac32-s}\nabla_v\{{\bf I-P}\}f\right\|_{\sigma}$ is reduced to estimate $\left\|\Lambda^{\frac34-\frac s2}({\bf P}f,E,B)\right\|$ and $\left\|\Lambda^{\frac32-s}({\bf P}f,E,B)\right\|$ suitably. To this end, if we assume further that $s\geq \frac 12$, then it is easy to see that $\frac 34-\frac s2\geq 1-s$, $\frac 32-s\geq 1-s$. Such a fact, the standard interpolation technique together with the fact that the the nonlinear energy method developed by \cite{Guo-IUMJ-04} and \cite{Liu-Yang-Yu} for the Boltzmann type equations can also yield an estimate on
$\nabla_x^k({\bf P}f, E, B)$ for positive integer $k$ imply that we only need to deduce certain estimates on $\|\Lambda^{1-s}{\bf P}f\|^2$ and $\|\Lambda^{1-s}(E,B)\|^2$. This is exactly what we want to do in the coming two lemmas. The first one is concerned with the macro dissipation estimate on $\|\Lambda^{1-s}{\bf P}f\|^2$.
\begin{lemma}\label{Lemma4.2}
Let $s\in [\frac12, \frac32)$, there exists an interactive functional $G^{-s}_{f}(t)$ satisfying
\begin{equation*}
G^{-s}_{f}(t)\lesssim \left\|\Lambda^{-s}f\right\|^2+\left\|\Lambda^{1-s}f\right\|^2
\end{equation*}
such that
\begin{equation}\label{Lemma4.2-1}
\begin{aligned}
\frac{d}{dt}G^{-s}_{f}(t)+\left\|\Lambda^{1-s}{\bf P}f\right\|^2+\left\|\Lambda^{-s}(a_+-a_-)\right\|^2
\lesssim\left\|\Lambda^{-s}\{{\bf I-P}\}f\right\|^2_\sigma+\left\|\Lambda^{1-s}\{{\bf I-P}\}f\right\|^2_\sigma+\mathcal{\bar{E}}_{1,0}(t)\mathcal{\bar{D}}_{2,0}(t)
\end{aligned}
\end{equation}
holds for any $0\leq t\leq T$.
\end{lemma}
\begin{proof}
Through Fourier transform, we have from (\ref{Macro-equation1}) and (\ref{Micro-equation1}) that
\begin{equation}\label{Macro-Fourier}
\begin{cases}
\partial_t\mathcal{F}\big[\frac{a_++a_-}2\big]+i\xi\cdot \hat{b}=0\\
\partial_t\hat{b}_i+i\xi_i\mathcal{F}[\frac{a_++a_-}2+2c\big]+\frac12\sum\limits_{j=1}^3i\xi_jA_{ij}\left(\{{\bf I-P}\}\hat{f}\cdot [1,1]\right)\\
\qquad\qquad=\mathcal{F}\big[E_i\frac{a_+-a_-}{2}\big]+\mathcal{F}\big[[G\times B]_i\big],\\
\partial_t\hat{c}+\frac13i\xi\cdot \hat{b}+\frac56\sum\limits_{i=1}^3i\xi_i B_i\left(\{{\bf I-P}\}\hat{f}\cdot [1,1]\right)=\frac16 \mathcal{F}\big[G\cdot E\big],
\end{cases}
\end{equation}
and
\begin{equation}\label{Micro-Fourier}
\begin{cases}
\frac12\partial_tA_{ij}\left(\{{\bf I-P}\}\hat{f}\cdot [1,1]\right) +i\xi_j\hat{b}_i+i\xi_i\hat{b}_j-\frac23\delta_{ij}i\xi\cdot \hat{b}\\
\qquad\qquad -\frac53\delta_{ij}i\xi\cdot B\left(\{{\bf I-P}\}\hat{f}\cdot [1,1]\right) =\frac12A_{ij}\left(\hat{r}_++\hat{r}_-+\hat{g}_++\hat{g}_-\right)-\frac13\delta_{ij}\mathcal{F}\big[G\cdot E\big],\\
\frac12\partial_t B_{j}\left(\{{\bf I-P}\}\hat{f}\cdot [1,1]\right) +i\xi_j\hat{c}=\frac12B_{j}\left(\hat{r}_++\hat{r}_-+\hat{g}_++\hat{g}_-\right).
\end{cases}
\end{equation}
Similarly, we can deduce from (\ref{a_+-a_--original}) that
\begin{equation}\label{a_+-a_-}
\begin{cases}
\partial_t(\hat{a}_+-\hat{a}_-)+i\xi\cdot \hat{G}=0,\\
\partial_t\hat{G}+i\xi(\hat{a}_+-\hat{a}_-)-2\hat{E}+i\xi\cdot A\left(\{{\bf I-P}\}\hat{f}\cdot q_1\right)\\
\qquad\qquad=\mathcal{F}[E(a_++a_-)]+\mathcal{F}[2b\times B]+\left\langle [v,-v]\mu^{1/2},L\hat{f}+\mathcal{F}[\Gamma(f,f)]\right\rangle,\\
i\xi\cdot \hat{E}=\hat{a}_+-\hat{a}_-.
\end{cases}
\end{equation}
Taking the $L^2$ inner product of $(\ref{Micro-Fourier})_2  $ with $i\xi_j|\xi|^{-2s}\bar{\hat{c}}_j$, the bound of $c$ follows from
\begin{equation}\label{4.18}
\begin{split}
&\left\|\Lambda^{1-s}c\right\|^2=\left\||\xi|^{1-s}\hat{c}\right\|^2=\sum_{j=1}^3\left(i\xi_j \hat{c}\ \Big|\  i\xi_j\hat{c}|\xi|^{-2s}\right)\\
\lesssim&-\sum_{i=1}^3\left(\partial_t B_{j}\left(\{{\bf I-P}\}\hat{f}\cdot [1,1]\right)\ \Big|\  i\xi_j\hat{c}|\xi|^{-2s}\right)+\sum_{i=1}^3\left(B_{j}\left(\hat{r}_++\hat{r}_-+\hat{g}_++\hat{g}_-\right)\ \Big|\  i\xi_j\hat{c}|\xi|^{-2s}\right)\\
=&-\frac{d}{dt}\left(B_{j}\left(\{{\bf I-P}\}\hat{f}\cdot [1,1]\right)\ \Big|\  i\xi_j\hat{c}|\xi|^{-2s}\right)
+\left(B_{j}\left(\{{\bf I-P}\}\hat{f}\cdot [1,1]\right)\ \Big|\  i\xi_j\partial_t \hat{c}|\xi|^{-2s}\right)\\
&+\sum_{i=1}^3\left(B_{j}\left(\hat{r}_++\hat{r}_-+\hat{g}_++\hat{g}_-\right)\ \Big|\  i\xi_j\hat{c}|\xi|^{-2s}\right).
\end{split}
\end{equation}
For the second term on the right hand of the second equality in \eqref{4.18}, we have from $(\ref{Macro-Fourier})_3$ that
\begin{equation*}
\begin{aligned}
&\left(B_{j}\left(\{{\bf I-P}\}\hat{f}\cdot [1,1]\right)\ \Big|\  i\xi_j\partial_t \hat{c}|\xi|^{-2s}\right)\\
=&\bigg(B_{j}\left(\{{\bf I-P}\}\hat{f}\cdot [1,1]\right)\ \Big|\  i\xi_j|\xi|^{-2s}\bigg\{\frac16 \mathcal{F}\big[G\cdot E\big]-\frac13i\xi\cdot \hat{b}-\frac56\sum_{j=1}^3i\xi_j B_j\left(\{{\bf I-P}\}\hat{f}\cdot [1,1]\right)\bigg\}\bigg)\\
\lesssim&\varepsilon\left\|\Lambda^{1-s}b\right\|^2+\left\|\Lambda^{-s}\{{\bf I-P}\}f\right\|_\sigma^2
+\left\|\Lambda^{1-s}\{{\bf I-P}\}f\right\|_\sigma^2+\mathcal{\bar{E}}_{1,0}(t)\mathcal{\bar{D}}_{2,0}(t),
\end{aligned}
\end{equation*}
while for the last term on the right hand of (\ref{4.18}), we can deduce that
\begin{equation*}
\begin{aligned}
&\left(B_{j}\left(\hat{r}_++\hat{r}_-+\hat{g}_++\hat{g}_-\right)\ \Big|\  i\xi_j\hat{c}|\xi|^{-2s}\right)\\
\lesssim&\varepsilon\left\|\Lambda^{1-s}c\right\|^2+\left\|\Lambda^{1-s}\{{\bf I-P}\}f\right\|_\sigma^2
+\left\|\Lambda^{-s}\{{\bf I-P}\}f\right\|^2_\sigma+\mathcal{\bar{E}}_{1,0}(t)\mathcal{\bar{D}}_{2,0}(t).
\end{aligned}
\end{equation*}
Thus we can get that
\begin{equation}\label{G-c}
\begin{aligned}
\frac{d}{dt}G^{-s}_c(t)+\left\|\Lambda^{1-s}c\right\|\lesssim\varepsilon\left\|\Lambda^{1-s}b\right\|^2
+\left\|\Lambda^{-s}\{{\bf I-P}\}f\right\|_\sigma^2+\left\|\Lambda^{1-s}\{{\bf I-P}\}f\right\|_\sigma^2+\mathcal{\bar{E}}_{1,0}(t)\mathcal{\bar{D}}_{2,0}(t).
\end{aligned}
\end{equation}
Here
\[
G^{-s}_c(t)\equiv\sum_{j=1}^3\left(B_{j}\left(\{{\bf I-P}\}\hat{f}\cdot [1,1]\right)\ \Big|\  i\xi_j\hat{c}|\xi|^{-2s}\right).
\]
On the other hand, noticing
\begin{eqnarray*}
&&\int_{\mathbb{R}^3}\left||\xi|^{-s}\left(\xi_j\hat{b}_i+\xi_i\hat{b}_j-\frac23\delta_{ij}\xi\cdot \hat{b}\right)\right|^2d\xi\\
&=&2\sum_{ij}\left\||\xi|^{-s}\xi_i b_j\right\|^2+\frac23\left\||\xi|^{-s}\xi\cdot b\right\|^2\\
&=&2\left\|\Lambda^{1-s}b\right\|^2+\frac23\left\|\Lambda^{-s}(\nabla_x\cdot b)\right\|^2,
\end{eqnarray*}
we can get by repeating the argument used to deduce (\ref{G-c}) and by using $(\ref{Micro-Fourier})_1$ and $(\ref{Macro-Fourier})_2$ that
\begin{equation}\label{G-b}
\begin{aligned}
\frac{d}{dt}G^{-s}_b(t)+\left\|\Lambda^{1-s}b\right\|^2\lesssim&\varepsilon\left(\left\|\Lambda^{1-s}(a_++a_-)\right\|^2
+\left\|\Lambda^{1-s}c\right\|^2\right)\\
&+\left\|\Lambda^{-s}\{{\bf I-P}\}f\right\|^2_\sigma+\left\|\Lambda^{1-s}\{{\bf I-P}\}f\right\|_\sigma^2+\mathcal{\bar{E}}_{1,0}(t)\mathcal{\bar{D}}_{2,0}(t).
\end{aligned}
\end{equation}
Next, we estimate $a_++a_-$. To this end, we have from $(\ref{Macro-Fourier})_2$, $(\ref{Macro-Fourier})_1$ and by employing the same argument to deduce $(\ref{G-c})$ that
\begin{equation}\label{G-a}
\begin{aligned}
\frac{d}{dt}G^{-s}_a(t)+\|\Lambda^{1-s}(a_++a_-)\|^2\lesssim\|\Lambda^{1-s}b\|^2+\|\Lambda^{1-s}c\|^2
+\|\Lambda^{1-s}\{{\bf I-P}\}f\|_\sigma^2+\mathcal{\bar{E}}_{1,0}(t)\mathcal{\bar{D}}_{2,0}(t)
\end{aligned}
\end{equation}
with
$$
G^{-s}_a(t)\equiv\sum_{i=1}^3\left(\hat{b}_i\ \Big|\  i\xi_i(\hat{a}_++\hat{a}_-)|\xi|^{-2s}\right).
$$

Set
\begin{equation*}
G^{-s}_{\bar{f}}(t)=G^{-s}_b(t)+\kappa_1G^{-s}_c(t)+\kappa_2G^{-s}_a(t), , 0<\kappa_2\ll\kappa_1\ll1.
\end{equation*}
we can deduce from (\ref{G-c}), (\ref{G-b}), and (\ref{G-a}) that
\begin{equation}\label{g-f-}
\begin{aligned}
\frac{d}{dt}G^{-s}_{\bar{f}}(t)+\left\|\Lambda^{1-s}(a_++a_-,b,c)\right\|^2\lesssim\left\|\Lambda^{-s}\{{\bf I-P}\}f\right\|^2_\sigma+\left\|\Lambda^{1-s}\{{\bf I-P}\}f\right\|^2_\sigma+\mathcal{\bar{E}}_{1,0}(t)\mathcal{\bar{D}}_{2,0}(t).
\end{aligned}
\end{equation}
Finally, for the corresponding estimate on $a_+-a_-$, we have from ($\ref{a_+-a_-}$) that
\begin{equation}
\begin{aligned}
&\left\|\Lambda^{1-s}(a_+-a_-)\right\|^2+2\left\|\Lambda^{-s}(a_+-a_-)\right\|^2
=\left(i\xi(\hat{a}_+-\hat{a}_-)-2\hat{E}\ \Big|\  i\xi(\hat{a}_+-\hat{a}_-)|\xi|^{-2s}\right)\\
=&\left(-\partial_t\hat{G}-i\xi\cdot A\left(\{{\bf I-P}\}\hat{f}\cdot q_1\right)+\mathcal{F}[E(a_++a_-)]\ \Big|\  i\xi(\hat{a}_+-\hat{a}_-)|\xi|^{-2s}\right)\\
&+\left(\mathcal{F}[2b\times B]+\left\langle [v,-v]\mu^{1/2},L\hat{f}+\mathcal{F}[\Gamma(f,f)]\right\rangle\ \Big|\  i\xi(\hat{a}_+-\hat{a}_-)|\xi|^{-2s}\right)\\
=&-\frac{d}{dt}\left(\hat{G}\ \Big|\  i\xi\left(\hat{a}_+-\hat{a}_-\right)|\xi|^{-2s}\right)+\left(\hat{G}\ \Big|\  \xi \xi\cdot\hat{G}|\xi|^{-2s}\right)\\
&+\left(-i\xi\cdot A\left(\{{\bf I-P}\}\hat{f}\cdot q_1\right)+\mathcal{F}[E(a_++a_-)]\ \Big|\  i\xi(\hat{a}_+-\hat{a}_-)|\xi|^{-2s}\right)\\
&+\left(\mathcal{F}[2b\times B]+\left\langle [v,-v]\mu^{1/2},L\hat{f}+\mathcal{F}[\Gamma(f,f)]\right\rangle\ \Big|\  i\xi(\hat{a}_+-\hat{a}_-)|\xi|^{-2s}\right)\\
\lesssim&-\frac{d}{dt}\left(\hat{G}\ \Big|\  i\xi(\hat{a}_+-\hat{a}_-)|\xi|^{-2s}\right)+\left\|\Lambda^{1-s}\{{\bf I-P}\}f\right\|^2_\sigma+\mathcal{\bar{E}}_{1,0}(t)\mathcal{\bar{D}}_{2,0}(t).
\end{aligned}
\end{equation}
Consequently
\begin{equation}\label{g-a}
\begin{aligned}
&\frac{d}{dt}\mathfrak{R}\left(\hat{G}\ \Big|\  i\xi(\hat{a}_+-\hat{a}_-)|\xi|^{-2s}\right)+\left\|\Lambda^{1-s}(a_+-a_-)\right\|^2
+\left\|\Lambda^{-s}(a_+-a_-)\right\|^2\\[2mm]
\lesssim&\left\|\Lambda^{1-s}\{{\bf I-P}\}f\right\|^2_\sigma+\mathcal{\bar{E}}_{1,0}(t)\mathcal{\bar{D}}_{2,0}(t).
\end{aligned}
\end{equation}
Here and in the rest of this manuscript $\mathfrak{R}$ is used to denote the real part of a complex function.

A suitable linear combination of (\ref{g-f-}) and (\ref{g-a}) gives
\begin{equation}
\begin{aligned}
&\frac{d}{dt}G^{-s}_{f}(t)+\left\|\Lambda^{1-s}\left(a_+\pm a_-,b,c\right)\right\|^2+\left\|\Lambda^{-s}(a_+-a_-)\right\|^2\\
\lesssim&\left\|\Lambda^{-s}\{{\bf I-P}\}f\right\|^2_\sigma+\left\|\Lambda^{1-s}\{{\bf I-P}\}f\right\|^2_\sigma+\mathcal{\bar{E}}_{1,0}(t)\mathcal{\bar{D}}_{2,0}(t).
\end{aligned}
\end{equation}
This is (\ref{Lemma4.2-1}) and the proof of Lemma \ref{Lemma4.2} is complete.
\end{proof}
For the estimate on $\|\Lambda^{1-s}(E,B)\|^2$, we have
\begin{lemma}\label{lemma3.3}
Let $s\in [\frac12, \frac32)$, there exists an interactive functional $G_{E,B}(t)$ satisfying
\begin{equation}\label{G_{E,B}}
G_{E,B}(t)\lesssim \left\|\Lambda^{1-s}(f,E,B)\right\|^2+\left\|\Lambda^{-s}(f,E,B)\right\|^2+\|\Lambda^{2-s}E\|^2
\end{equation}
such that
\begin{eqnarray}\label{EB-s}
&&\frac{d}{dt}G_{E,B}(t)
+\left\|\Lambda^{1-s}(E,B)\right\|^2+\left\|\Lambda^{-s}E\right\|^2+\left\|\Lambda^{-s}(a_+-a_-)\right\|^2_{H^1}\nonumber\\
&\lesssim& \left\|\Lambda^{-s}\{{\bf I-P}\}f\right\|^2_\sigma+\left\|\Lambda^{1-s}\{{\bf I-P}\}f\right\|^2_\sigma
+\left\|\Lambda^{2-s}\{{\bf I-P}\}f\right\|^2_\sigma+\mathcal{\bar{E}}_{1,0}(t)\mathcal{\bar{D}}_{2,0}(t)\nonumber
\end{eqnarray}
holds for any $0\leq t\leq T$.
\end{lemma}
\begin{proof}
Since
\begin{equation}
\begin{split}
2\hat{E}=&\partial_t\hat{G}+i\xi(\hat{a}_+-\hat{a}_-)+i\xi\cdot A\left(\{{\bf I-P}\}\hat{f}\cdot q_1\right) -\mathcal{F}[E(a_++a_-)]\\
&-\mathcal{F}[2b\times B]-\left\langle [v,-v]\mu^{1/2},L\hat{f}+\mathcal{F}[\Gamma(f,f)]\right\rangle,
\end{split}
\end{equation}
we can deduce that
\begin{equation}
\begin{aligned}
&\left\|\Lambda^{1-s}E\right\|^2=\left(\hat{E}\ \Big|\  \hat{E}|\xi|^{2-2s}\right)\\
\lesssim&\left(\partial_t\hat{G}+i\xi(\hat{a}_+-\hat{a}_-)+i\xi\cdot A\left(\{{\bf I-P}\}\hat{f}\cdot q_1\right)\ \Big|\  \hat{E}|\xi|^{2-2s}\right)\\
&-\left(\mathcal{F}[E(a_++a_-)]+\mathcal{F}[2b\times B]+\left\langle [v,-v]\mu^{1/2},L\hat{f}+\mathcal{F}[\Gamma(f,f)]\right\rangle\ \Big|\  \hat{E}|\xi|^{2-2s}\right)\\
\lesssim&\frac{d}{dt}\left(\hat{G}\ \Big|\  \hat{E}|\xi|^{2-2s}\right)-\left(\hat{G}\ \Big|\  \partial_t\hat{E}|\xi|^{2-2s}\right)+\varepsilon\left\|\Lambda^{1-s}E\right\|^2-\left\|\Lambda^{1-s}(a_+-a_-)\right\|^2\\
&+\left\|\Lambda^{2-s}\{{\bf I-P}\}f\right\|^2_\sigma+\left\|\Lambda^{-s}\{{\bf I-P}\}f\right\|^2_\sigma+\left\|\Lambda^{1-s}\{{\bf I-P}\}f\right\|^2_\sigma+\mathcal{\bar{E}}_{1,0}(t)\mathcal{\bar{D}}_{2,0}(t)\\
\lesssim&\frac{d}{dt}\left(\hat{G}\ \Big|\  \hat{E}|\xi|^{2-2s}\right)-\left(\hat{G}\ \Big|\  (i\xi\times\hat{B}-\hat{G})|\xi|^{2-2s}\right)+\varepsilon\left\|\Lambda^{1-s}E\right\|^2
-\left\|\Lambda^{1-s}(a_+-a_-)\right\|^2\\
&+\left\|\Lambda^{2-s}\{{\bf I-P}\}f\right\|^2_\sigma+\left\|\Lambda^{-s}\{{\bf I-P}\}f\right\|^2_\sigma+\left\|\Lambda^{1-s}\{{\bf I-P}\}f\right\|^2_\sigma+\mathcal{\bar{E}}_{1,0}(t)\mathcal{\bar{D}}_{2,0}(t)\\
\lesssim&\frac{d}{dt}\left(\hat{G}\ \Big|\  \hat{E}|\xi|^{2-2s}\right)+\varepsilon\left\|\Lambda^{1-s}B\right\|^2
-\left\|\Lambda^{1-s}(a_+-a_-)\right\|^2+\varepsilon\left\|\Lambda^{1-s}E\right\|^2\\
&+\left\|\Lambda^{2-s}\{{\bf I-P}\}f\right\|^2_\sigma+\left\|\Lambda^{-s}\{{\bf I-P}\}f\right\|^2_\sigma+\left\|\Lambda^{1-s}\{{\bf I-P}\}f\right\|^2_\sigma+\mathcal{\bar{E}}_{1,0}(t)\mathcal{\bar{D}}_{2,0}(t).
\end{aligned}
\end{equation}
Consequently
\begin{eqnarray}\label{E-1-s}
&&-\frac{d}{dt}\mathfrak{R}\left(\hat{G}\ \Big|\  \hat{E}|\xi|^{2-2s}\right)+\left\|\Lambda^{1-s}E\right\|^2+\left\|\Lambda^{1-s}(a_+-a_-)\right\|^2\nonumber\\
&\lesssim&\varepsilon\left\|\Lambda^{1-s}B\right\|^2
+\left\|\Lambda^{2-s}\{{\bf I-P}\}f\right\|^2_\sigma+\left\|\Lambda^{-s}\{{\bf I-P}\}f\right\|^2_\sigma\\
&&+\left\|\Lambda^{1-s}\{{\bf I-P}\}f\right\|^2_\sigma+\mathcal{\bar{E}}_{1,0}(t)\mathcal{\bar{D}}_{2,0}(t).\nonumber
\end{eqnarray}
Similarly, it holds that
\begin{eqnarray}\label{E-s}
&&-\frac{d}{dt}\mathfrak{R}\left(\hat{G}\ \Big|\  \hat{E}|\xi|^{-2s}\right)+\left\|\Lambda^{-s}E\right\|^2+\left\|\Lambda^{-s}(a_+-a_-)\right\|^2\nonumber\\
&\lesssim&\varepsilon\left\|\Lambda^{1-s}B\right\|^2
+\left\|\Lambda^{-s}\{{\bf I-P}\}f\right\|^2_\sigma\\
&&+\left\|\Lambda^{1-s}\{{\bf I-P}\}f\right\|^2_\sigma+\mathcal{\bar{E}}_{1,0}(t)\mathcal{\bar{D}}_{2,0}(t).\nonumber
\end{eqnarray}
For $B$, it follows that
\begin{eqnarray}
\left\|\Lambda^{1-s}B\right\|^2&=&\left(\hat{B}\ \Big|\  \hat{B}|\xi|^{2-2s}\right)\nonumber\\
&=&\left(i\xi\times\hat{B}\ \Big|\  i\xi\times\hat{B}|\xi|^{-2s}\right)\nonumber\\
&=&\left(\partial_t\hat{E}+\hat{G}\ \Big|\  i\xi\times\hat{B}|\xi|^{-2s}\right)\nonumber\\
&=&\frac{d}{dt}\left(\hat{E}\ \Big|\  i\xi\times\hat{B}|\xi|^{-2s}\right)-\left(\hat{E}\ \Big|\  i\xi\times\partial_t\hat{B}|\xi|^{-2s}\right)
+\left(\hat{G}\ \Big|\  i\xi\times\hat{B}|\xi|^{-2s}\right)\\
&=&\frac{d}{dt}\left(\hat{E}\ \Big|\  i\xi\times\hat{B}|\xi|^{-2s}\right)-\left(\hat{E}\ \Big|\  \xi\times(\xi\times\hat{E})|\xi|^{-2s}\right)
+\left(\hat{G}\ \Big|\  i\xi\times\hat{B}|\xi|^{-2s}\right)\nonumber\\
& \lesssim&\frac{d}{dt}\left(\hat{E}\ \Big|\  i\xi\times\hat{B}|\xi|^{-2s}\right)+\left\|\Lambda^{1-s}E\right\|^2
 +\varepsilon\left\|\Lambda^{1-s}B\right\|^2+\left\|\Lambda^{-s}\{{{\bf I-P}}\}f\right\|^2_\sigma.\nonumber
\end{eqnarray}
That is
\begin{equation}\label{B-s}
\begin{aligned}
-\frac{d}{dt}\mathfrak{R}\left(\hat{E}\ \Big|\  i\xi\times\hat{B}|\xi|^{-2s}\right)+\left\|\Lambda^{1-s}B\right\|^2
\lesssim\left\|\Lambda^{1-s}E\right\|^2+\varepsilon\left\|\Lambda^{1-s}B\right\|^2+\left\|\Lambda^{-s}\{{\bf I-P}\}f\right\|^2_\sigma.
\end{aligned}
\end{equation}
For sufficiently small $\kappa>0$, (\ref{E-1-s})$+$(\ref{E-s})$+\kappa$(\ref{B-s}) gives
\begin{equation}\label{EB-s-2}
\begin{aligned}
&\frac{d}{dt}\mathfrak{R}G_{E,B}(t)+\left\|\Lambda^{1-s}(E,B)\right\|^2+\left\|\Lambda^{-s}(a_+-a_-)\right\|^2_{H^1}\\
\lesssim&\left\|\Lambda^{-s}\{{\bf I-P}\}f\right\|^2_\sigma
+\left\|\Lambda^{1-s}\{{\bf I-P}\}f\right\|^2_\sigma+\left\|\Lambda^{2-s}\{{\bf I-P}\}f\right\|^2_\sigma+\mathcal{\bar{E}}_{1,0}(t)\mathcal{\bar{D}}_{2,0}(t).
\end{aligned}
\end{equation}
Here we set
\[
G_{E,B}(t)=-\left\{\left(\hat{G}\ \Big|\  \hat{E}|\xi|^{2-2s}\right)+\left(\hat{G}\ \Big|\  \hat{E}|\xi|^{2s}\right)+\kappa\left(\hat{E}\ \Big|\  i\xi\times\hat{B}|\xi|^{-2s}\right)\right\},
\]
one can deduce (\ref{EB-s}) from (\ref{EB-s-2}) and it is easy to see that $G_{E,B}(t)$ satisfies (\ref{G_{E,B}}). This completes the proof of Lemma \ref{lemma3.3}.
\end{proof}
To obtain the temporal decay estimates on the energy functional $\mathcal{E}^{k}_{N_0}(t)$,  our trick is trying to deduce an inequality of the following type
\begin{equation*}\label{Lemma4.5-1}
\frac{d}{dt}\mathcal{E}^k_{N_0}(t)+\mathcal{D}^k_{N_0}(t)\leq 0
\end{equation*}
for all $0\leq t\leq T$. Notice that although the energy functional $\mathcal{E}^k_{N_0}(t)$ contains only the derivative of the solutions of the VML system \eqref{f}, the dissipative estimate \eqref{coercive} of the linearized operator $L$ implies that the corresponding energy dissipation rate functional $\mathcal{D}^k_{N_0}(t)$ does contain the first order derivative of the microscopic component of the solutions of the VML system \eqref{f} with respect to $v$, which can be used to control the nonlinear terms related to the Lorenz force suitably. These facts will be presented in the following two lemmas:
\begin{lemma}\label{lemma4.3} Let $N_0\geq 3$, then there exists a positive constant $\hat{l}$ which depends only on $N_0$ and $\gamma$ such that the following estimates hold:
\begin{itemize}
\item[(i).] For $k=0,1,\cdots,N_0-2$, it holds that
\begin{equation}\label{Lemma4.3-1}
\begin{aligned}
&\frac{d}{dt}\left(\left\|\nabla^kf\right\|^2+\left\|\nabla^k(E,B)\right\|^2\right)+\left\|\nabla^k\{{\bf I-P}\}f\right\|^2_{\sigma}\\
\lesssim&\mathcal{\bar{E}}_{N_0,l'}(t)\left(\left\|\nabla^{k+1}(E,B)\right\|^2+\left\|\nabla^{k+1}f\right\|^2_\sigma\right)+\varepsilon\left\|\nabla^{k+1}f\right\|^2_\sigma.
\end{aligned}
\end{equation}
\item [(ii).] If $k=N_0-1$, it holds that
\begin{equation}\label{Lemma4.3-2}
\begin{aligned}
&\frac{d}{dt}\left(\left\|\nabla^{N_0-1}f\right\|^2+\left\|\nabla^{N_0-1}(E,B)\right\|^2\right)+\left\|\nabla^{N_0-1}\{{\bf I-P}\}f\right\|_\sigma^2\\
\lesssim&\mathcal{\bar{E}}_{N_0,l'}(t)\left(\left\|\nabla^{N_0-1}(E,B)\right\|^2+\left\|\nabla^{N_0-1}f\right\|^2_\sigma\right)
+\varepsilon\left\|\nabla^{N_0-1}f\right\|_\sigma^2.
\end{aligned}
\end{equation}
\item[(iii).] For $k=N_0\geq 3$, it follows that
\begin{equation}\label{Lemma4.3-3}
\begin{aligned}
&\frac{d}{dt}\left(\left\|\nabla^{N_0}f\right\|^2+\left\|\nabla^{N_0}(E,B)\right\|^2\right)+\left\|\nabla^{N_0}\{{\bf I-P}\}f\right\|^2_{\sigma}\\
\lesssim&\max\left\{\mathcal{E}_{N}(t),\mathcal{\bar{E}}_{N_0,l'}(t)\right\}\left(\left\|\nabla^{N_0-1}(E,B)\right\|^2
+\left\|\nabla^{N_0-1}f\right\|_\sigma^2+\left\|\nabla^{N_0}f\right\|_\sigma^2\right)+\varepsilon\left\|\nabla^{N_0}f\right\|_\sigma^2.
\end{aligned}
\end{equation}
holds provided that $N>\frac{5(N_0-1)}{3}$.
Here we choose $l'\geq \max\left\{\frac{\hat{l}}{-(\gamma+2)},\frac{\gamma-2}{2(\gamma+2)}\right\}$ with $\hat{l}=\max\{\hat{l}_1,\hat{l}_2,\cdots,\hat{l}_{11}\}$, where $\hat{l}_i(i=1,\cdots,11)$ will be specified in the  proof of this lemma.
\end{itemize}
\end{lemma}
\begin{proof} For the case $k=0$, we have by multiplying $(\ref{f})$ by $f$ and integrating the resulting identity with respect to $x$ and $v$ over $\mathbb{R}_x^3\times\mathbb{R}_v^3$ that
\begin{equation}\label{k-0}
\begin{aligned}
\frac{d}{dt}\left(\|f\|^2+\|(E,B)\|^2\right)+\|\{{\bf I-P}\}f\|^2_{\sigma}\lesssim|(v\cdot E f,f)|+|({\Gamma}(f,f),\{{\bf I-P}\}f)|
\end{aligned}
\end{equation}
For the first term on the right hand of (\ref{k-0}), we obtain from Lemmas \ref{lemma2.2} to \ref{lemma2.4} and Corollary \ref{corrollary} that
\begin{equation*}
\begin{aligned}
|(v\cdot E f,f)|\lesssim&|(v\cdot E f,{\bf P}f)|+|(v\cdot E {\bf P}f,\{{\bf I-P}\}f)|+|(v\cdot E \{{\bf I-P}\}f,\{{\bf I-P}\}f)|\\
\lesssim&\|E\|_{L^2_x}\left\|\mu^{\delta}f\right\|_{L^3_xL^2_v}\left\|\mu^{\delta}f\right\|_{L^6_xL^2_v}
+\|E\|_{L^6_x}\left\|\langle v\rangle^{-\frac{\gamma}2}\{{\bf I-P}\}f\right\|_{L^3_xL^2_v}
\left\|\langle v\rangle^{\frac{\gamma+2}2}\{{\bf I-P}\}f\right\|_{L^2_xL^2_v}\\
\lesssim&\|E\|\left\|\mu^{\delta}f\right\|_{L^2_vL^3_x}\left\|\mu^{\delta}f\right\|_{L^2_vL^6_x}
+\|E\|_{L^2_vL^6_x}\left\|\langle v\rangle^{-\frac{\gamma}2}\{{\bf I-P}\}f\right\|_{L^2_vL^3_x}
\left\|\{{\bf I-P}\}f\right\|_\sigma\\
\lesssim&\left\|\Lambda^{-\frac12}E\right\|^{\frac 23}\left\|\nabla_xE\right\|^{\frac 13}
\left\|\Lambda^{-\frac12}f\right\|^{\frac 13}\left\|\nabla_xf\right\|_\sigma^{\frac 23} \left\|\nabla_xf\right\|_{\sigma}+\|\nabla_xE\|\left\|\langle v\rangle^{-\frac{\gamma}2}\{{\bf I-P}\}f\right\|_{H^1_x}
\|\{{\bf I-P}\}f\|_\sigma\\
\lesssim&\mathcal{\bar{E}}_{1,\frac{\gamma}{\gamma+2}}(t)\left(\left\|\nabla_xE\right\|^2+\left\|\nabla_xf\right\|^2_{\sigma}\right)
+\varepsilon\left(\left\|\nabla_xf\right\|^2_{\sigma}+\left\|\{{\bf I-P}\}f\right\|^2_{\sigma}\right).
\end{aligned}
\end{equation*}
The other term can be estimated as follows
\begin{equation*}
\begin{aligned}
|({\bf\Gamma}(f,f),\{{\bf I-P}\}f)|\lesssim&\varepsilon\|\{{\bf I-P}\}f\|^2_{\sigma}
+\left\|\left|\mu^\delta f\right|_{L^2_v}|f|_{L^2_\sigma}\right\|^2\\
\lesssim&\varepsilon\|\{{\bf I-P}\}f\|^2_{\sigma}+\left\|\mu^\delta f\right\|^2_{L^3_x}\|f\|^2_{L^6_\sigma}\\
\lesssim&\varepsilon\|\{{\bf I-P}\}f\|^2_{\sigma}+\left\|\mu^\delta f\right\|^2_{H^1_x}\|\nabla_xf\|^2_\sigma.
\end{aligned}
\end{equation*}
Collecting the above estimates gives
\begin{equation}\label{0}
\begin{aligned}
\frac{d}{dt}\left(\|f\|^2+\|(E,B)\|^2\right)+\|\{{\bf I-P}\}f\|^2_{\sigma}\lesssim&\mathcal{\bar{E}}_{1,\frac{\gamma}{\gamma+2}}(t)\left(\left\|\nabla_xE\right\|^2
+\left\|\nabla_xf\right\|^2_{\sigma}\right)+\varepsilon\|\nabla_xf\|^2_{\sigma},
\end{aligned}
\end{equation}
which gives (\ref{Lemma4.3-1}) with $k=0$.

For $k=1,2,\cdots,N-2$, we have by applying $\nabla^k$ to $(\ref{f})$,
multiplying the resulting identity by $\nabla^kf$, and then integrating the final result with respect to $x$ and $v$ over $\mathbb{R}_x^3\times\mathbb{R}_v^3$ that
\begin{equation}\label{k}
\begin{aligned}
\frac{d}{dt}&\left(\left\|\nabla^kf\right\|^2+\left\|\nabla^k(E,B)\right\|^2\right)+\left\|\nabla^k\{{\bf I-P}\}f\right\|^2_{\sigma}\lesssim\underbrace{\left|\left(\nabla^k(v\cdot Ef),\nabla^kf\right)\right|}_{I_5}\\
&+\underbrace{\left|\left(\nabla^k(E\cdot\nabla_vf),\nabla^kf\right)\right|}_{I_6}
+\underbrace{\left|\left(\nabla^k\left((v\times B)\cdot\nabla_vf\right),\nabla^kf\right)\right|}_{I_7}
+\underbrace{\left|\left(\nabla^k{\bf\Gamma}(f,f),\nabla^kf\right)\right|}_{I_8}.
\end{aligned}
\end{equation}
By applying the macro-micro decomposition \eqref{macro-micro}, one can deduce that
\begin{eqnarray}\label{k-E}
I_5&\lesssim&\underbrace{\left|\left(\nabla^k(v\cdot Ef),\nabla^k{\bf P}f\right)\right|}_{I_{5,1}}
+\underbrace{\left|\left(\nabla^k(v\cdot E{\bf P}f),\nabla^k{\{\bf I-P\}}f\right)\right|}_{I_{5,2}}\\
&&+\underbrace{\left|\left(\nabla^k(v\cdot E\{{\bf I-P}\}f),\nabla^k{\{\bf I-P\}}f\right)\right|}_{I_{5,3}}.\nonumber
\end{eqnarray}
Applying Lemma \ref{lemma2.2} and Corollary \ref{corrollary}, $I_{5,1}$ can be dominated by
\begin{equation*}
\begin{aligned}
I_{5,1}
\lesssim&\sum_{j\leq k-1}\left\|\left(\nabla^jE \nabla^{k-1-j}\left(\mu^{\delta}f\right)\right)\right\|\left\|\nabla^{k+1}\left(\mu^{\delta}f\right)\right\|\\
\lesssim&\sum_{j\leq k-1}\left\|\nabla^jE\right\|_{L^6_x}
\left\|\nabla^{k-1-j}\left(\mu^{\delta}f\right)\right\|_{L^3_xL^2_v}\left\|\nabla^{k+1}\left(\mu^{\delta}f\right)\right\|\\
\lesssim&\sum_{j\leq k-1}\left\|\Lambda^{-\frac 12}E\right\|^{\frac{2k-2j}{2k+3}} \left\|\nabla^{k+1}E\right\|^{\frac{2j+3}{2k+3}}
\left\|\Lambda^{-\frac 12}\left(\mu^{\delta}f\right)\right\|^{\frac{2j+3}{2k+3}} \left\|\nabla^{k+1}\left(\mu^{\delta}f\right)\right\|^{\frac{2k-2j}{2k+3}}
\left\|\nabla^{k+1}\left(\mu^{\delta}f\right)\right\|\\
\lesssim&\mathcal{\bar{E}}_{0,0}(t)\left(\left\|\nabla^{k+1}E\right\|^2+\left\|\nabla^{k+1}f\right\|_\sigma^2\right)
+\varepsilon\left\|\nabla^{k+1}f\right\|_\sigma^2.
\end{aligned}
\end{equation*}
Similarly, for $I_{5,2}$, we can get that
\[
I_{5,2}\lesssim
\mathcal{\bar{E}}_{0,0}(t)\left(\left\|\nabla^{k+1}E\right\|^2+\left\|\nabla^{k+1}f\right\|_\sigma^2\right)
+\varepsilon\left\|\nabla^{k+1}f\right\|_\sigma^2.
\]
For $I_{5,3}$, it holds from the Cauchy inequality and the Holder inequality that
\begin{equation*}
\begin{aligned}
I_{5,3}
\lesssim&\sum_{j\leq k}\left|\left(\nabla^jE v \nabla^{k-j}\{{\bf I-P}\}f,\nabla^k{\{\bf I-P\}}f\right)\right|\\
\lesssim&\sum_{j\leq k}\left\|\nabla^jE  \nabla^{k-j}\{{\bf I-P}\}f\langle v\rangle^{-\frac \gamma2}\right\|
\left\|\nabla^k\{{\bf I-P}\}f\langle v\rangle^{\frac {\gamma+2}{2}}\right\|\\
\lesssim&\underbrace{\sum_{j\leq k}\left\|\nabla^jE  \nabla^{k-j}\{{\bf I-P}\}f\langle v\rangle^{-\frac \gamma2}\right\|^2}_{I_{5,3,1}}+\varepsilon\left\|\nabla^k\{{\bf I-P}\}f\right\|^2_\sigma.
\end{aligned}
\end{equation*}
To deal with $I_{5,3,1}$, when $j=0$, we have from Lemma \ref{lemma2.2} and Corollary \ref{corrollary} that
\begin{equation*}
\begin{aligned}
 I_{5,3,1}\lesssim&\left(\|E\|_{L^\infty_x}\left\|\nabla^{k}\{{\bf I-P}\}f\langle v\rangle^{-\frac \gamma2}\right\|\right)^2\\
 \lesssim&\left(\left\|\Lambda^{-s}E\right\|^{1-\theta}\left\|\nabla^{k+1}E\right\|^{\theta}
 \left\|\nabla^{k}\{{\bf I-P}\}f\langle v\rangle^{\frac {\gamma+2}2}\right\|^{1-\theta}
 \left\|\nabla^{k}\{{\bf I-P}\}f\langle v\rangle^{\hat{l}_1}\right\|^{\theta}\right)^2\\
 \lesssim& \mathcal{\bar{E}}_{k,-\frac{\hat{l}_1}{\gamma+2}}(t)\left(\left\|\nabla^{k+1}E\right\|^2
 +\left\|\nabla^{k}\{{\bf I-P}\}f\right\|^2_\sigma\right).
\end{aligned}
\end{equation*}
Here $\theta=\frac{3+2s}{2(k+1+s)}$ and
$\hat{l}_1=-\frac{\gamma+1}{\theta}+\frac{\gamma+2}{2}=-\frac{2(\gamma+1)(k+1+s)}{3+2s}+\frac{\gamma+2}{2}.$
While for the case $j\neq0$, we can deduce from Lemma \ref{lemma2.2} and Corollary \ref{corrollary} that
\begin{equation*}
\begin{aligned}
 I_{5,3,1}\lesssim&\sum_{1\leq j\leq k}\left(\left\|\nabla^jE\right\|_{L^6_x}\left\|\nabla^{k-j}\{{\bf I-P}\}f\langle v\rangle^{-\frac \gamma2}\right\|_{L^2_vL^3_x}\right)^2\\
 \lesssim&\sum_{1\leq j\leq k}\left(\left\|\Lambda^{-s}E\right\|^{1-\theta}\left\|\nabla^{k+1}E\right\|^{\theta}
 \left\|\{{\bf I-P}\}f\langle v\rangle^{-\frac \gamma2}\right\|^{1-\alpha}\left\|\nabla^{k}\{{\bf I-P}\}f\langle v\rangle^{-\frac \gamma2}\right\|^{\alpha}\right)^2\\
  \lesssim&\sum_{1\leq j\leq k}\left(\left\|\Lambda^{-s}E\right\|^{1-\theta}\left\|\nabla^{k+1}E\right\|^{\theta}
 \left\|\{{\bf I-P}\}f\langle v\rangle^{-\frac \gamma2}\right\|^{1-\alpha}\right.\\
 &\quad\times\left.\left(\left\|\nabla^{k}\{{\bf I-P}\}f\langle v\rangle^{\hat{l}_2}\right\|^{1-\beta}\left\|\nabla^{k}\{{\bf I-P}\}f\langle v\rangle^{\frac {\gamma+2}2}\right\|^\beta\right)^{\alpha}\right)^2\\[2mm]
  \lesssim&\sum_{1\leq j\leq k}\left(\left\|\Lambda^{-s}E\right\|^{1-\theta}
 \left\|\{{\bf I-P}\}f\langle v\rangle^{-\frac \gamma2}\right\|^{1-\alpha}\left\|\nabla^{k}\{{\bf I-P}\}f\langle v\rangle^{\hat{l}_2}\right\|^{\alpha(1-\beta)}\right.\\
 &\quad\left.\times\left\|\nabla^{k}\{{\bf I-P}\}f\langle v\rangle^{\frac {\gamma+2}2}\right\|^{{\alpha}\beta}\left\|\nabla^{k+1}E\right\|^{\theta}\right)^2\\
 \lesssim& \mathcal{\bar{E}}_{k,-\frac{\hat{l}_2}{\gamma+2}}(t)\left(\left\|\nabla^{k+1}E\right\|^2+\left\|\nabla^{k}\{{\bf I-P}\}f\right\|^2_\sigma\right).
\end{aligned}
\end{equation*}
Here we can deduce from Lemma \ref{lemma2.2} that $\theta=\frac{j+s+1}{k+1+s},\ \alpha=\frac{2k-2j+1}{2k},\ \frac{\gamma+2}2\cdot\beta+\hat{l}_2(1-\beta)=-\frac\gamma 2.$
If we set $\alpha\beta=1-\theta$,  then we can get that $\hat{l}_2=\frac{-\frac\gamma 2-\frac{\gamma+2}2\cdot\beta}{1-\beta}=\frac{\gamma+2}{2}-\frac{\gamma+1}{1-\beta}=\frac{\gamma+2}{2}-\frac{(\gamma+1)(2k-2j+1)(k+1+s)}{(2k-2j+1)(k+1+s)-2k(k-j)} $ with $1\leq j\leq k$.

Consequently
$$
I_5\lesssim\mathcal{\bar{E}}_{k,-\frac{\max\{\hat{l}_1,\hat{l}_2\}}{\gamma+2}}(t)\left(\left\|\nabla^{k+1}E\right\|^2+\left\|\nabla^{k+1}f\right\|_\sigma^2
+\left\|\nabla^{k}\{{\bf I-P}\}f\right\|^2_\sigma\right)+ \varepsilon\left(\left\|\nabla^{k+1}f\right\|_\sigma^2+\left\|\nabla^k\{{\bf I-P}\}f\right\|^2_\sigma\right).
$$
For the terms $I_6$ and $I_7$ on the right-hand side of (\ref{k}), by paying particular attention to special structure of the dissipation property \eqref{coercive} of the linearized operator $L$, we can deduce by repeating the argument used above to get that
\begin{equation*}
\begin{aligned}
I_6+I_7\lesssim\mathcal{\bar{E}}_{k,-\frac{\hat{l}_3}{\gamma+2}}(t)\left(\left\|\nabla^{k+1}(E,B)\right\|^2
+\left\|\nabla^{k+1}f\right\|_\sigma^2
+\left\|\nabla^{k}\{{\bf I-P}\}f\right\|^2_\sigma\right)+ \varepsilon\left(\left\|\nabla^{k+1}f\right\|_\sigma^2+\left\|\nabla^k\{{\bf I-P}\}f\right\|^2_\sigma\right).
\end{aligned}
\end{equation*}
Here $\hat{l}_3=\frac{\gamma+2}{2}-\frac{\gamma}{1-\beta}=\frac{\gamma+2}{2}-\frac{\gamma(2k-2j+1)(k+1+s)}{(2k-2j+1)(k+1+s)-2k(k-j)}$
with $1 \leq j\leq k$.

For the last term on the right-hand side of (\ref{k}), it follows from Lemma \ref{lemma2.2} that
\begin{equation}\label{Gamma-k}
\begin{aligned}
I_8
\lesssim& \sum_{j\leq k}\left\|\left|\nabla^j\mu^\delta f\right|_{L^2_v}\left|\nabla^{k-j}f\right|_{L^2_\sigma}\right\|\left\|\nabla^k\{{\bf I-P}\}f\right\|_\sigma\\
\lesssim&\sum_{j\leq k}\left\|\mu^\delta\nabla^j f\right\|_{L^3_xL^2_v}\left\|\nabla^{k-j}f\right\|_{L^6_xL^2_\sigma} \left\|\nabla^k\{{\bf I-P}\}f\right\|_\sigma\\
\lesssim&\sum_{j\leq k}\left\|\mu^\delta\nabla^j f\right\|_{L^2_vL^3_x}\left\|\nabla^{k-j}f\right\|_{L^2_\sigma L^6_x}\left\|\nabla^k\{{\bf I-P}\}f\right\|_\sigma\\
\lesssim&\sum_{j\leq k}\left\|\mu^\delta \nabla^mf\right\|^{\frac{j+1}{k+1}}\left\|\mu^\delta\nabla^{k+1} f\right\|^{1-\frac{j+1}{k+1}}\| f\|_\sigma^{1-\frac{j+1}{k+1}}\left\| \nabla^{k+1}f\right\|_\sigma^{\frac{j+1}{k+1}}\left\|\nabla^k\{{\bf I-P}\}f\right\|_\sigma\\
\lesssim&\max\left\{\mathcal{E}_{N_0-2}(t), \mathcal{E}_{1,1}(t)\right\} \left\|\nabla^{k+1} f\right\|_\sigma^2 +\varepsilon\left\|\nabla^k\{{\bf I-P}\}f\right\|_\sigma^2.
\end{aligned}
\end{equation}
Here $m=\frac{k+1}{2(j+1)}\leq\frac{k+1}{2}\leq \frac N2.$

Thus collecting the above estimates on $I_5$, $I_6$, $I_7$, and $I_8$, one can deduce (\ref{Lemma4.3-1}) immediately from (\ref{k}) and \eqref{0}.

When $k=N_0-1\geq 2$, (\ref{k}) tells us that
\begin{equation}\label{N_0-1}
\begin{aligned}
&\frac{d}{dt}\left(\left\|\nabla^{N_0-1}f\right\|^2+\left\|\nabla^{N_0-1}(E,B)\right\|^2\right)
+\left\|\nabla^{N_0-1}\{{\bf I-P}\}f\right\|^2_{\sigma}\\
\lesssim&\underbrace{\left|\left(\nabla^{N_0-1}(E\cdot vf),\nabla^{N_0-1}f\right)\right|}_{I_9}
+\underbrace{\left|\left(\nabla^{N_0-1}(E\cdot \nabla_vf),\nabla^{N_0-1}f\right)\right|}_{I_{10}}\\
&+\underbrace{\left|\left(\nabla^{N_0-1}(v\times B\cdot \nabla_vf),\nabla^{N_0-1}f\right)\right|}_{I_{11}}+\underbrace{\left|\left(\nabla^{N_0-1}{\Gamma}(f,f),\nabla^{N_0-1}f\right)\right|}_{I_{12}}.
\end{aligned}
\end{equation}
To estimate $I_9$, due to
\begin{eqnarray}\label{N_0-1-e}
I_9&\lesssim&\underbrace{\left|\left(E\cdot v\nabla^{N_0-1}f,\nabla^{N_0-1}f\right)\right|}_{I_{9,1}}
+\underbrace{\left|\left(\nabla^{N_0-1} E\cdot vf,\nabla^{N_0-1}f\right)\right|}_{I_{9,2}}\\
&&+\underbrace{\sum_{1\leq j\leq N_0-2}\left|\left(\nabla^{j} E\cdot v\nabla ^{N_0-1-j}f,\nabla^{N_0-1}f\right)\right|}_{I_{9,3}},\nonumber
\end{eqnarray}
we can get from Lemma \ref{lemma2.2} and Corollary \ref{corrollary} that $I_{9,1}$ can be estimated as follows
\begin{equation*}
\begin{aligned}
I_{9,1}
\lesssim&\|E\|_{L^\infty}\left\|\nabla^{N_0-1}f\langle v\rangle^{-\frac\gamma2}\right\|\left\|\nabla^{N_0-1}f\right\|_\sigma\\
\lesssim&\left\|\Lambda^{-s}E\right\|^{1-\theta}\left\|\nabla^{N_0-1}E\right\|^{\theta}\left\|\nabla^{N_0-1}f\langle v\rangle^{\hat{l}_4}\right\|^{\theta}\left\|\nabla^{N_0-1}f\langle v\rangle^{\frac{\gamma+2}{2}}\right\|^{1-\theta}\left\|\nabla^{N_0-1}f\right\|_\sigma\\
\lesssim&\mathcal{\bar{E}}_{N_0-1,-\frac{\hat{l}_4}{\gamma+2}}(t)\left(\left\|\nabla^{N_0-1}E\right\|^2+\left\|\nabla^{N_0-1}f\right\|_\sigma^2\right)
+\varepsilon\left\|\nabla^{N_0-1}f\right\|_\sigma^2,
\end{aligned}
\end{equation*}
where $\theta=\frac{3+2s}{2(N_0-1+s)}$ and $\hat{l}_4=\frac{\gamma+2}2-\frac{\gamma+1}{\theta}=\frac{\gamma+2}2-\frac{2(N_0-1+s)(\gamma+1)}{3+2s}.$

For $I_{9,2}$, we can deduce similarly that
\begin{equation*}
\begin{aligned}
I_{9,2}\lesssim&\left\|\nabla^{N_0-1} E\right\|\left\|\nabla f\langle v\rangle^{-\frac\gamma2}\right\|_{L^\infty}\left\|\nabla^{N_0-1}f\right\|_\sigma\\
\lesssim&\left\|\nabla^{N_0-1}E\right\|\left\|\nabla^2 f\langle v\rangle^{-\frac\gamma2}\right\|_{H^1} \left\|\nabla^{N_0-1}f\right\|_\sigma\\
\lesssim&\mathcal{\bar{E}}_{3,\frac{\gamma}{2(\gamma+2)}}(t)\left\|\nabla^{N_0-1}E\right\|^2+\varepsilon\left\|\nabla^{N_0-1}f\right\|_\sigma^2.
\end{aligned}
\end{equation*}
For $I_{9,3}$, we can get from Lemma \ref{lemma2.2} and Corollary \ref{corrollary} that
\begin{equation}\label{E_N_0-1}
\begin{aligned}
I_{9,3}\lesssim&\sum_{1\leq j\leq N_0-2}\left\|\nabla^{j} E\right\|_{L^6}\left\|\nabla ^{N_0-1-j}f\langle v\rangle^{-\frac\gamma2}\right\|_{L^3}\left\|\nabla^{N_0-1}f\right\|_{\sigma}\\
\lesssim&\sum_{1\leq j\leq N_0-2}\left\|\Lambda^{-s}E\right\|^{1-\theta}\left\|\nabla^{N_0-1}E\right\|^{\theta}
\left\|\nabla ^{N_0-1}f\langle v\rangle^{-\frac\gamma2}\right\|^\alpha
\left\|f\langle v\rangle^{-\frac\gamma2}\right\|^{1-\alpha}\left\|\nabla^{N_0-1}f\right\|_\sigma\\
\lesssim&\sum_{1\leq j\leq N_0-2}\left\|\Lambda^{-s}E\right\|^{1-\theta}\left\|\nabla^{N_0-1}E\right\|^{\theta} \left(\left\|\nabla^{N_0-1}f\langle v\rangle^{\hat{l}_5}\right\|^{1-\beta}\left\|\nabla^{N_0-1}f\langle v\rangle^{\frac{\gamma+2}{2}}\right\|^\beta\right)^{\alpha}
\left\|f\langle v\rangle^{-\frac\gamma2}\right\|^{1-\alpha}\left\|\nabla^{N_0-1}f\right\|_\sigma\\
\lesssim&\sum_{1\leq j\leq N_0-2}\left\|\Lambda^{-s}E\right\|^{1-\theta}
\left\|f\langle v\rangle^{-\frac\gamma2}\right\|^{1-\alpha}\left\|\nabla^{N_0-1}f\langle v\rangle^{\hat{l}_5}\right\|^{\alpha-\alpha\beta}\left\|\nabla^{N_0-1}E\right\|^{\theta}
\left\|\nabla^{N_0-1}f\langle v\rangle^{\frac{\gamma+2}{2}}\right\|^{\alpha\beta}
\left\|\nabla^{N_0-1}f\right\|_\sigma\\
\lesssim&\mathcal{\bar{E}}_{N_0-1,-\frac{\hat{l}_5}{\gamma+2}}(t)\left(\left\|\nabla^{N_0-1}E\right\|^2
+\left\|\nabla^{N_0-1}f\right\|_\sigma^2\right)+\varepsilon\left\|\nabla^{N_0-1}f\right\|_\sigma^2.
\end{aligned}
\end{equation}
Here $\theta=\frac{j+s+1}{N_0-1+s}$, $\alpha=\frac{2N_0-1-j}{2(N_0-1)}$, and $\frac{\gamma+2}2\cdot\beta+\hat{l}_5(1-\beta)=-\frac\gamma 2$. Thus if we set $\alpha\beta=1-\theta$, we can deduce that
$\hat{l}_5=\frac{-\frac{\gamma}2-\frac{\gamma+2}2\cdot\beta}{1-\beta}=\frac{\gamma+2}2-\frac{\gamma+1}{1-\beta}=\frac{\gamma+2}2-\frac{(\gamma+1)(N_0-1+s)(2N_0-1-j)}{(N_0-1+s)(2N_0-1-j)-2(N_0-1)(N_0-2-j)}$ with $1\leq j\leq N_0-2$.

Consequently the estimates on $I_{9,1}$, $I_{9,2}$, and $I_{9,3}$ tell us that
$$
I_9\lesssim \max\left\{\mathcal{\bar{E}}_{3,\frac{\gamma}{2(\gamma+2)}}(t),\mathcal{\bar{E}}_{N_0-1,-\frac{\max\{\hat{l}_4,\hat{l}_5\}}{\gamma+2}}(t)\right\}\left(\left\|\nabla^{N_0-1}E\right\|^2
+\left\|\nabla^{N_0-1}f\right\|_\sigma^2\right)+\varepsilon\left\|\nabla^{N_0-1}f\right\|_\sigma^2.
$$

Repeating the arguments used to deal with $I_9$, we can bound $I_{10}$ and $I_{11}$ by
$$
I_{10}+I_{11}\lesssim\max\left\{\mathcal{\bar{E}}_{3,\frac{\gamma-2}{2(\gamma+2)}}(t),\mathcal{\bar{E}}_{N_0-1,-\frac{\hat{l}_6}{\gamma+2}}(t)\right\}\left(\left\|\nabla^{N_0-1}(E,B)\right\|^2
+\left\|\nabla^{N_0-1}f\right\|_\sigma^2\right)+\varepsilon\left\|\nabla^{N_0-1}f\right\|_\sigma^2.
$$
Here $\hat{l}_6=\frac{-\frac{\gamma}2+1-\frac{\gamma+2}2\cdot\beta}{1-\beta}=\frac{\gamma+2}2-\frac{\gamma}{1-\beta}=\frac{\gamma+2}2-\frac{\gamma(N_0-1+s)(2N_0-1-j)}{(N_0-1+s)(2N_0-1-j)-2(N_0-1)(N_0-2-j)}$
with $1 \leq j\leq N_0-1$.

As to $I_{12}$, we have from Lemma \ref{Lemma2.1} and Lemma \ref{lemma2.2} that
\begin{equation}\label{Gamma_N_0-1}
\begin{aligned}
I_{12}\lesssim&
\left\|\left|\mu^\delta f\right|_{L^2_v}\left|\nabla^{N_0-1}f\right|_{L^2_\sigma}\right\|
\left\|\nabla^{N_0-1}\{{\bf I-P}\}f\right\|_\sigma+\left\|\left|\mu^\delta \nabla^{N_0-1}f\right|_{L^2_v}|f|_{L^2_\sigma}\right\|\left\|\nabla^{N_0-1}\{{\bf I-P}\}f\right\|_\sigma\\
&+\sum_{1\leq j\leq N_0-2}\left\|\left|\nabla^j\mu^\delta f\right|_{L^2_v} \left|\nabla^{N_0-1-j}f\right|_{L^2_\sigma}\right\|\left\|\nabla^{N_0-1}\{{\bf I-P}\}f\right\|_\sigma\\
\lesssim& \left\|\mu^\delta f\right\|_{L^\infty_xL^2_v}\left\|\nabla^{N_0-1}f\right\|_\sigma\left\|\nabla^{N_0-1}\{{\bf I-P}\}f\right\|_\sigma
+\left\|\mu^\delta\nabla^{N_0-1} f\right\|\|f\|_{L^\infty_xL^2_\sigma}\left\|\nabla^{N_0-1}\{{\bf I-P}\}f\right\|_\sigma\\
&+\sum_{1\leq j\leq N_0-2}\left\|\nabla^j\mu^\delta f\right\|_{L^3_xL^2_v} \left\|\nabla^{N_0-1-j}f\right\|_{L^6_xL^2_\sigma}\left\|\nabla^{N_0-1}\{{\bf I-P}\}f\right\|_\sigma\\
\lesssim&\max\left\{\mathcal{E}_{N_0}(t),\mathcal{E}_{1,1}(t)\right\}\left\|\nabla^{N_0-1}f\right\|_\sigma^2
+\varepsilon\left\|\nabla^{N_0-1}\{{\bf I-P}\}f\right\|^2_\sigma.
\end{aligned}
\end{equation}
Here in deducing the last inequality, we have used the following inequality
\begin{equation}
\begin{aligned}
&\sum_{1\leq j\leq N_0-2}\left\|\nabla^j\mu^\delta f\right\|_{L^3_xL^2_v} \left\|\nabla^{N_0-1-j}f\right\|_{L^6_xL^2_\sigma}\left\|\nabla^{N_0-1}\{{\bf I-P}\}f\right\|_\sigma\\
&\lesssim \sum_{1\leq j\leq N_0-2}\left\|\nabla^m\mu^\delta f\right\|^{1-\frac{j-1}{N_0-1}}\left\|\nabla^{N_0-1}\mu^\delta f\right\|^{\frac{j-1}{N_0-1}}\|f\|_\sigma^{\frac{j-1}{N_0-1}}\left\|\nabla^{N_0-1} f\right\|_\sigma^{1-\frac{j-1}{N_0-1}}\left\|\nabla^{N_0-1}\{{\bf I-P}\}f\right\|_\sigma\\
&\lesssim \max\left\{\mathcal{E}_{N_0}(t),\mathcal{E}_{1,1}(t)\right\}\left\|\nabla^{N_0-1}f\right\|_\sigma^2
+\varepsilon\left\|\nabla^{N_0-1}\{{\bf I-P}\}f\right\|^2_\sigma,
\end{aligned}
\end{equation}
where $m=\frac{3N_0-1}{2(N_0-j)}\leq \frac{3N_0}{4}.$

Thus collecting the above estimates on $I_9$, $I_{10}$, $I_{11}$, and $I_{12}$ gives (\ref{Lemma4.3-2}).

Now we turn to the case of  $k=N_0\geq 3$. To this end, we have from (\ref{k}) with $k=N_0$ that
\begin{equation}\label{N_0-f}
\begin{aligned}
&\frac{d}{dt}\left(\left\|\nabla^{N_0}f\right\|^2+\left\|\nabla^{N_0}E\right\|^2\right)+\left\|\nabla^{N_0}\{{\bf I-P}\}f\right\|^2_{\sigma}\\
\lesssim&\underbrace{\left|\left(\nabla^{N_0}(E\cdot vf),\nabla^{N_0}f\right)\right|}_{I_{13}}
+\underbrace{\left|\left(\nabla^{N_0}(E\cdot \nabla_vf),\nabla^{N_0}f\right)\right|}_{I_{14}}\\
&+\underbrace{\left|\left(\nabla^{N_0}((v\times B)\cdot \nabla_vf),\nabla^{N_0}f\right)\right|}_{I_{15}}
+\underbrace{\left|\left(\nabla^{N_0}{\Gamma}(f,f),\nabla^{N_0}f\right)\right|}_{I_{16}}.
\end{aligned}
\end{equation}
Now we turn to estimate $I_j (j=13, 14,15, 16)$ term by term. For $I_{13}$, noticing that
\begin{eqnarray}\label{N_0-E}
I_{13}&\lesssim&\underbrace{\left|\left(E\cdot v\nabla^{N_0}f,\nabla^{N_0}f\right)\right|}_{I_{13,1}}
+\underbrace{\left|\left(\nabla E\cdot v\nabla^{N_0-1}f,\nabla^{N_0}f\right)\right|}_{I_{13,2}}\nonumber\\
&&+\underbrace{\left|\left(\nabla^{N_0-1} E\cdot v\nabla f,\nabla^{N_0}f\right)\right|}_{I_{13,3}}
+\underbrace{\left|\left(\nabla^{N_0} E\cdot vf,\nabla^{N_0}f\right)\right|}_{I_{13,4}}\nonumber\\
&&+\underbrace{\sum_{2\leq j\leq N_0-2}\left|\left(\nabla^{j} E\cdot v\nabla ^{N_0-j}f,\nabla^{N_0}f\right)\right|}_{I_{13,5}},
\end{eqnarray}
we have from the Holder inequality and Lemma \ref{lemma2.2} and Corollary \ref{corrollary} that
\begin{equation*}
\begin{aligned}
I_{13,1}\lesssim&\|E\|_{L^\infty}\left\|\nabla^{N_0}f\langle v\rangle^{-\frac\gamma2}\right\|\left\|\nabla^{N_0}f\right\|_\sigma\\
\lesssim&\left\|\Lambda^{-s}E\right\|^{1-\theta}\left\|\nabla^{N_0-1}E\right\|^{\theta}\left\|\nabla^{N_0}f\langle v\rangle^{\hat{l}_7}\right\|^\theta
\left\|\nabla^{N_0}f\langle v\rangle^{\frac{\gamma+2}{2}}\|^{1-\theta}\right\|\nabla^{N_0}f\|_\sigma\\
\lesssim&\mathcal{\bar{E}}_{N_0,-\frac{\hat{l}_7}{\gamma+2}}(t)\left(\left\|\nabla^{N_0-1}E\right\|^2
+\left\|\nabla^{N_0}f\right\|_\sigma^2\right)
+\varepsilon\left\|\nabla^{N_0}f\right\|_\sigma^2.
\end{aligned}
\end{equation*}
Here $\theta=\frac{3+2s}{2(N_0-1+s)}$ and $\hat{l}_7=\frac{\gamma+2}2-\frac{\gamma+1}{\theta}=\frac{\gamma+2}2-\frac{2(N_0-1+s)(\gamma+1)}{3+2s}$.

For $I_{13,2}$ and $I_{13,3}$, we can get from the Holder inequality and Lemma \ref{lemma2.2} that
\begin{equation*}
\begin{aligned}
I_{13,3}\lesssim
 \left\|\nabla^{N_0-1} E\right\|\left\|\nabla f\langle v\rangle^{-\frac\gamma2}\right\|_{L^\infty}\left\|\nabla^{N_0}f\right\|_\sigma
\lesssim\mathcal{\bar{E}}_{N_0,\frac{\gamma}{2(\gamma+2)}}(t)\left\|\nabla^{N_0-1}E\right\|^2+\varepsilon\left\|\nabla^{N_0}f\right\|_\sigma^2
\end{aligned}
\end{equation*}
and
\begin{equation*}
\begin{aligned}
I_{13,2}\lesssim&\|\nabla E\|_{L^3}\left\|\nabla^{N_0-1}f\langle v\rangle^{-\frac\gamma2}\right\|_{L^6}\left\|\nabla^{N_0}f\right\|_\sigma\\
\lesssim&\left\|\Lambda^{-s}E\right\|^{1-\theta}\left\|\nabla^{N_0-1}E\right\|^{\theta}\left\|\nabla^{N_0}f\langle v\rangle^{\hat{l}_8}\right\|^\theta
\left\|\nabla^{N_0}f\langle v\rangle^{\frac{\gamma+2}{2}}\right\|^{1-\theta}\left\|\nabla^{N_0}f\right\|_\sigma\\
\lesssim&\mathcal{\bar{E}}_{N_0,-\frac{\hat{l}_8}{\gamma+2}}(t)\left(\left\|\nabla^{N_0-1}E\right\|^2
+\left\|\nabla^{N_0}f\right\|_\sigma^2\right)
+\varepsilon\left\|\nabla^{N_0}f\right\|_\sigma^2
\end{aligned}
\end{equation*}
with  $\theta=\frac{3+2s}{2(N_0-1+s)}$, $\hat{l}_8=\frac{\gamma+2}2-\frac{\gamma+1}{\theta}=\frac{\gamma+2}2-\frac{2(N_0-1+s)(\gamma+1)}{3+2s}.$

Similar to that of (\ref{E_N_0-1}), $I_{13,5}$ can be dominated by
\begin{equation}
\begin{aligned}
I_{13,5}\lesssim&\sum_{2\leq j\leq N_0-2}\left\|\nabla^{j} E\right\|_{L^6}\left\|\nabla ^{N_0-j}f\langle v\rangle^{-\frac\gamma2}\right\|_{L^3}
\left\|\nabla^{N_0}f\right\|_{\sigma}\\
\lesssim&\sum_{2\leq j\leq N_0-2}\left\|\Lambda^{-s}E\right\|^{1-\theta}\left\|\nabla^{N_0-1}E\right\|^{\theta}
\left\|\nabla ^{N_0-1}f\langle v\rangle^{-\frac\gamma2}\right\|^\alpha\left\|f\langle v\rangle^{-\frac\gamma2}\right\|^{1-\alpha}
\left\|\nabla^{N_0}f\right\|_\sigma\\
\lesssim&\sum_{2\leq j\leq N_0-2}\left\|\Lambda^{-s}E\right\|^{1-\theta}\left\|\nabla^{N_0-1}E\right\|^{\theta}\left(\left\|\nabla^{N_0-1}f\langle v\rangle^{\hat{l}_9}\right\|^{1-\beta}\left\|\nabla^{N_0-1}f\langle v\rangle^{\frac{\gamma+2}{2}}\right\|^\beta\right)^{\alpha}
\left\|f\langle v\rangle^{-\frac\gamma2}\right\|^{1-\alpha}\left\|\nabla^{N_0}f\right\|_\sigma\\
\lesssim&\sum_{2\leq j\leq N_0-2}\left\|\Lambda^{-s}E\right\|^{1-\theta}\left\|f\langle v\rangle^{-\frac\gamma2}\right\|^{1-\alpha}
\left\|\nabla^{N_0-1}f\langle v\rangle^{\hat{l}_9}\right\|^{\alpha-\alpha\beta}\left\|\nabla^{N_0-1}E\right\|^{\theta}\left\|\nabla^{N_0-1}f\langle v\rangle^{\frac{\gamma+2}{2}}\right\|^{\alpha\beta}\left\|\nabla^{N_0}f\right\|_\sigma\\
\lesssim&\mathcal{\bar{E}}_{N_0-1,-\frac{\hat{l}_9}{\gamma+2}}(t)\left(\left\|\nabla^{N_0-1}E\right\|^2+\left\|\nabla^{N_0-1}f\right\|_\sigma^2\right)
+\varepsilon\left\|\nabla^{N_0-1}f\right\|_\sigma^2.
\end{aligned}
\end{equation}
Here $\theta=\frac{j+s+1}{N_0-1+s}$, $\alpha=\frac{2N_0-2j+1}{2(N_0-1)}$, and $\frac{\gamma+2}2\cdot\beta+\hat{l}(1-\beta)=-\frac\gamma 2$. If we set $\alpha\beta=1-\theta$, then $\hat{l}_9=-\frac{-\frac{\gamma}2-\frac{\gamma+2}2\cdot\beta}{1-\beta}=\frac{\gamma+2}2-\frac{\gamma+1}{1-\beta}=\frac{\gamma+2}2-\frac{(\gamma+1)(N_0-1+s)(2N_0-2j+1)}{(N_0-1+s)(2N_0-2j+1)-2(N_0-1)(N_0-2-j)}$ with $2\leq j\leq N_0-2$.

For the most important term $I_{13,4}$, if $n>\frac 23N_0-\frac53$, we have from Lemma \ref{lemma2.2} that
\begin{equation}\label{E-N_0}
\begin{aligned}
I_{13,4}
\lesssim&\left\|\nabla^{N_0-1}E\right\|^{\frac {n}{n+1}}\left\|\nabla^{N_0+n}E\right\|^{\frac1{n+1}}
\left\|\langle v\rangle^{-\frac\gamma2}f\right\|^\frac{2N_0-5}{2N_0-2}\left\|\langle v\rangle^{-\frac\gamma2}\nabla^{N_0-1}f\right\|^{\frac{3}{2N_0-2}}\left\|\nabla^{N_0}f\right\|_{\sigma}\\
\lesssim&\left\|\nabla^{N_0-1}E\right\|^{\frac {n}{n+1}}\left\|\nabla^{N_0+n}E\right\|^{\frac1{n+1}}
\left(\left\|\langle v\rangle^{\hat{l}_{10}}\nabla^{N_0-1}f\right\|^{1-\alpha}\left\|\langle v\rangle^{\frac{\gamma+2}2}\nabla^{N_0-1}f\right\|^{\alpha}\right)^{\frac{3}{2N_0-2}}
\left\|\langle v\rangle^{-\frac\gamma2}f\right\|^\frac{2N_0-5}{2N_0-2}\left\|\nabla^{N_0}f\right\|_{\sigma}\\
\lesssim&\left\|\nabla^{N_0-1}E\right\|^{\frac {n}{n+1}}\left\|\nabla^{N_0+n}E\right\|^{\frac1{n+1}}
\left\|\langle v\rangle^{\frac{\gamma+2}2}\nabla^{N_0-1}f\right\|^{\frac{3\alpha}{2N_0-2}}
\left\|\langle v\rangle^{\hat{l}_{10}}\nabla^{N_0-1}f\right\|^{\frac{3(1-\alpha)}{2N_0-2}}
\left\|\langle v\rangle^{-\frac\gamma2}f\right\|^\frac{2N_0-5}{2N_0-2}\left\|\nabla^{N_0}f\right\|_{\sigma}\\
\lesssim&\left\|\nabla^{N_0-1}E\right\|^{\frac {n}{n+1}}\left\|\nabla^{N_0+n}E\right\|^{\frac1{n+1}}
\left\|\langle v\rangle^{\frac{\gamma+2}2}\nabla^{N_0-1}f\right\|^{\frac1{n+1}}
\left\|\langle v\rangle^{\hat{l}_{10}}\nabla^{N_0-1}f\right\|^{\frac{3(1-\alpha)}{2N_0-2}}
\left\|\langle v\rangle^{-\frac\gamma2}f\right\|^\frac{2N_0-5}{2N_0-2}\left\|\nabla^{N_0}f\right\|_{\sigma}\\
\lesssim&\max\left\{\mathcal{E}_{N_0+n}(t),\mathcal{E}_{N_0-1,-\frac{\hat{l}_{10}}{\gamma+2}}(t)\right\}\left(\left\|\nabla^{N_0-1}E\right\|^2+\left\|\nabla^{N_0-1}f\right\|_\sigma^2\right)
+\varepsilon\left\|\nabla^{N_0}f\right\|_\sigma^2.
\end{aligned}
\end{equation}
Here we require that $\frac{3\alpha}{2N_0-2}=\frac{1}{n+1}$ which deduce that $\alpha=\frac{2N_0-2}{3(n+1)}$. We set $\frac{\gamma+2}2\cdot\alpha+\hat{l}_{10}(1-\alpha)=-\frac\gamma 2$ which yields that
$\hat{l}_{10}=\frac{-\frac{\gamma}2-\frac{\gamma+2}2\cdot\alpha}{1-\alpha}=\frac{\gamma+2}2-\frac{\gamma+1}{1-\alpha}=\frac{\gamma+2}2-\frac{3(\gamma+1)(n+1)}{3n-2 N_0+1}.$

Consequently, we can get from the estimates on $I_{13,j} (j=1,2,3,4,5)$ that
$$
\begin{array}{rl}
I_{13}\lesssim&\max\left\{\mathcal{E}_{N_0+n}(t),\mathcal{\bar{E}}_{N_0,\frac{\gamma}{2(\gamma+2)}}(t),\mathcal{E}_{N_0-1,-\frac{\max\{\hat{l}_7,\hat{l}_8,\hat{l}_9,\hat{l}_{10}\}}{\gamma+2}}(t)\right\}
\left(\left\|\nabla^{N_0-1}(E,B)\right\|^2
+\left\|\nabla^{N_0-1}f\right\|_\sigma^2+\left\|\nabla^{N_0}f\right\|_\sigma^2\right)\\[2mm]
&+\varepsilon\left\|\nabla^{N_0}f\right\|_\sigma^2.
\end{array}
$$

Similarly, we can conclude that
$$
\begin{array}{rl}
I_{14}+I_{15}\lesssim&\max\left\{\mathcal{E}_{N_0+n}(t),\mathcal{\bar{E}}_{N_0,\frac{\gamma-2}{2(\gamma+2)}}(t),\mathcal{E}_{N_0,-\frac{\hat{l}_{11}}{\gamma+2}}(t)\right\}
\left(\left\|\nabla^{N_0-1}(E,B)\right\|^2+\left\|\nabla^{N_0-1}f\right\|_\sigma^2
+\left\|\nabla^{N_0}f\right\|_\sigma^2\right)\\
&+\varepsilon\left\|\nabla^{N_0}f\right\|_\sigma^2
\end{array}
$$
and
\[
I_{16}\lesssim\max\left\{\mathcal{E}_{N_0}(t),\mathcal{E}_{1,1}(t)\right\}\left\|\nabla^{N_0}f\right\|_\sigma^2
+\varepsilon\left\|\nabla^{N_0}\{{\bf I-P}\}f\right\|_\sigma^2.
\]
Here
$
\hat{l}_{11}=\max\left\{\frac{\gamma+2}2-\frac{2(N_0-1+s)\gamma}{3+2s},\frac{\gamma+2}2-\frac{3\gamma(n+1)}{3n-2 N_0+1},
 \frac{\gamma+2}2-\frac{\gamma(N_0-1+s)(2N_0-2j+1)}{(N_0-1+s)(2N_0-2j+1)-2(N_0-1)(N_0-2-j)}\right\}
$
with $2\leq j\leq N_0-2$.

Putting the above estimates on $I_j (j=13, 14, 15, 16)$ together, one can deduce that (\ref{Lemma4.3-3}) holds with $N=N_0+n$ which implies that $N>\frac{5(N_0-1)}{3}$ since $n$ is assumed to satisfy $n>\frac{2N_0-5}{3}$. This completes the proof of Lemma \ref{lemma4.3}.
\end{proof}

The next lemma is concerned with the macro dissipation $\mathcal{D}_{N,mac}(t)$ defined by
$$
\mathcal{D}_{N,mac}(t)\sim\left\|\nabla_x(a_\pm,b,c)\right\|^2_{H^{N-1}}+\|a_+-a_-\|^2+\|E\|^2_{H^{N-1}}+\|\nabla_xB\|^2_{H^{N-2}}.
$$
\begin{lemma}\label{Lemma4.4}

\begin{itemize} For the macro dissipation estimates on $f(t,x,v)$, we have the following results:
\item[(i).] For $k=0,1,2\cdots,N_0-2$, there exist interactive energy functionals $G^k_f(t)$ satisfying
\[
G^k_f(t)\lesssim \left\|\nabla^k(f,E,B)\right\|^2+\left\|\nabla^{k+1}(f,E,B)\right\|^2+\left\|\nabla^{k+2}E\right\|^2
\]
such that
\begin{equation}\label{Lemma3.4-1}
\begin{split}
&\frac{d}{dt}G^k_f(t)+\left\|\nabla^k(E,a_+-a_-)\right\|_{H^1}^2+\left\|\nabla^{k+1}({\bf P}f,B)\right\|^2\\
\lesssim&\mathcal{\bar{E}}_{N_0-1,0}(t)\left(\left\|\nabla^{k+1}(E,B)\right\|^2+\left\|\nabla^{k+1}f\right\|^2_\sigma\right)
+\left\|\nabla^k\{{\bf I-P}\}f\right\|^2_\sigma\\
&
+\left\|\nabla^{k+1}\{{\bf I-P}\}f\right\|^2_\sigma
+\left\|\nabla^{k+2}\{{\bf I-P}\}f\right\|^2_\sigma;
\end{split}
\end{equation}
\item[(ii).] For $k=N_0-1$, there exists an interactive energy functional $G^{N_0-1}_f(t)$ satisfying
$$
G^{N_0-1}_f(t)\lesssim\left\|\nabla^{N_0-2}(f,E,B)\right\|^2+\left\|\nabla^{N_0-1}(f,E,B)\right\|^2
+\left\|\nabla^{N_0}(f,E)\right\|^2
$$
such that
\begin{equation}\label{Lemma3.4-2}
\begin{split}
&\frac{d}{dt}G^{N_0-1}_f(t) +\left\|\nabla^{N_0-2}(E,a_+-a_-)\right\|_{H^1}^2+\left\|\nabla^{N_0-1}B
\right\|^2+\left\|\nabla^{N_0}{\bf P}f\right\|^2\\
\lesssim&\mathcal{\bar{E}}_{N_0,0}(t)\left(\left\|\nabla^{N_0-1}(E,B)\right\|^2
+\left\|\nabla^{N_0-1}f\right\|^2_\sigma\right)+\left\|\nabla^{N_0-2}\{{\bf I-P}\}f\right\|^2_\sigma\\
&+\left\|\nabla^{N_0-1}\{{\bf I-P}\}f\right\|^2_\sigma
+\left\|\nabla^{N_0}\{{\bf I-P}\}f\right\|^2_\sigma;
\end{split}
\end{equation}
\item[(iii).] As \cite{Duan-VML}, there exists an interactive energy functional $\mathcal{E}^{int}_N(t)$ satisfying
$$
\mathcal{E}^{int}_N(t)\lesssim\sum_{|\alpha|\leq N}\left\|\partial^\alpha(f,E,B)\right\|^2$$
such that
\begin{equation}\label{lemma3.4-3}
\frac{d}{dt}\mathcal{E}^{int}_N(t)+\mathcal{D}_{N,mac}(t)\lesssim\sum_{|\alpha|\leq N}\left\|\partial^\alpha\{{\bf I-P}\}f\right\|_\sigma^2+\mathcal{E}_N(t)\mathcal{D}_N(t)
\end{equation}
holds for any $t\in[0,T]$.
\end{itemize}
\end{lemma}
\begin{proof} For the case of $k=0,1,\cdots,N_0-2$, as in \cite{Guo-JAMS-11}, by repeating the argument used in Lemma \ref{Lemma4.2} and by using the local conservation laws and the macroscopic equations \eqref{Macro-equation}, \eqref{Micro-equation}, \eqref{Macro-equation1}, \eqref{Micro-equation1}, \eqref{Micro-equation2}, and \eqref{a_+-a_--original} which are derived from the so-called macro-micro decomposition \eqref{macro-micro}, we can deduce that there exists an interactive functional $G^k_{\bar{f}}(t)$ such that
\begin{equation}
\begin{aligned}
&\frac{d}{dt}G^k_{\bar{f}}(t)+\left\|\nabla^k(a_+-a_-)\right\|^2+\left\|\nabla^{k+1}{\bf P}f\right\|^2\\[2mm]
\lesssim&\underbrace{\left\|\left\langle\nabla^k\left({\bf \Gamma}(f,f)+\frac12 v\cdot E f+ (E+v\times B)\cdot\nabla_{ v}f\right),\ \mu^\delta\right\rangle\right\|^2}_{I_{17}}\\[2mm]
&+\left\|\nabla^k\{{\bf I-P}\}f\right\|^2_\sigma+\left\|\nabla^{k+1}\{{\bf I-P}\}f\right\|^2_\sigma.
\end{aligned}
\end{equation}
Similar to that of Lemma \ref{lemma4.3}, $I_{17}$ can be bounded by
$$
I_{17}\lesssim\mathcal{\bar{E}}_{N_0-1,0}(t)\left(\left\|\nabla^{k+1}(E,B)\right\|^2+\left\|\nabla^{k+1}f\right\|^2_\sigma\right)
+\left\|\nabla^k\{{\bf I-P}\}f\right\|^2_\sigma,
$$
which yields that
\begin{equation}\label{pf-k}
\begin{aligned}
&\frac{d}{dt}G^k_{\bar{f}}(t)+\left\|\nabla^k(a_+-a_-)\right\|^2+\left\|\nabla^{k+1}{\bf P}f\right\|^2\\
\lesssim&\mathcal{\bar{E}}_{N_0-1,0}(t)\left(\left\|\nabla^{k+1}(E,B)\right\|^2+\left\|\nabla^{k+1}f\right\|^2_\sigma\right)
+\left\|\nabla^k\{{\bf I-P}\}f\right\|^2_\sigma+\left\|\nabla^{k+1}\{{\bf I-P}\}f\right\|^2_\sigma,
\end{aligned}
\end{equation}
where
$$
G^k_{\bar{f}}(t)\lesssim \left\|\nabla^k f\right\|^2+\left\|\nabla^{k+1}f\right\|^2.
$$

When $k=N_0-1$, we have by employing the same argument that
\begin{equation}\label{pf-N_0-1}
\begin{aligned}
&\frac{d}{dt}G^{N_0-1}_{\bar{f}}(t)+\left\|\nabla^{N_0-1}(a_+-a_-)\right\|^2+\left\|\nabla^{N_0}{\bf P}f\right\|^2\\
\lesssim&\mathcal{\bar{E}}_{N_0,0}(t)\left(\left\|\nabla^{N_0-1}(E,B)\right\|^2+\left\|\nabla^{N_0-1}f\right\|^2_\sigma\right)
+\left\|\nabla^{N_0-1}\{{\bf I-P}\}f\right\|^2_\sigma+\left\|\nabla^{N_0}\{{\bf I-P}\}f\right\|^2_\sigma,
\end{aligned}
\end{equation}
where
$$
G^{N_0-1}_{\bar{f}}(t)\lesssim \left\|\nabla^{N_0-1} f\right\|^2+\left\|\nabla^{N_0}f\right\|^2.
$$

By the similar procedure as in the proof of Lemma \ref{lemma3.3}, we have for $k=0,1,\cdots,N_0-2$ that
\begin{eqnarray*}
&&\frac{d}{dt}G^k_{E,B}(t)+\left\|\nabla^{k}(E,a_+-a_-)\right\|_{H^1}^2+\left\|\nabla^{k+1}B\right\|^2\\
&\lesssim&\underbrace{\left\|\nabla^k\left\{E(a_++a_-)+b\times B+\left\langle{ \Gamma}(f,f),\ \mu^\delta\right\rangle\right\}\right\|^2}_{I_{18}}\\
&&+\underbrace{\left\|\nabla^{k+1}\left\{E(a_++a_-)+b\times B+\left\langle{ \Gamma}(f,f),\ \mu^\delta\right\rangle\right\}\right\|^2}_{I_{19}}\\
&&+\left\|\nabla^k\{{\bf I-P}\}f\right\|^2_\sigma
+\left\|\nabla^{k+1}\{{\bf I-P}\}f\right\|^2_\sigma+\left\|\nabla^{k+2}\{{\bf I-P}\}f\right\|^2_\sigma,
\end{eqnarray*}
where
$$
G^k_{E,B}(t)\lesssim \left\|\nabla^k(f,E,B)\right\|^2+\left\|\nabla^{k+1}(f,E,B)\right\|^2
+\left\|\nabla^{k+2}E\right\|^2.
$$
$I_{18}$ and $I_{19}$ can be bounded by
$$
\left|I_{18}\right|+\left|I_{19}\right|\lesssim\mathcal{\bar{E}}_{N_0,0}(t)\left(\left\|\nabla^{k+1}(E,B)\right\|^2
+\left\|\nabla^{k+1}f\right\|^2_\sigma\right)
+\left\|\nabla^k\{{\bf I-P}\}f\right\|^2_\sigma+\left\|\nabla^{k+1}\{{\bf I-P}\}f\right\|^2_\sigma,
$$
which deduces that
\begin{equation}\label{E,B-k}
\begin{aligned}
&\frac{d}{dt}G^k_{E,B}(t)+\left\|\nabla^{k}(E,a_+-a_-)\right\|_{H^1}^2+\left\|\nabla^{k+1}B\right\|^2\\
\lesssim&\mathcal{\bar{E}}_{N_0,0}(t)\left(\left\|\nabla^{k+1}(E,B)\right\|^2+\left\|\nabla^{k+1}f\right\|^2_\sigma\right)+\left\|\nabla^k\{{\bf I-P}\}f\right\|^2_\sigma\\
&+\left\|\nabla^{k+1}\{{\bf I-P}\}f\right\|^2_\sigma
+\left\|\nabla^{k+2}\{{\bf I-P}\}f\right\|^2_\sigma.
\end{aligned}
\end{equation}

With the above estimates in hand, we now turn to deduce \eqref{Lemma3.4-1}, (\ref{Lemma3.4-2}), and \eqref{lemma3.4-3}. In fact for the case of $k=0,1,\cdots,N_0-2$, a proper linear combination of \eqref{pf-k} and \eqref{E,B-k} gives \eqref{Lemma3.4-1} with
$G^k_{f}(t)\thicksim G^k_{\bar{f}}(t)+\kappa_1G^k_{E,B}(t), 0<\kappa_1\ll1,$
while (\ref{Lemma3.4-2}) follows by a suitable linear combination of \eqref{pf-N_0-1} and \eqref{E,B-k} with $k=N_0-2$,
where
$G^{N_0-1}_{f}(t)\thicksim G^{N_0-1}_{\bar{f}}(t)+\kappa_2G^{N_0-2}_{E,B}(t), 0<\kappa_2\ll1.$

Finally, \eqref{lemma3.4-3} follows by the same trick as in \cite{Duan-VML}, we omit the details for brevity.
Thus we have completed the proof of Lemma \ref{lemma4.3}.
\end{proof}

\section{The proofs of our main results}

This section is devoted to proving our main results. For this purpose, suppose that the Cauchy problem (\ref{f}) and (\ref{f-initial}) admits a unique local solution $f(t,x,v)$ defined on the time interval $ 0\leq t\leq T$ for some $0<T<\infty$ and such a solution $f(t,x,v)$
satisfies the a priori assumption
\begin{equation}\label{E-priori}
\begin{aligned}
X(t)=\sup_{0\leq \tau\leq t}\left\{\bar{\mathcal{E}}_{N_0,l_0+l^*}(\tau)+\mathcal{E}_{N}(\tau)+(1+\tau)^{-\frac{1+\epsilon_0}{2}}\mathcal{E}_{N,l}(\tau)\right\}\leq M,
\end{aligned}
\end{equation}
where the parameters $N_0,N,l, l_0,$ and $l^*$ are given in Theorem \ref{Th1.1} and $M$ is a sufficiently small positive constant. Then to use the continuation argument to extend such a solution step by step to a global one, one only need to deduce certain uniform-in-time energy type estimates on $f(t,x,v)$ such that the a priori assumption (\ref{E-priori}) can be closed. Recall that in the coming analysis, the parameter $\gamma$ is assumed to satisfy $-3\leq \gamma<-2$ and it is worth point out that such a fact plays an important role in the analysis.

For this purpose, we first deduce the temporal decay of the energy functional $\mathcal{E}^k_{N_0}(t)$ in the following lemma
\begin{lemma}\label{Lemma4.5}
Assume that the a priori assumption (\ref{E-priori}) is true and let $l$ be as in Theorem $\ref{Th1.1}$, then there exist an energy functional $\mathcal{E}^k_{N_0}(t)$ and the corresponding energy dissipation rate functional $\mathcal{D}^k_{N_0}(t)$ satisfying \eqref{E_k} and \eqref{D_k} respectively such that
\begin{equation}\label{Lemma4.5-1}
\frac{d}{dt}\mathcal{E}^k_{N_0}(t)+\mathcal{D}^k_{N_0}(t)\leq 0
\end{equation}
holds for $k=0,1,2,\cdots, N_0-2$ and all $0\leq t\leq T.$

Furthermore, we can get that
\begin{equation}\label{Lemma4.5-2}
\mathcal{E}^k_{N_0}(t)\lesssim\sup_{0\leq \tau\leq t}\left\{\mathcal{\bar{E}}_{N_0,\frac{k+s}{2}}(\tau),\mathcal{E}_{N_0+k+s}(\tau)\right\}(1+t)^{-(k+s)},\quad 0\leq t\leq T.
\end{equation}
Here we need to assume that $N\geq N_0+k+s$.
\end{lemma}
\begin{proof} From Lemma \ref{lemma4.3} and \ref{Lemma4.4}, we can obtain by a suitable linear combination of the corresponding estimates obtained there that there exist an energy functional $\mathcal{E}^k_{N_0}(t)$ and the corresponding energy dissipation rate functional $\mathcal{D}^k_{N_0}(t)$ satisfying \eqref{E_k} and \eqref{D_k} respectively such that
\begin{equation*}
\frac{d}{dt}\mathcal{E}^k_{N_0}(t)+\mathcal{D}^k_{N_0}(t)\leq 0
\end{equation*}
{\color{red}holds for all $0\leq t\leq T$.}

To get the time decay rate of $\mathcal{E}^k_{N_0}(t)$, as in \cite{Guo-CPDE-12}, we need to compare the difference between $\mathcal{E}^k_{N_0}(t)$ and $\mathcal{D}^k_{N_0}(t)$.  For this purpose, we first deduce from Lemma \ref{lemma2.2} and Corollary \ref{corrollary} that
\begin{equation*}
\begin{aligned}
\left\|\nabla^k({\bf P}f,E,B)\right\|\leq \left\|\nabla^{k+1}({\bf P}f,E,B)\right\|^{\frac{k+s}{k+s+1}}
\left\|\Lambda^{-s}({\bf P}f,E,B)\right\|^{\frac{1}{k+s+1}}.
\end{aligned}
\end{equation*}
The above inequality together with the facts that
\begin{equation*}
\begin{aligned}
\left\|\nabla^mf\right\|\leq \left\|\langle v\rangle^{\frac{\gamma+2}2}\nabla^mf\right\|^{\frac{k+s}{k+s+1}}
\left\|\langle v\rangle^{\bar{l}}\nabla^mf\right\|^{\frac{1}{k+s+1}},\quad \bar{l}=-\frac{(\gamma+2)(k+s)}{2},
\end{aligned}
\end{equation*}
$$
\left\|\nabla^{N_0}(E,B)\right\|\lesssim\left\|\nabla^{N_0-1}(E,B)\right\|^\frac{k+s}{k+s+1}
\left\|\nabla^{N_0+k+s}(E,B)\right\|^\frac{1}{k+s+1}
$$
imply
\begin{equation*}
\begin{aligned}
\mathcal{E}^k_{N_0}(t)\leq \left(\mathcal{D}^k_{N_0}(t)\right)^\frac{k+s}{k+s+1}\left\{\sup_{0\leq \tau\leq t}\left\{\mathcal{\bar{E}}_{N_0,\frac{k+s}{2}}(\tau),\mathcal{E}_{N_0+k+s}(\tau)\right\}\right\}^\frac{1}{k+s+1}.
\end{aligned}
\end{equation*}
Hence, we can deduce that
\begin{equation*}
\frac{d}{dt}\mathcal{E}^k_{N_0}(t)+\left\{\sup_{0\leq \tau\leq t}\left\{\mathcal{\bar{E}}_{N_0,\frac{k+s}{2}}(\tau),\mathcal{E}_{N_0+k+s}(\tau)\right\}\right\}^{-\frac{1}{k+s}}
\left\{\mathcal{E}^k_{N_0}(t)\right\}^{1+\frac{1}{k+s}}\leq 0
\end{equation*}
and we can get by solving the above inequality directly that
\begin{equation*}
\mathcal{E}^k_{N_0}(t)\lesssim\sup_{0\leq \tau\leq t}\left\{\mathcal{\bar{E}}_{N_0,\frac{k+s}{2}}(\tau),\mathcal{E}_{N_0+k+s}(\tau)\right\}(1+t)^{-(k+s)}.
\end{equation*}
This completes the proof of Lemma \ref{Lemma4.5}.
\end{proof}
Bases on the time decay estimate on $\mathcal{E}^k_{N_0}(t)$ obtained in the above lemma, we are now ready to deduce the Lyapunov inequality for the energy functional $\mathcal{E}_{N}(t)$.
\begin{lemma}\label{lemma3.7} Under the assumptions listed in Lemma \ref{Lemma4.5}, we can deduce that there exist an energy functional $\mathcal{E}_{N}(t)$ and the corresponding energy dissipation functional $\mathcal{D}_{N}(t)$ which satisfy (\ref{E_N}), (\ref{D_N}) respectively such that
\begin{equation}\label{lemma3.7-1}
\frac{d}{dt}\mathcal{E}_{N}(t)+\mathcal{D}_{N}(t)\lesssim \frac\delta{(1+t)^{1+\vartheta}}\mathcal{D}_{N,l}(t)+\mathcal{E}_{N}(t)\mathcal{E}^1_{N_0,l_0}(t)
\end{equation}
holds for all $0\leq t\leq T$. Here we choose $N_0\geq 4$, $0<\vartheta\leq \frac s2 $ when  $s\in \left[\frac 12,1\right]$ and $N_0\geq 3,$ $0<\vartheta\leq \frac s2-\frac12$ when  $s\in \left(1,\frac 32\right)$.
\end{lemma}
\begin{proof}
First of all, it is straightforward to establish the energy identities
\begin{eqnarray*}
&&\frac12\frac{d}{dt}\sum\limits_{N_0+1\leq|\alpha|\leq N}\left(\left\|\partial^\alpha f\right\|^2+ \left\|\partial^\alpha(E,B)\right\|^2\right)+\sum\limits_{N_0+1\leq|\alpha|\leq N}\left(L\partial^\alpha f,\partial^\alpha f\right)\\
&=&\underbrace{\sum\limits_{N_0+1\leq|\alpha|\leq N}\left(\partial^\alpha\left(\frac{q_0}{2}E\cdot vf\right),\partial^\alpha f\right)}_{J_1}-\underbrace{\sum\limits_{N_0+1\leq|\alpha|\leq N}\left(\partial^\alpha\left(q_0(E+v\times B)\cdot\nabla_vf\right),\partial^\alpha f\right)}_{J_2}\\
&&+\underbrace{\sum\limits_{N_0+1\leq|\alpha|\leq N}\left(\partial^\alpha\Gamma(f,f),\partial^\alpha f\right)}_{J_3}.
\end{eqnarray*}

For the $\partial^\alpha$ derivative term related to $(E,B)$ with $N_0+1\leq |\alpha|\leq N$, i.e., the estimates on $J_1$ and $J_2$, one has
\begin{equation}
\begin{aligned}
J_1
\lesssim&\|E\|_{L^\infty} \left\|\langle v\rangle^{1/2}\partial^{\alpha}f\right\|
\left\|\langle v\rangle^{1/2}\partial^\alpha f\right\|\\
&+\sum_{1\leq|\alpha_1|\leq N_0-1}\left\|\partial^{\alpha_1}E\right\|_{L^6}
\left\|\langle v\rangle^{1/2}\partial^{\alpha-\alpha_1}f\right\|_{L^3}
\left\|\langle v\rangle^{1/2}\partial^\alpha f\right\|\\
&+\sum_{|\alpha_1|= N_0}\left\|\partial^{\alpha_1}E\right\|
\left\|\langle v\rangle^{1/2}\partial^{\alpha-\alpha_1}f\right\|_{L^\infty}
\left\|\langle v\rangle^{1/2}\partial^\alpha f\right\|\\
&+\sum_{|\alpha_1|\geq N_0+1,\alpha_1\neq\alpha}\left\|\partial^{\alpha_1}E\right\|_{L^6}
\left\|\langle v\rangle^{-\frac\gamma2}\partial^{\alpha-\alpha_1}f\right\|_{L^3}
\left\|\partial^\alpha f\right\|_\sigma\\
&+\left\|\partial^{\alpha}E\right\|
\left\|\langle v\rangle^{-\frac\gamma2}f\right\|_{L^\infty}
\left\|\partial^\alpha f\right\|_\sigma\\
\lesssim&\left(\| E\|^{\frac{\gamma}{\gamma+1}}_{L^{\infty}}+\left\|\nabla^2E\right\|_{H^{ N_0-2}}\right)\mathcal{D}_{N,l}(t)
+\mathcal{E}_N(t)\mathcal{E}^1_{N_0,l_0}(t)+\varepsilon\mathcal{D}_{N}(t).
\end{aligned}
\end{equation}
Here the most subtle part is the estimation on the first term on the right-hand side of the above first inequality which is carried as follows
\begin{equation}
\begin{aligned}
&\|E\|_{L^\infty} \left\|\langle v\rangle^{1/2}\partial^{\alpha}f\right\|
\left\|\langle v\rangle^{1/2}\partial^\alpha f\right\|\\
\lesssim&\|E\|_{L^\infty} \left\|\langle v\rangle\partial^{\alpha}f\right\|^{1+\frac1\gamma}
\left\|\langle v\rangle^{\frac{\gamma+2}2}\partial^{\alpha}f\right\|^{-\frac1\gamma}\\
\lesssim&\|E\|^{\frac{\gamma}{\gamma+1}}_{L^\infty} \left\|\langle v\rangle\partial^{\alpha}f\right\|^2+\varepsilon\left\|\langle v\rangle^{\frac{\gamma+2}2}\partial^{\alpha}f\right\|^2,
\end{aligned}
\end{equation}
and the other terms are easier in some sense and can be treated similarly by noticing that $l_0\geq \frac{\gamma}{2(\gamma+2)}$ and $N\leq2N_0$.

For $J_2$, due to $\left((v\times B)\cdot \partial^\alpha\nabla_vf,\partial^\alpha f\right)=0$ and $\left(E\cdot \partial^\alpha\nabla_vf,\partial^\alpha f\right)=0$, we can deduce by employing the same argument to deal with $J_1$ that
$$
J_2\lesssim\left\|\nabla^2(E,B)\right\|_{H^{ N_0-2}}\mathcal{D}_{N,l}(t)
+\mathcal{E}_N(t)\mathcal{E}^1_{N_0,l_0}(t)
+\varepsilon\mathcal{D}_{N,l}(t).
$$
Here we have used the facts that $l_0\geq \frac{\gamma-2}{2(\gamma+2)}$ and $N\leq 2N_0$.

Finally, we can get from Lemma \ref{Lemma2.1} that
$$
J_3\lesssim \mathcal{E}_{N_0,N_0}^{1/2}(t)\mathcal{D}_N(t).
$$

Collecting the above estimates gives
\begin{equation}\label{E1}
\begin{split}
&\frac{d}{dt}\sum_{N_0+1\leq|\alpha|\leq N}\left(\left\|\partial^\alpha  f\right\|^2+\left\|\partial^\alpha(E,B)\right\|^2\right)
+\sum_{N_0+1\leq|\alpha|\leq N}\left\|\partial^\alpha\{{\bf I-P}\}f\right\|_\sigma\\
\lesssim&\left(\| E\|^{\frac{\gamma}{\gamma+1}}_{L^{\infty}}+\left\|\nabla^2(E,B)\right\|_{H^{ N_0-2}}\right)\mathcal{D}_{N,l}(t)
+\mathcal{E}_N(t)\mathcal{E}^1_{N_0,l_0}(t)+\varepsilon\mathcal{D}_{N}(t).
\end{split}
\end{equation}
If we choose $N_0\geq 4$, $0<\vartheta\leq \frac s2$ when $s\in \left[\frac 12,1\right]$ and $N_0\geq 3$, $0<\vartheta\leq \frac s2-\frac12$ when  $s\in \left(1,\frac 32\right)$, we can deduce from Lemma \ref{Lemma4.5} and a proper linear combination of  (\ref{E1}) and(\ref{Lemma4.5-1}) with $k=0$ that there exist an energy functional $\mathcal{E}_N(t)$ and an energy dissipation functional $\mathcal{D}_{N}(t)$ which satisfies (\ref{E_N}) and (\ref{D_N}) such that
\begin{equation}
\frac{d}{dt}\mathcal{E}_{N}(t)+\mathcal{D}_{N}(t)\lesssim \frac\delta{(1+t)^{1+\vartheta}}\mathcal{D}_{N,l}(t)+\mathcal{E}_{N}(t)\mathcal{E}^1_{N_0,l_0}(t).
\end{equation}
Thus we have completed the proof of Lemma \ref{lemma3.7}.
\end{proof}
Now we turn to the weighted energy estimates on $\mathcal{E}_{N,l}(t)$.
\begin{lemma}\label{lemma3.8} Under the assumptions listed in Lemma \ref{Lemma4.5}, we can deduce that there exist an energy functional $\mathcal{E}_{N,l}(t)$ and the corresponding energy dissipation functional $\mathcal{D}_{N,l}(t)$ which satisfy (\ref{E-}), (\ref{D-}) such that
\begin{equation}\label{lemma3.8-1}
\frac{d}{dt}\mathcal{E}_{N,l}(t)+\mathcal{D}_{N,l}(t)\lesssim\mathcal{E}_{N}(t)\mathcal{E}^1_{N_0,l_0}(t)
+\sum_{|\alpha|= N}\|\partial^\alpha E\|\left\|\mu^\delta\partial^\alpha f\right\|
\end{equation}
holds for all $0\leq t\leq T$. Here we choose $N_0\geq 4$, $0<\vartheta\leq \frac s2 $ when  $s\in \left[\frac 12,1\right]$ and $N_0\geq 3,$ $0<\vartheta\leq \frac s2-\frac12$ when  $s\in \left(1,\frac 32\right)$.
\end{lemma}
\begin{proof} The proof of this lemma is divided into two steps. The first step is to deduce the desired energy type estimates on the derivatives of $f(t,x,v)$ with respect to the $x-$variable only. For this purpose, the standard energy estimate on $\partial^\alpha f$ with $1\leq|\alpha|\leq N$ weighted by the time-velocity dependent function $w_{l}=w_l(t,v)$ gives
\begin{equation}\label{alpha-f}
\begin{aligned}
&\frac{d}{dt}\sum_{1\leq|\alpha|\leq N}\left\|w_l\partial^\alpha f\right\|^2+\sum_{1\leq|\alpha|\leq N}\left\|w_l\partial^\alpha f\right\|^2_\sigma
+\frac{\vartheta q}{(1+t)^{1+\vartheta}}\left\|\langle v\rangle w_l\partial^\alpha f\right\|^2\\
\lesssim&\sum_{1\leq|\alpha|\leq N}\left\|\partial^\alpha f\right\|_\sigma^2+\sum_{1\leq|\alpha|\leq N}
\left\|\partial^\alpha E\right\|\left\|\mu^\delta\partial^\alpha f\right\|
+\underbrace{\sum_{1\leq|\alpha|\leq N}\left|\left(\partial^\alpha (E\cdot vf),w^2_l\partial^\alpha f\right)\right|}_{J_4}\\
&+\underbrace{\sum_{1\leq|\alpha|\leq N}\left|\left(\partial^\alpha[ (E+v\times B)\cdot\nabla_v f],w^2_l\partial^\alpha f\right)\right|}_{J_5}+\underbrace{\sum_{1\leq|\alpha|\leq N}\left|\left(\partial^\alpha \Gamma(f,f),w^2_l\partial^\alpha f\right)\right|}_{J_6}.
\end{aligned}
\end{equation}
As to the estimates on $J_4$, $J_5$, and $J_6$, we can deduce by following exactly the argument used above to control $J_1$, $J_2$, and $J_3$ that
\begin{equation}
\begin{aligned}
J_4\lesssim&\|E\|_{L^\infty}\left\|\langle v\rangle^{1/2}w_l\partial^{\alpha}f\right\|_{L^2}
\left\|\langle v\rangle^{1/2}w_l\partial^\alpha f\right\|\\
&+\sum_{\substack{1\leq|\alpha_1|\leq N_0-1,\\ \alpha_1\neq\alpha}}\left\|\partial^{\alpha_1}E\right\|_{L^6}
\left\|\langle v\rangle^{1/2}w_l\partial^{\alpha-\alpha_1}f\right\|_{L^3}
\left\|\langle v\rangle^{1/2}w_l\partial^\alpha f\right\|\\
&+\sum_{\substack{|\alpha_1|= N_0,\\ \alpha_1\neq\alpha}}\left\|\partial^{\alpha_1}E\right\|
\left\|\langle v\rangle^{1/2}w_l\partial^{\alpha-\alpha_1}f\right\|_{L^\infty}
\left\|\langle v\rangle^{1/2}w_l\partial^\alpha f\right\|\\
&+\sum_{\substack{|\alpha_1|\geq N_0+1,\\ \alpha_1\neq\alpha}}\left\|\partial^{\alpha_1}E\right\|_{L^6}
\left\|\langle v\rangle^{-\frac\gamma2}w_l\partial^{\alpha-\alpha_1}f\right\|_{L^3}
\left\|w_l\partial^\alpha f\right\|_\sigma\\
&+\left\|\partial^{\alpha}E\right\|\left\|\langle v\rangle^{-\frac\gamma2}w_lf\right\|_{L^\infty}
\left\|w_l\partial^\alpha f\right\|_\sigma\\
\lesssim&(1+t)^{1+\vartheta}\left(\| E\|^{\frac{\gamma}{\gamma+1}}_{L^{\infty}} +\left\|\nabla^2E\right\|_{H^{ N_0-2}}\right)\mathcal{D}_{N,l}(t)
+\mathcal{E}_N(t)\mathcal{E}^1_{N_0,l_0}(t)+\varepsilon\mathcal{D}_{N,l}(t),
\end{aligned}
\end{equation}
\begin{equation}
\begin{aligned}
J_5\lesssim&(1+t)^{1+\vartheta}\left(\| E\|^{\frac{\gamma}{\gamma+1}}_{L^{\infty}} +\left\|\nabla^2_x(E,B)\right\|_{H^{ N_0-2}}\right)\mathcal{D}_{N,l}(t)
+\mathcal{E}_N(t)\mathcal{E}^1_{N_0,l_0}(t)+\varepsilon\mathcal{D}_{N,l}(t),
\end{aligned}
\end{equation}
and
\begin{equation}
\begin{aligned}
J_6\lesssim&\left\|\mu^\delta f\right\|_{L^\infty}\left\|w_l\partial^\alpha f\right\|^2_\sigma
 +\left\|\mu^\delta \partial^\alpha f\right\|\left\|w_l f\right\|_{L^\infty_xL^2_\sigma}
 \left\|w_l\partial^\alpha f\right\|_\sigma\\
 &+\sum_{1\leq|\alpha_1|\leq|\alpha|-1}\left\|\mu^{\delta}\partial^{\alpha_1}f\right\|_{L^3}
 \left\|w_l\partial^{\alpha-\alpha_1}f\right\|_{L^6_xL^2_\sigma}\left\|w_l\partial^\alpha f\right\|_\sigma\\
\lesssim& \mathcal{E}^{1/2}_N(t)\mathcal{D}_{N,l}(t).
\end{aligned}
\end{equation}
Here we have used the facts that $l_0\geq l+\frac{\gamma-2}{2(\gamma+2)}$ and $N\leq 2N_0$.

Collecting the above estimates gives the desired weighted energy type estimates on the derivatives of $f(t,x,v)$ with respect to the $x-$variables only as follows
\begin{equation}\label{3.65}
\begin{aligned}
&\frac{d}{dt}\sum_{1\leq|\alpha|\leq N}\left\|w_l\partial^\alpha f\right\|^2
+\sum_{1\leq|\alpha|\leq N}\left\|w_l\partial^\alpha f\right\|^2_\sigma
+\frac{\vartheta q}{(1+t)^{1+\vartheta}}\left\|\langle v\rangle w_l\partial^\alpha f\right\|^2\\[2mm]
\lesssim&\sum_{1\leq|\alpha|\leq N}\left\|\partial^\alpha f\right\|_\sigma^2
+\sum_{1\leq|\alpha|\leq N}\left\|\partial^\alpha E\right\|\left\|\mu^\delta\partial^\alpha f\right\| +\left(\mathcal{E}^{1/2}_N(t)+\varepsilon\right)\mathcal{D}_{N,l}(t)\\
&+(1+t)^{1+\vartheta}\left(\| E\|^{\frac{\gamma}{\gamma+1}}_{L^{\infty}}+\left\|\nabla^2_x(E,B)\right\|_{H^{ N_0-2}}\right)\mathcal{D}_{N,l}(t)
+\mathcal{E}_N(t)\mathcal{E}^1_{N_0,l_0}(t).
\end{aligned}
\end{equation}
Now we turn to get the desired energy type estimates on the mixed derivatives with respect to both $x$ and $v$ variables. Due to
$$
\left\|w_{l-|\beta|}\partial^\alpha_\beta{\bf P}f\right\|\lesssim\left\|\partial^\alpha(a_+,a_-,b,c)\right\|
$$
and notice that $\left\|\partial^\alpha(a_+,a_-,b,c)\right\|$ has been estimated in (\ref{3.65}), we only need to control $\left\|w_{l-|\beta|}\partial^\alpha_\beta\{{\bf I- P}\}f\right\|$ suitably and for this purpose, we need to deduce first the equations which the microscopic components of $f(t,x,v)$ satisfy. In fact, by applying the microscopic projection $\{{\bf I-P}\}$ to the first equation of (\ref{f}), we can get that
\begin{equation}\label{I-P}
\begin{split}
\partial_t\{{\bf I-P}\}f+v\cdot \nabla_x\{{\bf I-P}\}f-&E\cdot v\mu^{1/2}q_1+Lf
=\{{\bf I-P}\}g+{\bf P}(v\cdot\nabla_x f)-v\cdot\nabla_x{\bf P}f.
\end{split}
\end{equation}
From (\ref{I-P}), one can deduce the following weighted energy estimate on $\{{\bf I-P}\}f$
\begin{equation}\label{3.67}
\begin{aligned}
&\frac{d}{dt}\left\|w_l\{{\bf I-P}\}f\right\|^2+\left\|w_l\{{\bf I-P}\}f\right\|_\sigma^2
+\frac{\vartheta q}{(1+t)^{1+\vartheta}}\left\|\langle v\rangle w_l\{{\bf I-P}\}f\right\|^2\\
\lesssim&\left\|\{{\bf I-P}\}f\right\|_\sigma^2+\|E\|^2+\|\nabla f\|_\sigma^2+\left(\mathcal{E}^{1/2}_N(t)+\varepsilon\right)\mathcal{D}_{N,l}(t)\\
&+(1+t)^{1+\vartheta}\| E\|^{\frac{\gamma}{\gamma+1}}_{L^{\infty}}\mathcal{D}_{N,l}(t)
+\mathcal{E}_N(t)\mathcal{E}^1_{N_0,l_0}(t).
\end{aligned}
\end{equation}
Similarly, for the weighted energy estimate on $\{{\bf I-P}\}\partial^\alpha_\beta f$ with $|\alpha|+|\beta|\leq N$ and $|\beta|\geq 1$, we have
\begin{equation}\label{3.68}
\begin{aligned}
&\frac{d}{dt}\left(\sum_{m=1}^{N}C_m\sum_{\substack{|\beta|=m,\\|\alpha|+|\beta|\leq N}}\left\|w_{l-|\beta|}\partial^\alpha_\beta\{{\bf I-P}\}f\right\|^2\right)\\
&+\kappa\sum_{\substack{|\alpha|+|\beta|\leq N,\\|\beta|\geq 1}}\left(\left\| w_{l-|\beta|}\partial^\alpha_\beta\{{\bf I-P}\}f\right\|_\sigma^2
+\frac{q}{(1+t)^{1+\vartheta}}\left\|\langle v\rangle w_{l-|\beta|}\partial^\alpha_\beta\{{\bf I-P}\}f\right\|^2\right)\\
\lesssim&\sum_{|\alpha|\leq N}\left\|w_l\partial^\alpha\{{\bf I-P}\}f\right\|_\sigma^2+\sum_{|\alpha|\leq N-1 }\left(\left\|\nabla_x\partial^\alpha {\bf P} f\right\|^2
+\left\|\partial^\alpha E\right\|^2\right)+\left(\mathcal{E}^{1/2}_N(t)+\varepsilon\right)\mathcal{D}_{N,l}(t)\\
&+(1+t)^{1+\vartheta}\left(\| E\|^{\frac{\gamma}{\gamma+1}}_{L^{\infty}}+\left\|\nabla^2_x(E,B)\right\|_{H^{ N_0-2}}\right)\mathcal{D}_{N,l}(t)
+\mathcal{E}_N(t)\mathcal{E}^1_{N_0,l_0}(t).
\end{aligned}
\end{equation}
Here we used the fact that $\left((v\times B)\cdot\partial^\alpha_\beta\nabla_v \{{\bf I-P} \}f,w^2_{l-|\beta|}\partial_\beta^\alpha \{{\bf I-P} \}f\right)=0.$

Therefore, if we choose $N_0\geq 4$, $0<\vartheta\leq \frac s2$ when $s\in \left[\frac 12,1\right]$ and $N_0\geq 3$, $0<\vartheta\leq \frac s2-\frac12$ when  $s\in \left(1,\frac 32\right)$, a proper linear combination of (\ref{lemma3.7-1}), (\ref{3.65}), (\ref{3.67}) and (\ref{3.68}) implies
\[
\frac{d}{dt}\mathcal{E}_{N,l}(t)+\mathcal{D}_{N,l}(t)\lesssim\mathcal{E}_{N}(t)\mathcal{E}^1_{N_0,l_0}(t)
+\sum_{|\alpha|= N}\left\|\partial^\alpha E\right\|\left\|\mu^\delta\partial^\alpha f\right\|.
\]
Thus we have completed the proof of Lemma \ref{lemma3.8}.
\end{proof}
Having obtained the above estimates, we now turn to close the a priori estimate (\ref{E-priori}). Our first result in this direction is to deduce an estimate on
$\mathcal{\bar{E}}_{N_0,l_0+l^*}(t)$ which is the main content of the following lemma
\begin{lemma}\label{lemma3.10} Under the assumptions listed in Lemma \ref{Lemma4.5}, for $\ell$ satisfying $\ell= l_0+l^*$ with $l^*=l'+\frac{N_0-1}{2}$, we can deduce that there exist an energy functional $\bar{\mathcal{E}}_{N_0,\ell}(t)$ and the corresponding energy dissipation functional $\bar{\mathcal{D}}_{N_0,\ell}(t)$ which satisfy (\ref{E}), (\ref{D}) such that
\begin{equation}\label{lemma3.10-1}
\frac{d}{dt}\mathcal{\bar{E}}_{N_0,\ell}(t)+\mathcal{\bar{D}}_{N_0,\ell}(t)\lesssim\sum_{|\alpha|=N_0}\varepsilon\left\|\partial^\alpha E\right\|^2.
\end{equation}
Furthermore, we can get that
\begin{equation}\label{lemma3.10-2}
\sup_{0\leq \tau\leq t}\mathcal{\bar{E}}_{N_0,l_0+l^*}(\tau)\lesssim Y_0^2+\varepsilon \int_0^t\sum_{|\alpha|=N_0}\left\|\partial^\alpha E\right\|^2d\tau.
\end{equation}
\begin{proof}
Due to the fact that the weighted estimate on the term evolving the $N_0-$order derivative of the term $E\cdot v\mu^{1/2}$ with respect to $x$ can be dominated by
\[
\sum_{|\alpha|=N_0}\left(\partial^{\alpha}E\cdot v\mu^{1/2}, w_{\ell}\partial^\alpha f\right)\lesssim\sum_{|\alpha|=N_0}\left(\varepsilon\left\|\partial^\alpha E\right\|^2+C_\varepsilon\left\|\mu^\delta\partial^\alpha f\right\|^2\right),
\]
and noticing that the derivatives of the electromagnetic field of order up to $N_0$ decay in time, we can get by repeating the argument used to deduce the energy inequality (\ref{lemma3.8-1}) for $\mathcal{E}_{N,l}(t)$ that
\begin{equation}\label{E-ell-1}
\frac{d}{dt}\mathcal{E}_{N_0,\ell}(t)+\mathcal{D}_{N_0,\ell}(t)\lesssim\varepsilon\sum_{|\alpha|=N_0}\|\partial^\alpha E\|^2.
\end{equation}
A proper linear combination of (\ref{Lemma4.1-1}), (\ref{Lemma4.2-1}), (\ref{EB-s}), and (\ref{E-ell-1}) gives (\ref{lemma3.10-1}), therefore (\ref{lemma3.10-2}) follows by the time integration of (\ref{lemma3.10-1}) and using the time decay of $\|\nabla^{N_0} E\|^2$ with $N_0\geq 4$ for $s\in[\frac 12,1]$ and $N_0\geq 3$ for $s\in(1,\frac 32)$ (recall that $k=0,1,\cdots,N_0-2$). This completes the proof of Lemma \ref{lemma3.10}.
\end{proof}
\end{lemma}
To deduce estimates on $\mathcal{E}_N(t)$ and $\mathcal{E}_{N,l}(t)$ based on the a priori estimate (\ref{E-priori}), we need to deduce the time decay of $\mathcal{E}^1_{N_0,l_0}(t)$ which is presented in the following lemma
\begin{lemma}\label{lemma3.11}
Under the assumptions listed in Lemma \ref{Lemma4.5}, for any $\ell$ satisfying $l_0\leq\ell\leq l_0+(N_0-1)/2$, we can deduce that there exist an energy functional $\mathcal{E}^k_{N_0,\ell}(t)$ and the corresponding energy dissipation functional $\mathcal{D}^k_{N_0,\ell}(t)$ which satisfy (\ref{E_l-k}), (\ref{D_l-k}) such that
\begin{equation}\label{lemma3.11-1}
\frac{d}{dt}\mathcal{E}^k_{N_0,\ell}(t)+\mathcal{D}^k_{N_0,\ell}(t)\lesssim \sum_{|\alpha|=N_0}\left\|\partial^\alpha E\right\|^2
\end{equation}
holds for any $0\leq t\leq T$, where $k=0,1,\cdots,N_0-2$.
\end{lemma}
\begin{proof}
Let $l'$ be given in Lemma \ref{lemma4.3}, we first deduce certain weighted energy type estimates on the derivatives of $f(t,x,v)$ with respect to the $x$ variable only.
For this purpose, take $k+1\leq|\alpha|\leq N_0-1$, by applying $\partial^\alpha$ to $(\ref{f})$ and taking inner product of the resulting identity with $w_{\ell}^2\partial^\alpha f$ over $\mathbb{R}^3_x\times\mathbb{R}^3_v$, we have by using the argument used in deducing Lemma \ref{lemma4.3} that
\begin{equation}\label{l_0-alpha}
\begin{aligned}
&\frac{d}{dt}\left\|w_{\ell}\partial^\alpha f\right\|^2+\left\|w_{\ell}\partial^\alpha f\right\|^2_\sigma
+\frac{1}{(1+t)^{1+\vartheta}}\left\|\langle v\rangle w_{\ell}\partial^\alpha f\right\|^2\\
\lesssim&\left\|\nabla^{|\alpha|}E\right\|^2
+\mathcal{\bar{E}}_{N_0,\ell+l'}(t)\left(\left\|w_{\ell}\nabla^{|\alpha|}f\right\|^2_\sigma+\left\|\nabla^{|\alpha|}(E,B)\right\|^2\right),
\end{aligned}
\end{equation}
while for the case of $|\alpha|= N_0$, one has
\begin{equation}\label{l-N_0}
\begin{aligned}
&\frac{d}{dt}\left\|w_{\ell}\partial^\alpha f\right\|^2+\left\|w_{\ell}\partial^\alpha f\right\|^2_\sigma
+\frac{1}{(1+t)^{1+\vartheta}}\left\|\langle v\rangle w_{\ell}\partial^\alpha f\right\|^2\\
\lesssim&\left\|\nabla^{|\alpha|} E\right\|^2
+\max\left\{\mathcal{E}_{N}(t),\mathcal{\bar{E}}_{N_0,\ell+l'}(t)\right\}
\left(\left\|\nabla^{N_0-1}(E,B)\right\|^2+\left\|\nabla^{N_0-1}f\right\|_\sigma^2+\left\|\nabla^{N_0}f\right\|_\sigma^2\right).
\end{aligned}
\end{equation}
As for the case of $k=1,\cdots, N_0-2$, we have
\begin{equation}\label{k-I-P}
\begin{aligned}
&\frac{d}{dt}\left\|w_{\ell}\nabla^k\{{\bf I-P}\}f\right\|^2+\left\|w_{\ell}\nabla^k\{{\bf I-P}\}f\right\|^2_\sigma
+\frac{1}{(1+t)^{1+\vartheta}}\left\|\langle v\rangle w_{\ell}\nabla^k\{{\bf I-P}\}f\right\|^2\\
\lesssim&\left\|\nabla^k\{{\bf I-P}\}f\right\|_\sigma^2+\left\|\nabla^{k+1}f\right\|_{\sigma}^2
+\left\|\nabla^{k+1}(E,B)\right\|^2+\left\|\nabla^{k}E\right\|^2\\
&+\mathcal{\bar{E}}_{N_0-1,\ell+l'}(t)\left(\left\|w_{\ell}\nabla^k\{{\bf I-P}\}f\right\|^2_\sigma+\left\|\nabla^{k+1}(E,B)\right\|^2\right).
\end{aligned}
\end{equation}

As to the weighted energy type estimates on the mixed derivatives of $f(t,x,v)$ with respect to $x$ and $v$ variables, as pointed out in Lemma \ref{lemma3.8}, we only need to
deduce an estimate on $\left\|w_{l-|\beta|}\partial^\alpha_\beta{\{\bf I- P\}}f\right\|$. To do so, applying $\partial^\alpha_\beta$ with $|\alpha|+|\beta|\leq N_0$ and $k\leq|\alpha|\leq N_0-1$ to $(\ref{I-P})$  and then taking the $L^2$ inner product of the resulting identity with $w^2_{{\ell}-|\beta|}\partial^\alpha_\beta {\bf \{I-P\}}f$, we obtain
\begin{equation}\label{alpha-beta (I-P)f}
\begin{aligned}
&\frac{d}{dt}\left\|w_{\ell-|\beta|}\partial^\alpha_\beta\{{\bf I-P}\}f\right\|^2
+\left\|w_{\ell-|\beta|}\partial^\alpha_\beta\{{\bf I-P}\}f\right\|_\sigma^2
+\frac{1}{(1+t)^{1+\vartheta}}\left\|\langle v\rangle w_{\ell-|\beta|}\partial^\alpha_\beta\{{\bf I-P}\}f\right\|^2\\
\lesssim
&\sum_{0\leq\beta_1\leq\beta}\mathcal{E}_{N_0-1,\ell+l'}(t)\left(\left\|w_{\ell-|\beta_1|}\nabla^{|\alpha|}_x\partial_{\beta_1}\{{\bf I-P}\}f\right\|^2_\sigma
+\left\|\nabla_x^{|\alpha|+1} f\right\|_\sigma^2+\sum_{m=\min\{N_0-1,|\alpha|+1\}}\left\|\nabla_x^{m}(E,B)\right\|^2\right)\\
&+\sum_{|\beta'|<|\beta|}\left\|w_{\ell-|\beta|}\partial^\alpha_{\beta'}\{{\bf I-P}\}f\right\|_\sigma^2
+\left\|w_{\ell-|\beta|+1}\partial^{\alpha+e_i}_{\beta-e_i}\{{\bf I-P}\}f\right\|^2_\sigma+\left\|\nabla_x^{|\alpha|+1} f\right\|_\sigma^2\\
&+\sum_{m=\min\{N_0-1,|\alpha|+1\}}\left\|\nabla_x^{m}(E,B)\right\|^2+\left\|\nabla^k E\right\|^2
+\sum_{k\leq |\alpha|\leq N_0}\left\|\partial^\alpha\{{\bf I-P}\}f\right\|^2_\sigma.
\end{aligned}
\end{equation}
Therefore, a proper linear combination of (\ref{Lemma4.5-1}) (\ref{alpha-f}), (\ref{k-I-P}) and (\ref{alpha-beta (I-P)f}) implies
\begin{equation}
\frac{d}{dt}\mathcal{E}^k_{N_0,\ell}(t)+\mathcal{D}^k_{N_0,\ell}(t)\lesssim \sum_{|\alpha|=N_0}\left\|\partial^\alpha E\right\|^2.
\end{equation}
This completes the proof of the Lemma \ref{lemma3.11}.
\end{proof}
Based on the above lemma, we can obtain the time decay of $\mathcal{E}^k_{N_0,\ell}(t)$.
\begin{lemma}\label{lemma-E-K}
Let $0\leq i\leq k\leq N_0-3$ be an integer,
it holds that
\begin{equation}\label{lemma-E-K-1}
 \begin{split}
\mathcal{E}^k_{N_0,l_0+i/2}(t)
\lesssim X(t)(1+t)^{-k-s+i}, \quad \quad i=0,1,\cdots,k.
\end{split}
\end{equation}
Furthermore, when $s\in[1,3/2)$, we have for $k=0,1,\cdots,N_0-3$ that
\begin{equation}\label{lemma-E-K-2}
\mathcal{E}^k_{N_0,l_0+k/2+1/2}(t)
\lesssim X(t)(1+t)^{1-s}.
\end{equation}
\end{lemma}
\begin{proof}
From Lemma \ref{lemma3.11},
multiplying (\ref{lemma3.11-1}) with $\ell=l_0+i/2$ by $(1+t)^{k+s-i+\epsilon}$ gives
\begin{equation}\label{E-i}
 \begin{split}
&\frac{d}{dt}\left\{(1+t)^{k+s-i+\epsilon}\mathcal{E}^k_{N_0,l_0+i/2}(t)\right\}
+(1+t)^{k+s-i+\epsilon}\mathcal{D}^k_{N_0,l_0+i/2}(t)\\
\lesssim &\sum_{|\alpha|=N_0}(1+t)^{k+s-i+\epsilon}\left\|\partial^\alpha E\right\|^2+(1+t)^{k+s-i-1+\epsilon}\mathcal{E}^k_{N_0,l_0+i/2}(t).
\end{split}
\end{equation}
Here $\epsilon$ is taken as a sufficiently small positive constant.

When $s\in[1/2,1)$, we take $\ell=l_0+k/2+1/2$ in Lemma \ref{lemma3.11}, it holds that
\begin{equation}\label{E_1}
\frac{d}{dt}\mathcal{E}^k_{N_0,l_0+k/2+1/2}(t)+\mathcal{D}^k_{N_0,l_0+k/2+1/2}(t)\lesssim \sum_{|\alpha|=N_0}\left\|\partial^\alpha E\right\|^2.
\end{equation}
Using the difference between the energy functional $\mathcal{E}^k_{N_0,\ell}(t)$ and its dissipation rate $\mathcal{D}^k_{N_0,\ell}(t)$, a proper linear combination of (\ref{E-i}) and (\ref{E_1}) yields
\begin{equation}\label{E-sum}
 \begin{split}
&\frac{d}{dt}\sum_{i=0}^k\left(C_i(1+t)^{k+s-i+\epsilon}\mathcal{E}^k_{N_0,l_0+i/2}(t)+C_{k+1}\mathcal{E}^k_{N_0,l_0+k/2+1/2}(t)\right)\\
&+\left(\sum_{i=0}^k(1+t)^{k+s-i+\epsilon}\mathcal{D}^k_{N_0,l_0+i/2}(t)+\mathcal{D}^k_{N_0,l_0+k/2+1/2}(t)\right)\\
&\lesssim \sum_{|\alpha|=N_0}(1+t)^{k+s+\epsilon}\left\|\partial^\alpha E\right\|^2+(1+t)^{k+s-1+\epsilon}\left\|\nabla^k({\bf P}f,E,B)\right\|^2.
\end{split}
\end{equation}
On the other hand, Lemma \ref{Lemma4.5} tells us that
\begin{equation}\label{E-decay}
 \begin{split}
\sum_{|\alpha|=N_0}(1+t)^{k+s+\epsilon}\left\|\partial^\alpha E\right\|^2+(1+t)^{k+s-1+\epsilon}\left\|\nabla^k({\bf P}f,E,B)\right\|^2\lesssim X(t)(1+t)^{-1+\epsilon}.
\end{split}
\end{equation}
Plugging (\ref{E-decay}) into (\ref{E-sum}) and taking the time integration, it follows that
\begin{equation}\label{e-k-1}
 \begin{split}
&\sum_{i=0}^k(1+t)^{k+s-i+\epsilon}\mathcal{E}^k_{N_0,l_0+i/2}(t)+\mathcal{E}^k_{N_0,l_0+k/2+1/2}(t)\\
&+\int_0^t\left(\sum_{i=0}^k(1+\tau)^{k+s-i+\epsilon}\mathcal{D}^k_{N_0,l_0+i/2}(\tau)+\mathcal{D}^k_{N_0,l_0+k/2+1/2}(\tau)\right)d\tau\\
\lesssim& X(t)(1+t)^{\epsilon},
\end{split}
\end{equation}
which implies that
\begin{equation}
 \begin{split}
\mathcal{E}^k_{N_0,l_0+i/2}(t)
\lesssim X(t)(1+t)^{-k-s+i}, \quad \quad i=0,1,\cdots,k.
\end{split}
\end{equation}
This proves (\ref{lemma-E-K-1}) for the case of $s\in[\frac 12,1)$.

 When $s\in [1,3/2)$, multiplying (\ref{lemma3.11-1}) with $\ell=l_0+k/2+1/2$ by $(1+t)^{s-1+\epsilon}$ gives
\begin{equation}\label{E_s}
 \begin{split}
&\frac{d}{dt}\left((1+t)^{s-1+\epsilon}\mathcal{E}^k_{N_0,l_0+k/2+1/2}(t)\right)+(1+t)^{s-1+\epsilon}\mathcal{D}^k_{N_0,l_0+k/2+1/2}(t)\\
\lesssim& \sum_{|\alpha|=N_0}(1+t)^{s-1+\epsilon}\left\|\partial^\alpha E\right\|^2 +(1+t)^{s-2+\epsilon}\mathcal{E}^k_{N_0,l_0+k/2+1/2}(t).
\end{split}
\end{equation}
On the other hand, we have by taking $\ell=l_0+k/2+1$ in (\ref{lemma3.11-1}) that
\begin{equation}\label{E_2}
\frac{d}{dt}\mathcal{E}^k_{N_0,l_0+k/2+1}(t)+\mathcal{D}^k_{N_0,l_0+k/2+1}(t)\lesssim \sum_{|\alpha|=N_0}\left\|\partial^\alpha E\right\|^2.
\end{equation}
Similar to that of the case of $s\in[1/2,1)$, a proper linear combination of (\ref{E-i}), (\ref{E-s}), (\ref{E_s}), and
(\ref{E_2}) yields
\begin{equation}\label{E-sum-1}
 \begin{split}
&\frac{d}{dt}\sum_{i=0}^{k+1}\left(C_i(1+t)^{k+s-i+\epsilon}\mathcal{E}^k_{N_0,l_0+i/2}(t)
+C_{k+2}\mathcal{E}^k_{N_0,l_0+k/2+1/2}(t)\right)\\
&+\left(\sum_{i=0}^{k+1}(1+t)^{k+s-i+\epsilon}\mathcal{D}^k_{N_0,l_0+i/2}(t)+\mathcal{D}^k_{N_0,l_0+k/2+1/2}(t)\right)\\
\lesssim& \sum_{|\alpha|=N_0}(1+t)^{k+s+\epsilon}\left\|\partial^\alpha E\right\|^2
+(1+t)^{k+s-1+\epsilon}\left\|\nabla^k({\bf P}f,E,B)\right\|^2.
\end{split}
\end{equation}
By the same way used to deduce (\ref{e-k-1}), it follows that
\begin{equation}\label{e-k-2}
 \begin{split}
&\sum_{i=0}^{k+1}(1+t)^{k+s-i+\epsilon}\mathcal{E}^k_{N_0,l_0+i/2}(t)+\mathcal{E}^k_{N_0,l_0+k/2+1}(t)\\
&+\int_0^t\left(\sum_{i=0}^{k+1}(1+\tau)^{k+s-i+\epsilon}\mathcal{D}^k_{N_0,l_0+i/2}(\tau)
+\mathcal{D}^k_{N_0,l_0+k/2+1}(\tau)\right)d\tau \\
\lesssim& X(t)(1+t)^{\epsilon}.
\end{split}
\end{equation}
Combining (\ref{e-k-1}) and (\ref{e-k-2}) yields (\ref{lemma-E-K-1}) and (\ref{lemma-E-K-2}). This completes the proof of Lemma \ref{lemma-E-K}.
\end{proof}
Now we are ready to obtain the closed estimate on $\mathcal{E}_N(t)$ and $\mathcal{E}_{N,l}(t)$ of the time-weighted energy norm $X(t)$ in the following lemma
\begin{lemma}\label{lemma3.12}
It holds that
\begin{equation}
\sup_{0\leq s\leq t}\left\{\mathcal{E}_{N}(s)+(1+s)^{-\frac{1+\epsilon_0}{2}}\mathcal{E}_{N,l}(t) \right\}+\int_{0}^t\mathcal{D}_N(s)ds\lesssim Y_0^2.
\end{equation}
\end{lemma}
\begin{proof}
Multiplying (\ref{lemma3.7-1}) by $(1+t)^{-\epsilon_0}$, it holds that
\begin{equation}\label{e-n-1}
\begin{aligned}
&\frac{d}{dt}\left((1+t)^{-\epsilon_0}\mathcal{E}_{N}(t)\right)
+\epsilon_0(1+t)^{-1-\epsilon_0}\mathcal{E}_{N}(t)+(1+t)^{-\epsilon_0}\mathcal{D}_{N}(t)\\
\lesssim&\frac{\delta}{(1+t)^{1+\epsilon_0+\vartheta}}\mathcal{D}_{N,l}(t)
+(1+t)^{-\epsilon_0}\mathcal{E}_{N}(t)\mathcal{E}^1_{N_0,l_0}(t).
\end{aligned}
\end{equation}
Moreover, we can get by multiplying (\ref{lemma3.8-1}) by $(1+t)^{-(1+\epsilon_0)/2}$ that
\begin{equation}\label{e-n-2}
\begin{aligned}
&\frac{d}{dt}\left((1+t)^{-\frac{\epsilon_0+1}2}\mathcal{E}_{N,l}(t)\right)
+\frac{\epsilon_0+1}2(1+t)^{-\frac{\epsilon_0+3}2}\mathcal{E}_{N,l}(t)
+(1+t)^{-\frac{\epsilon_0+1}2}\mathcal{D}_{N,l}(t)\\
\lesssim&(1+t)^{-\frac{\epsilon_0+1}2}\mathcal{E}_{N}(t)\mathcal{E}^1_{N_0,l_0}(t)
+(1+t)^{-\frac{\epsilon_0+1}2}\left\|\nabla^NE\right\|\left\|\mu^\delta\nabla^N f\right\|\\
\lesssim&(1+t)^{-\frac{\epsilon_0+1}2}\mathcal{E}_{N}(t)\mathcal{E}^1_{N_0,l_0}(t)
+(1+t)^{-\epsilon_0-1}\left\|\nabla^NE\right\|^2+\varepsilon\left\|\mu^\delta\nabla^N f\right\|^2.
\end{aligned}
\end{equation}
A proper linear combination of (\ref{lemma3.7-1}), (\ref{e-n-1}) and (\ref{e-n-2}) and
taking the time integration yield
\begin{equation}
\begin{aligned}
&\mathcal{E}_N(t)+(1+t)^{-\frac{\epsilon_0+1}2}\mathcal{E}_{N,l}(t)
+\int^t_0\left((1+s)^{-1-\epsilon_0}\mathcal{E}_N(s)+\mathcal{D}_N(s)\right)ds\\
&+\int^t_0\left((1+s)^{-\frac{\epsilon_0+3}2}\mathcal{E}_{N,l}(s)+(1+s)^{-\frac{\epsilon_0+1}2}
\mathcal{D}_{N,l}(s)\right)ds
\lesssim Y_0^2.
\end{aligned}
\end{equation}
Here when $s\in[\frac12,1]$ we use the fact $\mathcal{E}^1_{N_0,l_0}(t)\leq (1+t)^{-1-s}X(t)$ obtained in Lemma \ref{lemma-E-K}. This completes the proof of Lemma \ref{lemma3.12}.
\end{proof}
Recalling the definition of the $X(t)-$norm, we have from the estimates obtained in Lemma \ref{lemma3.10} and Lemma \ref{lemma3.12} that
\begin{equation}\label{bound-on-X(t)}
X(t)\lesssim Y_0^2.
\end{equation}

Having obtained (\ref{bound-on-X(t)}), the global solvability result follows immediately from the well-established local existence result and by employing the continuation argument in the usual way. This completes the proof of Theorem \ref{Th1.1}.

Based on Theorem 1.1, for $k=0,1,2,\cdots, N_0-2$, we have from  Lemma \ref{Lemma4.5} that
\begin{equation}
\mathcal{E}^k_{N_0}(t)\lesssim Y_0^2(1+t)^{-(k+s)},
\end{equation}
this gives (\ref{TH2-1}).

On the other hand, let $0\leq i\leq k\leq N_0-3$ be an integer, we have from Lemma \ref{lemma-E-K} that
\begin{equation}
 \begin{split}
\mathcal{E}^k_{N_0,l_0+i/2}(t)
\lesssim Y_0^2(1+t)^{-k-s+i}, \quad \quad i=0,1,\cdots,k
\end{split}
\end{equation}
and
\begin{equation}
\mathcal{E}^k_{N_0,l_0+k/2+1/2}(t)
\lesssim Y_0^2(1+t)^{1-s}
\end{equation}
when $s\in[1,3/2)$. Consequently, (\ref{TH2-2})and (\ref{TH2-3}) follow from the above two inequalities.

Next to prove (\ref{TH2-4}), when $N_0+1\leq|\alpha|\leq N-1$, we have by exploiting the interpolation method and by using the time decay of $\left\|\nabla^{N_0}f\right\|$ and the uniform bound of $\left\|\nabla^{N}f\right\|$ that
\begin{equation}
\begin{split}
\left\|\partial^\alpha f\right\|^2\lesssim \left\|\nabla^Nf\right\|^{\frac{|\alpha|-N_0}{N-N_0}}\left\|\nabla^{N_0}f\right\|^{\frac{N-|\alpha|}{N-N_0}}
\lesssim Y_0^2(1+t)^{-\frac{(N-|\alpha|)(N_0-2+s)}{N-N_0}}.
\end{split}
\end{equation}

Finally for (\ref{TH2-5}), by the similar way as above, we can get that
\begin{equation}
\begin{split}
\left\|w_l\partial^\alpha f\right\|^2\lesssim \left\|w_l\nabla^Nf\right\|^{\frac{|\alpha|-N_0}{N-N_0}}\left\|w_l\nabla^{N_0}f\right\|^{\frac{N-|\alpha|}{N-N_0}}
\lesssim Y_0^2(1+t)^{-\frac{(N-|\alpha|)(N_0-3+s)-(1+\epsilon_0)(|\alpha|-N_0)/2}{N-N_0}}.
\end{split}
\end{equation}
Thus the proof of Theorem \ref{Th1.2} is complete.
\bigbreak
\begin{center}
{\bf Acknowledgment}
\end{center}
This work was supported by two grants from the National Natural Science Foundation of China under contracts 10925103 and 11261160485 respectively and ``the Fundamental Research Funds for the Central Universities".

\end{document}